\documentclass[11pt, reqno]{amsart}
\usepackage{mathrsfs}
\usepackage[active]{srcltx}
\usepackage{mathrsfs}
\usepackage{array}
\usepackage{amssymb}
\usepackage{multirow}
\usepackage{color}
\usepackage{float}
\usepackage{caption}
\usepackage{yfonts}
\usepackage{diagbox}
\usepackage{hhline}
\usepackage{mathtools}
\usepackage[table, usenames, dvipsnames, svgnames]{xcolor}
\usepackage{tabularx}
\DeclareMathAlphabet{\mathpzc}{OT1}{pzc}{m}{rm}

\usepackage[font=footnotesize,labelfont=bf]{caption}

\newenvironment{dedication}
{\vspace{6ex}\begin{quotation}\begin{center}\begin{em}}
			{\par\end{em}\end{center}\end{quotation}}

\allowdisplaybreaks

\newcommand{\tagequation}[3]
{
	\begin{equation} 
		\label{#1}
		#2
		\tag{$#3$}
	\end{equation}
}

\usepackage{color}
\usepackage[unicode]{hyperref}
\hypersetup{
	colorlinks = true,%
	citecolor = [rgb]{0.0,0.0,0.9},
	filecolor=black,%
	linkcolor = [rgb]{0.65,0.0,0.0},%
	anchorcolor = red,
	pagecolor = red,
	urlcolor= [rgb]{0.65,0.0,0.0}
}

\numberwithin{equation}{section}

\usepackage{graphicx}

\usepackage[left=2.5cm, right=2.5cm, top=3.0cm, bottom=3.0cm]{geometry}

\newtheorem{proposition}{Proposition}[section]
\newtheorem{theorem}{Theorem}[section]
\newtheorem{lemma}{Lemma}[section]

\newtheorem{remark}{Remark}[section]

\newcommand{\ds}{\displaystyle}

\usepackage{lmodern}
\newcommand*{\myfont}{\fontfamily{lmss}\selectfont}

\newcommand{\LP}{\myfont{}\text{P}\normalfont}
\newcommand{\LQ}{\myfont{}\text{Q}\normalfont}
\newcommand{\LA}{\myfont{}\text{A}\normalfont}

\begin{document}

\title[Clamped plates in positively curved spaces]
%{The Rayleigh problem for clamped plates on low-dimensional non-positively curved spaces}
{%Sharp isoperimetric inequalities for c
%	Lord Rayleigh's conjecture for  
%	clamped domains in  positively curved spaces
%Lord Rayleigh's conjecture \\ for  positively curved vibrating
%clamped plates 
Lord Rayleigh's Conjecture for Vibrating Clamped Plates in Positively Curved Spaces
%Lord Rayghley's conjecture for vibrating plates\\ clamped onto positively curved manifolds
}

\vspace{-0.9cm}
\author{Alexandru Krist\'aly}
\address{Department of Economics, Babe\c s-Bolyai University, Cluj-Napoca, Romania   \& 
Institute of Applied Mathematics, \'Obuda
University, %(formerly Budapest Tech Polytechnical Institution),
Budapest, Hungary}
\email{alex.kristaly@econ.ubbcluj.ro; kristaly.alexandru@uni-obuda.hu}

\thanks{Research supported by the UEFISCDI/CNCS grant PN-III-P4-ID-PCE2020-1001.}


%    General info

	\vspace*{-0.5cm}
	
\begin{abstract}
	We affirmatively solve the analogue of Lord Rayleigh's conjecture  on Riemannian manifolds with positive Ricci curvature for \textit{any}  clamped plates in 2 and 3 dimensions, and for sufficiently \textit{large} clamped plates in dimensions beyond 3. These results complement those from the flat
	(\href{https://dx.doi.org/10.1215/S0012-7094-95-07801-6}{M.\ Ashbaugh \& R.\ Benguria, 1995}; \href{https://doi.org/10.1007/BF00375124}{N.\ Nadirashvili, 1995}) and negatively curved (\href{https://doi.org/10.1016/j.aim.2020.107113}{A.\ Krist\'aly, 2020})
	cases that are valid only in 2 and 3 dimensions, and at the same time also provide the first positive answer to Lord Rayleigh's conjecture in higher dimensions.\
	 The proofs rely on an Ashbaugh--Benguria--Nadirashvili--Talenti  nodal-decomposition argument, on the L\'evy--Gromov isoperimetric inequality, on fine properties of Gaussian hypergeometric functions and on sharp spectral gap estimates of fundamental tones for both small and large clamped spherical caps. Our results show that positive curvature enhances genuine differences between low- and high-dimensional settings,  a tacitly accepted paradigm in the theory of vibrating clamped plates. In the limit case -- when the Ricci curvature is non-negative -- we establish  a Lord Rayleigh-type isoperimetric inequality  that involves the asymptotic volume ratio of the non-compact complete Riemannian manifold; moreover, the inequality is strongly rigid in 2 and 3 dimensions, i.e.,\ if equality holds for a given clamped plate then the manifold is isometric to the Euclidean space.

\end{abstract}

\maketitle
\vspace{-1.7cm}

	\begin{dedication}
	%	\hspace{4cm}
%	\vspace*{9cm}
	{Dedicated to the memory of my professor and friend, Csaba Varga 
		$($1959--2021$)$
	}
\end{dedication}
\tableofcontents

%
%
%
%
%\vspace{-0.3cm}
\section{Introduction}

The paper is devoted to the analogue of Lord Rayleigh's conjecture, concerning the lowest principal frequency of  vibrating clamped plates on positively curved spaces. Our results can be viewed as the concluding piece in the theory of clamped plates after the seminal works of  Ashbaugh and Benguria \cite{A-B} and Nadirashvili \cite{Nad} in Euclidean spaces, and the recent paper by the author \cite{Kristaly-Adv-math} on non-positively curved spaces, all valid in dimensions 2 and 3. To our surprise,  positively curved spaces provide an appropriate geometric setting for the validity of Lord Rayleigh's conjecture  not only in  dimensions 2 and 3  for any clamped plate, but also in dimensions beyond  3 for sufficiently large domains. Before presenting our results in details, we shortly recall some historical milestones related to the subtleties in the theory of vibrating clamped plates.

\subsection{Historical aspects}

The original problem appeared in 1877, when John William Strutt, 3rd Baron Rayleigh \cite{Rayleigh} formulated, inter alia, two isoperimetric inequalities arising from mathematical physics;  he claimed that the disc has the minimal principal frequency among either clamped plates or fixed membranes with a given area.\ Although Lord Rayleigh formulated his conjectures for planar domains, he surely had the feeling that the statement should be valid in any dimension as subsequent literature referred to these conjectures; in particular, both the ``clamped plate" and ``fixed membrane" notions are commonly used in any dimension. 

In the 1920s,  Lord Rayleigh's conjecture for the fixed membrane problem
%	\begin{equation}\label{FM-problem-0}
%	\left\{ \begin{array}{lll}
%		-\Delta u=\lambda u &\mbox{in} &  \Omega, \\
%		% u\geq 0 &\mbox{in} &   \Omega;\\
%		u=0 &\mbox{on} &  \partial \Omega,
%	\end{array}\right.
%\end{equation}
%where $\Omega\subset \mathbb R^n$  is a bounded open domain, $n\geq 2$, 
  has been confirmed independently by Faber \cite{Faber} and Krahn \cite{Krahn},  by showing that the
  principal/first Dirichlet eigenvalue of the Laplace operator for any bounded open domain $\Omega\subset \mathbb R^n$   is not less than the corresponding Dirichlet eigenvalue of a  ball $\Omega^\star\subset \mathbb R^n$ that has the same volume as $\Omega$.
%   principal/first Dirichlet eigenvalues for the Laplace 
%  operator verify $\lambda(\Omega)\geq \lambda(\Omega^\star)$, where $\Omega\subset \mathbb R^n$ $(n\geq 2)$ is any bounded open domain and $\Omega^\star\subset \mathbb R^n$ is a ball with the same volume as $\Omega$. 
Their arguments are based on the classical isoperimetric inequality in $\mathbb R^n$ combined with a Schwarz-type rearrangement. In the 1980s, Lord Rayleigh's conjecture for  fixed membranes has been solved on Riemannian manifolds of positive Ricci curvature, see B\'erard and Meyer \cite{BM}, and on those Cartan--Hadamard manifolds (i.e., complete and simply connected Riemannian manifolds with non-positive sectional curvature) which verify the Cartan--Hadamard conjecture, see Chavel \cite{Chavel}; all these arguments rest upon the sign-definite character of the first eigenfunction for the second-order fixed membrane problems.  

%The clamped plate problem is definitely a harder
%nut to crack than the fixed membrane problem,  as t

Transferring simply the arguments from the fixed membrane problem to clamped plates  can be elusive.
%The arguments transferred from  can be delusive.
To be more precise, the clamped plate problem  can be formulated as 
\begin{equation}\label{CP-euclidean}
\left\{
\begin{array}
	[c]{@{~}r@{~}l@{~}lll@{~}}
	\Delta^{2}u & = & \Lambda_{0}u & \text{in} & \Omega,\\
	u&= & \dfrac{\partial u}{\partial\mathbf{n}}  =  0 & \text{on} &\partial\Omega,
\end{array}
\right.
\end{equation}
where   $\Omega\subset \mathbb R^n$  $(n\geq 2)$ is a bounded open domain,   $\Delta^2$ is the bi-Laplace operator,   $\frac{\partial }{\partial \textbf{n}}$ is the outward normal derivative on $\partial \Omega$, while the principal frequency (or fundamental tone) of the clamped plate $\Omega$ can be characterized variationally as
\begin{equation}\label{variational-charact-0}
	\Lambda_0\coloneqq \Lambda_0(\Omega)=\inf_{u\in W_0^{2,2}(\Omega)\setminus \{0\}}\frac{\displaystyle \int_{\Omega}(\Delta u)^2 {\rm d}x}{\displaystyle \int_{\Omega}u^2 {\rm d}x}.
\end{equation} 
Assuming that the first eigenfunction of \eqref{CP-euclidean} is of fixed sign  for a domain $\Omega\subset \mathbb R^n$ ($n\geq 2$), 
Szeg\H o \cite{Szego} proved in the early 1950s 
%by symmetrization/rearrangement techniques 
the validity of Lord Rayleigh's conjecture, i.e., $\Lambda_0(\Omega)\geq \Lambda_0(\Omega^\star)$,  where $\Omega^\star$ is a ball in $\mathbb R^n$ with the same volume as $\Omega$. Szeg\H o's proof used standard symmetrization/rearrangement techniques and he implicitly expressed his hope that any clamped domain should produce  principal eigenfunctions of fixed sign.  However, his hope has been shattered soon, as Duffin \cite{Duffin} (see also Coffman, Duffin and Shaffer \cite{CDS}) constructed clamped plates with  sign-changing first eigenfunctions. Accordingly,   Szeg\H o's initial argument for proving Lord Rayleigh's conjecture for clamped plates failed,  becoming a hard
nut to crack though several decades;
  in fact, the main obstructions to follow the arguments from the fixed membrane problem are formed by both
  the lack of a maximum principle for the fourth-order clamped plate problem and the failure of a suitable rearrangement of the first eigenfuntion $u_1$ in \eqref{CP-euclidean} with a suitable estimate of $\displaystyle\int_\Omega (\Delta u_1)^2{\rm d}x$.  These phenomena are deeply analyzed in the monograph of  Gazzola,  Grunau and Sweers \cite{Gazzola-etc}.

A breakthrough idea has been arising from Talenti \cite{Talenti} in the early 1980s, who decomposed the domain $\Omega\subset \mathbb R^n$  corresponding to the positive and negative parts of the first eigenfunction $u_1$ in \eqref{CP-euclidean}. Using a Schwarz-type rearrangement of these domains/functions, he was able to control  in a suitable manner the quantities $\displaystyle\int_\Omega (\Delta u_1)^2{\rm d}x$ and $\displaystyle\int_\Omega  u_1^2{\rm d}x$ in  \eqref{variational-charact-0},  obtaining a two-ball minimization problem by means of which  he provided the non-sharp estimate $\Lambda_0(\Omega)\geq t_n \Lambda_0(\Omega^\star)$ with $t_n\in \left[\frac{1}{2},1\right)$ for every $n\geq 2$ and $\lim_{n\to \infty} t_n=\frac{1}{2}$.

 %After 115 years since its  initial formulation, 
 
 More than 115 years had to pass before Nadirashvili \cite{Nad-0,Nad}  announced the solution to the original (i.e., $2$-dimensional) Lord Rayleigh's conjecture, by slightly  modifying Talenti's  argument. Inspired by Nadirashvili's achievement,  Ashbaugh and  Benguria \cite{A-B}  proved Lord Rayleigh's conjecture in dimensions 2 and 3, by using sharp estimates in  Talenti's decomposition  combined with fine properties of Bessel functions. We note that the conjecture is still open in higher dimensions; however, almost simultaneously with \cite{A-B}, Ashbaugh and Laugesen \cite{A-L} provided an asymptotically sharp estimate, i.e., $\Lambda_0(\Omega)\geq {w}_n \Lambda_0(\Omega^\star)$ with $w_n\in [0.89,1)$ for every $n\geq 4$ and $\lim_{n\to \infty} w_n=1$. Recently, Chasman and Langford \cite{CL2} proved a non-sharp isoperimetric inequality for clamped plates on Gaussian spaces, stating that $\Gamma_w(\Omega)\geq  c\Gamma_w(\Omega^\star)$ for some $ c=c(\Omega,n)\in (0,1)$, where 
 %depends  on  the  volume of $\Omega$ and  dimension $n\geq 2$,
   $\Gamma_w(\Omega)$ and $\Gamma_w(\Omega^\star)$ are the fundamental tones of clamped plates with respect to the Gaussian density $w$.  
 
Since the fixed membrane problem of curved spaces is  fully described, see \cite{BM, Chavel},  a similar question naturally arises also for clamped plates. Fixing  a complete $n$-dimensional $(n\geq 2)$ Riemannian manifold $(M,g)$, 	 we consider instead of (\ref{variational-charact-0}) the fundamental tone 
 \begin{equation}\label{variational-charact}
 	\Lambda_g(\Omega)\coloneqq \inf_{u\in W_0^{2,2}(\Omega)\setminus \{0\}}\frac{\displaystyle \int_{\Omega}(\Delta_g u)^2 {\rm d}v_g}{\displaystyle \int_{\Omega}u^2 {\rm d}v_g},
 \end{equation} 
of $\Omega\subset M$,    where  $\Delta_g$ and ${\rm d}v_g$ are the  Laplace--Beltrami operator and canonical measure on $(M,g)$.
  
   When $(M,g)$ is a Cartan--Hadamard manifold with sectional curvature bounded from above by $ -\kappa\leq 0$,
% We study Lord Rayleigh's problem for clamped plates on an arbitrary $n$-dimen\-sional $(n\geq 2)$ We first prove a McKean-type spectral gap estimate, i.e.  the fundamental tone of any domain in $(M,g)$ is universally bounded  from below by $\frac{(n-1)^4}{16}\kappa^4$ whenever the $\kappa$-Cartan-Hada\-mard conjecture holds on $(M,g)$, e.g. in 2-\ and 3-dimensions due to Bol (1941) and Kleiner (1992), respectively. 
the author \cite{Kristaly-Adv-math} proved  the validity of Lord Rayleigh's conjecture in dimensions 2 and 3 for small clamped plates, i.e., $\Lambda_g(\Omega)\geq \Lambda_\kappa(\Omega^\star)$ holds for every domain $\Omega\subset M$ having volume $V_g(\Omega)=V_\kappa(\Omega^\star)\leq c_n/\kappa^{n/2}$ with $c_2\approx 21.031$ and  $c_3\approx 1.721$, respectively;  here $V_g(\Omega)$  and $V_\kappa(\Omega^\star)$ denote the volumes of $\Omega$ in $(M,g)$ and the geodesic ball $\Omega^\star$ in the space form $(N_\kappa^n,g_\kappa)$ of constant curvature $-\kappa$, respectively, while  $\Lambda_\kappa(\Omega^\star)$ stands for the fundamental tone of $\Omega^\star$ in $(N_\kappa^n,g_\kappa)$ corresponding to \eqref{variational-charact}. 
% In  2- and 3-dimensions we prove  sharp isoperimetric inequalities for sufficiently small clamped plates, i.e.  the fundamental tone of any domain in $(M,g)$ of volume $v>0$ is not less than the corresponding fundamental tone of a geodesic ball of the same volume $v$ in the space of constant curvature $-\kappa^2$ provided that  $v\leq c_n/\kappa^n$ with $c_2\approx 21.031$ and  $c_3\approx 1.721$, respectively. 
 In particular,  the above result provides in the limit  case $\kappa\to 0$ the main result of Ashbaugh and  Benguria \cite{A-B}. The proofs in \cite{Kristaly-Adv-math} are based on the generalized Cartan-Hada\-mard conjecture (see e.g.\ Kloeckner and Kuperberg \cite{KK}) and peculiar properties of the Gaussian hypergeometric function, both valid only in dimensions 2 and 3. Some non-sharp estimates of $\Lambda_g(\Omega)$ for clamped plates $\Omega\subset M$ are also provided in dimensions beyond 3 in the same geometrical setting. 
 % whenever the sectional curvature verifies $\textbf{K}\leq 0.$

Since a systematic study concerning the fundamental tone of clamped plates on positively curved Riemannian manifolds is unavailable, the objective of the present paper is to fill this gap by solving the analogue of Lord Rayleigh's conjecture (and proving further related results) in this geometric setting; it turns out that unexpected phenomena occur with respect to the non-positively curved framework that are  presented in the next subsection.

	\subsection{Main results}
	Let $(M,g)$ be a compact $n$-dimensional $(n\geq 2)$ Riemannian manifold $(M,g)$ with Ricci curvature ${\sf Ric}_{(M,g)}\geq (n-1)\kappa>0$, and consider the clamped plate problem   
	\begin{equation}\label{CP-problem-0}
	\left\{
	\begin{array}
		[c]{@{~}r@{~}l@{~}lll@{~}}
		\Delta_g^{2}u & = & \Lambda u & \text{in} & \Omega,\\
		u&= & \dfrac{\partial u}{\partial\mathbf{n}}  =  0 & \text{on} &\partial\Omega,
	\end{array}
	\right.
	\end{equation}
where $\Omega\subset M$ is a bounded open domain,      $\Delta_g^2$ stands for the biharmonic Laplace--Beltrami operator on $(M,g)$ and $\frac{\partial }{\partial \textbf{n}}$ is the outward normal   derivative on $\partial \Omega$. The fundamental tone $\Lambda_g(\Omega)$ of the set $\Omega\subset M$ associated with \eqref{CP-problem-0} is variationally expressed  by \eqref{variational-charact}. As a model space, let $\mathbb S_\kappa^n\subset \mathbb R^{n+1}$ be the $n$-dimensional sphere with radius $1/\sqrt{\kappa}$ (i.e., with constant curvature $\kappa>0$), endowed with its natural Riemannian metric $g_\kappa$; for simplicity of notation, we use $\Lambda_{\kappa}(\Omega)$ and $V_\kappa(\Omega)$ instead of the fundamental tone  $\Lambda_{g_\kappa}(\Omega)$ and  volume $V_{g_\kappa}(\Omega)$, respectively, of the open domain $\Omega \subset \mathbb S_\kappa^n$. Moreover, $d_\kappa(N,x)$ is the geodesic distance on $\mathbb S_\kappa^n$ between the North pole $N=(0,\ldots,0,1/\sqrt{\kappa})$ and $x\in \mathbb S_\kappa^n.$

A deeper understanding of the feature of fundamental tones for generic clamped plates on  $(M,g)$ motivates the investigation of some particular cases.	Our first result provides a quite complete picture on the behavior of the first eigenfunctions on 2-dimensional spherical belts depending on their relative width. Given $\kappa>0$, let $$B_\kappa(r,R)=\{x\in \mathbb S_\kappa^2:r<d_\kappa(N,x)<R\}$$ be the \textit{spherical belt} with radii $0<r<R<\pi/\sqrt{\kappa}.$ 
	 We also recall the critical Coffman--Duffin--Schaffer  constant ${\sf C}_{CDS}\approx{762.3264}$ (see \cite{CDS}), whose origin will be explained in the proof of the following result.

%	To state our first result, we need the CDS constant; since it will explicitly appear  in our proof, we postpone its  precise definition. 

%\textbf{	BOGGIO-HADAMARD!!!! Grunau konyv es Ghergu-Radulescu konyve}
%	
%	
%	, where $\kappa$ is the standard metric on the sphere $\mathbb S^n_\kappa.$

%\textbf{	GREEN FUNCTIONS ON THE SPHERE}: nem lehetne megadni effektiv a megoldast, vagy pedig csak igazolni azt, hogy sign-preserving! 

%IDE talan a GRUNAU konyv lenne jo.

%\textbf{ waviness hullamzas}
	
	\begin{theorem}\label{theorem-belt} %Let $\kappa>0$. 
		{$($\rm\textbf{Spherical belts}$)$}
		Let $R>r>0$. Then there exists $\kappa_0\in (0,\pi^2/R^2)$ with the following properties.  
		\begin{itemize}
			\item[{\rm (i)}] {\rm Narrow relative width:} if $R/r<{\sf C}_{CDS}$, then for every $\kappa\in (0,\kappa_0)$, the first eigenfunction of problem \eqref{CP-problem-0} on the spherical belt $\Omega\coloneqq B_\kappa(r,R)\subset \mathbb S_\kappa^2$  is of fixed sign$.$ 
			
			\item[{\rm (ii)}] {\rm Wide relative width:} if $R/r>{\sf C}_{CDS}$, then for every $\kappa\in (0,\kappa_0)$, the  first eigenfunctions of problem \eqref{CP-problem-0} on the spherical belt $\Omega\coloneqq B_\kappa(r,R)\subset \mathbb S_\kappa^2$ are sign-changing,  having a pair of azimuthally  opposite nodal circular arcs$.$
		\end{itemize}
	\end{theorem}

Theorem \ref{theorem-belt}  sheds light on the possibility of the rippling behavior of the first eigenfunctions on positively curved manifolds, similarly to the flat case, see also Figure \ref{1-abra}/(b). The proof of  Theorem \ref{theorem-belt} is based on fine properties of Gaussian hypergeometric functions and a careful asymptotic argument which traces us back to Euclidean annuli studied by  Coffman, Duffin and Schaffer \cite{CDS}.

Our second result contains precise asymptotic estimates of the fundamental tones for small and large clamped spherical caps, which will play a key role  in the proof of Lord Rayleigh's conjecture. Given $\kappa>0$ and $L\in (0,\pi/\sqrt{\kappa})$, the \textit{spherical cap} is  
$$C_\kappa^n(L)=\{x\in \mathbb S_\kappa^n:d_\kappa(N,x)<L\}. $$
The aforementioned asymptotic estimates are based on the first zeros of the transcendental equations 
\begin{equation}\label{transzdendental-2-dimenzio}
	\frac{\pi}{2}\tan\left(\pi\sqrt{\frac{1}{4}+\mu}\right)-\Psi\left(\sqrt{\frac{1}{4}+\mu}+\frac{1}{2}\right)+\Re\Psi\left(\sqrt{\frac{1}{4}-\mu}+\frac{1}{2}\right)=0, \ \ \mu>0,
\end{equation}
where $\Psi\coloneqq(\ln \Gamma)'$ is the Digamma function and $\Re z$ is the real part of $z\in \mathbb C,$ and
\begin{equation}\label{transzdendental-3-dimenzio}
	\sqrt{\mu-1}\coth\big(\pi\sqrt{\mu-1}\big)-\sqrt{\mu+1}\cot\big(\pi\sqrt{\mu+1}\big)=0,\ \ \mu>1.
%\Lambda_-\cot(2\widetilde \Lambda_-\arcsin (\sqrt{t}))- \Lambda_+\cot(2 \Lambda_+\arcsin (\sqrt{t}))=0
\end{equation}
 For later use, if  $\nu\geq 0$ is fixed,  $J_\nu$ and $I_\nu$ stand for the  Bessel and modified Bessel functions of the first kind, while $\mathfrak  h_{\nu}$ and $\mathfrak  j_{\nu}$ denote the first positive zeros of the cross-product ${J'_{\nu}}{I_{\nu}}-{J_{\nu}}{I'_{\nu}}$ and the Bessel function $J_\nu$, respectively. 

\begin{theorem}\label{aszimptotak-CP-pozitiv} {$($\rm\textbf{Spherical caps}$)$}
If  $n \in \mathbb N_{\geq 2}$, $\kappa > 0$ and $L\in (0,\pi/\sqrt{\kappa})$ are fixed and $C_\kappa^n(L)\subset \mathbb S_\kappa^n$  is a spherical cap, then we have the following asymptotic estimates.
	\begin{itemize}
		\item[(i)] {\rm Small spherical caps}$:$   \begin{equation}\label{small-caps}
			\Lambda_{\kappa}(C_\kappa^n(L))\sim\frac{\mathfrak  h_{\frac{n}{2}-1}^4}{L^4}\ \ {as}\ \ L\to 0.
%			\Lambda_{\kappa}(C_\kappa^n(L))\sim\left({\frac{(n-1)^2}{4}{\kappa}+\frac{\mathfrak  h_{\frac{n}{2}-1}^2}{L^2}}\right)^2\ \ {as}\ \ L\to 0;
		\end{equation}
		\item[(ii)] {\rm Large spherical caps}$:$   
		\begin{equation}\label{large-caps}
			\Lambda_{\kappa}(C_\kappa^n(L))\sim  \left\{ \begin{array}{lll}
				\mu_n^2 \kappa^2 &\mbox{if} & n\in \{2,3\}, \\
				% u\geq 0 &\mbox{in} &   \Omega;\\
				0 &\mbox{if} &  n\geq 4,
			\end{array}\right. \ \ {as}\ \ L\to \frac{\pi}{\sqrt{\kappa}},
		\end{equation}	
		where  $\mu_2\approx 0.9125$ and $\mu_3\approx 1.0277$ are the smallest positive zeros to the transcendental equations \eqref{transzdendental-2-dimenzio} and \eqref{transzdendental-3-dimenzio}, respectively. 
		% while $\mu_n=0$ for every $n\geq 4$.
	\end{itemize}
\end{theorem}

As expected, in small scales,  the fundamental tone has a Euclidean character (cf.\ relation \eqref{small-caps}); indeed, for any $L>0$, one has that 	$\Lambda_0(B_0(L))={\mathfrak  h_{\frac{n}{2}-1}^4}/{L^4},$ where $\Lambda_0(\Omega)$ comes from \eqref{variational-charact-0} and $B_0(L)\subset \mathbb R^n$ is the $n$-dimensional Euclidean ball of radius $L>0$ and center 0.  On the other hand, for large spherical caps (cf.\ relation \eqref{large-caps}), we surprisingly see an essential difference: while in high dimensions the fundamental tones are gapless, in low dimensions there are spectral gaps given by means of the first positive zeros to the transcendental equations \eqref{transzdendental-2-dimenzio} and \eqref{transzdendental-3-dimenzio}, respectively. In fact, relation \eqref{large-caps} provides another evidence to the tacitly accepted view concerning the difference between the low- and high-dimensional character of the fundamental tone for vibrating clamped plates. 
%; see also the result of Ashbaugh and Benguria \cite{A-B} for the flat case and Krist\'aly \cite{Kristaly-Adv-math} in the non-positively case.  
We also note that relation \eqref{large-caps} specifies a gap in  Cheng, Ichikawa and Mametsuka \cite[p. 676]{Cheng-Ichikawa-Mametsuka}, who  -- regardless of dimension -- claimed that the fundamental tone always tends to $0$ for clamped plates  converging to the whole sphere; note however that this statement is valid only for fixed membranes on the sphere, see e.g.\ Betz,  C\'amera and  Gzyl \cite{BCG}. 
The accuracy of estimates \eqref{small-caps} and  \eqref{large-caps} are shown in Tables \ref{table-11} and \ref{table-1}, respectively.  

%OSSZEKOTO SZOVEG IDE: SPECTRAL GAP PROBLEM. EZ KET SZEMPONTBOL FONTOS, EGYRESZT ONMAGAERT, MASRESZT AZERT, MERT A BIZNOYITAS EZEN IS ALAPSZIK (PL $n\geq  4$ eseten)! ETTOL FUGGETLENUL KIJELENTJUK, MINT KULON TETEL, FOR ITS OWN SAKE. 

Theorem \ref{aszimptotak-CP-pozitiv} paves the way to prove Lord Rayleigh's conjecture  on positively curved  spaces for any clamped plate in 2 and 3 dimensions, as well as for sufficiently large domains in high dimensions. The precise statement of our main result reads as follows:

%We notice that "Japanok" claimed the asymptotic estimate \eqref{large-caps} \textit{only} in high-dimensions, i.e., the gapless character of the fundamental tone of spherical caps tending to the whole sphere.   

	\begin{theorem}\label{fotetel-CP-pozitiv}  {$($\rm\textbf{Lord Rayleigh's conjecture; positively curved spaces}$)$}
	Let $(M,g)$ be a compact $n$-dimensional Riemannian manifold with ${\sf Ric}_{(M,g)}\geq (n-1)\kappa>0$ where $n\geq 2$. Then there exists $v_{n}\in [0,1)$ $($not depending on $\kappa>0),$ with $v_{2}=v_{3}=0$ and $v_n>0$ for $n\geq 4$ such that if  $\Omega\subset M$ is a smooth domain and $\Omega^\star\subset \mathbb S^n_\kappa$ is a  spherical cap with $\frac{V_g(\Omega)}{V_g(M)}=\frac{V_{\kappa}(\Omega^\star)}{V_{\kappa}(\mathbb S^n_\kappa)}> v_{n},$   then 
	\begin{equation}\label{foegyenlotlenseg}
	\Lambda_g(\Omega)\geq \Lambda_{\kappa}(\Omega^\star).
	\end{equation}
%	whenever any of the following assumptions hold$:$ 
%	\begin{itemize}
%		\item[(i)] $n\in \{2,3\};$
%		\item[(ii)] when $n\geq 4$ there exists $C_n\in (...)$ such that
%	\end{itemize}
		Equality holds in \eqref{foegyenlotlenseg} if and only if $(M,g)$ is isometric to $(\mathbb S_\kappa^n,g_{\kappa})$ and $\Omega$ is isometric to  $\Omega^\star$. 	In addition, 
$
v_\infty\coloneqq\displaystyle\limsup_{n\to \infty} v_{n}<1.$
\end{theorem}

Since $v_{2}=v_{3}=0$, there are no restrictions on the size of clamped plates in 2 and 3 dimensions. However, 
our arguments work only for sufficiently large domains $\Omega\subset M$ in higher dimensions, satisfying   ${V_g(\Omega)}> v_{n}{V_g(M)}$ with $v_n>0$ for every $n\geq 4$. 
% i.e., we can  handle only 'large' clamped domains in dimensions greater than or equal to  4.  
 Nevertheless, Theorem \ref{fotetel-CP-pozitiv} provides the first positive answer in any geometric setting to Lord Rayleigh's conjecture  in \textit{arbitrarily high dimensions}. We notice that $v_n\in (0,1)$, $n\geq 4$, is implicitly given as a solution to a highly nonlinear equation. However,  we have  $v_\infty=
 \displaystyle\limsup_{n\to \infty} v_{n}<1$, which shows that clamped plates in dimensions beyond 3 need not be
 particularly close -- in the sense of  volume -- to the whole manifold in order for the isoperimetric inequality \eqref{foegyenlotlenseg} to hold. Numerical tests indicate that $v_n\leq 1/2$ for every  $n\geq 4,$ and $v_\infty=1/2,$ see Table \ref{table-2}. Clearly, the most optimistic scenario would be to have $v_n=0$ for every  $n\geq 4,$ which definitely requires a new approach with respect to the one presented in our paper. 
 % \S \ref{high-dimension}.  

The first part of the proof of Theorem \ref{fotetel-CP-pozitiv}  is inspired by Talenti \cite{Talenti} together with the subsequent refinements of Ashbaugh and Benguria \cite{A-B} and Nadirashvili \cite{Nad}. Indeed, since the minimizer in \eqref{variational-charact} can be sign-changing (see e.g.\ Theorem \ref{theorem-belt}), a suitable nodal-decomposition is performed that we combine with the L\'evy--Gromov isoperimetric inequality, reducing the initial problem to a coupled minimization problem involving spherical caps on the model space $\mathbb S_\kappa^n$.\  Then, fine asymptotic properties of the fundamental tone for clamped spherical caps (cf.\ Theorem \ref{aszimptotak-CP-pozitiv}) combined with further features of Gaussian hypergeometric functions provide the proof of Theorem \ref{fotetel-CP-pozitiv}. 

We conclude the paper with the limit case when $\kappa\to 0$, i.e., $(M,g)$ is a complete non-compact $n$-dimensional Riemannian manifold with  ${\sf Ric}_{(M,g)}\geq 0$. The quantity 
\begin{equation}\label{volume-ratio}
	{\sf AVR}_{(M,g)}=\lim_{r\to \infty}\frac{{ V}_g(B_x(r))}{\omega_n r^n}
\end{equation}
stands for the  {\it asymptotic volume
	ratio} of $(M,g)$; here  $B_x(r)$  is the open metric ball on $M$ with center $x\in M$ and radius $r>0,$  and $\omega_n$ is the volume of the Euclidean unit ball in $\mathbb R^n$. By the Bishop--Gromov comparison theorem one has that ${\sf AVR}_{(M,g)}\leq 1,$ and this number is independent of the choice of
$x\in M,$  thus it is a global geometric invariant of $(M,g)$; moreover,  ${\sf AVR}_{(M,g)}= 1$ if and only if $(M,g)$ is isometric to the usual Euclidean space $(\mathbb R^n,g_0).$ 
%In the limit case, we are able to prove the following result. 
% We say that $(M,g)$ has \textit{Euclidean volume growth} if ${\sf AVR}_{(M,g)}>0.$ 

	\begin{theorem}\label{Huisken-CP} {$($\rm\textbf{Lord Rayleigh's conjecture; non-negatively curved spaces}$)$}
	Let $(M,g)$ be a complete non-compact $n$-dimensional $(n\geq 2)$ Riemannian manifold with ${\sf Ric}_{(M,g)}\geq 0$ and ${\sf AVR}_{(M,g)}>0$. Then  
	\begin{equation}\label{Huisken-egyenlotlenseg}
	\Lambda_g(\Omega)\geq {\sf AVR}_{(M,g)}^\frac{4}{n}   w_n \Lambda_{0}(\Omega^\star),
	\end{equation}
	for every smooth bounded domain $\Omega\subset M$,  where $\Omega^\star\subset \mathbb R^n$ is a ball with  $V_g(\Omega)=V_0(\Omega^\star)$ and 
	 \begin{equation}\label{w-n-limit-ertek}
	 	w_n=  \left\{ \begin{array}{lll}
	 		1 &\mbox{if} & n\in \{2,3\}, \\
	 		% u\geq 0 &\mbox{in} &   \Omega;\\
	 		2^\frac{4}{n} \frac{\mathfrak  j_{\frac{n}{2}-1}^4}{\mathfrak  h_{\frac{n}{2}-1}^4}&\mbox{if} &  n\geq 4.
	 	\end{array}\right. 
	 \end{equation}	
	 If $n\in \{2,3\},$ then equality holds in \eqref{Huisken-egyenlotlenseg} for some $\Omega\subset M$ if and only if $(M,g)$ is isometric to $(\mathbb R^n,g_0)$  and $\Omega\subset M$ is isometric to the ball  $\Omega^\star\subset \mathbb R^n$. In addition, $w_\infty\coloneqq
	 \displaystyle\lim_{n\to \infty} w_{n}=1.$
\end{theorem}

As in the proof of Theorem \ref{fotetel-CP-pozitiv}, the conclusion of Theorem \ref{Huisken-CP} follows from a similar 
%Ashbaugh--Benguria--Nadirashvili--Talenti
 nodal-decomposition, combined with a recent isoperimetric inequality, valid on any complete non-compact $n$-dimensional Riemannian manifold $(M,g)$ with ${\sf Ric}_{(M,g)}\geq 0$ and ${\sf AVR}_{(M,g)}>0$, see Brendle \cite{Brendle}  and Balogh and Krist\'aly \cite{Balogh-Kristaly} which is proved by the  ABP-method and the optimal mass transport theory, respectively. We also emphasize the strong rigidity character of inequality \eqref{Huisken-egyenlotlenseg} for $n\in \{2,3\}$; indeed, if a particular domain $\Omega\subset M$ produces equality in \eqref{Huisken-egyenlotlenseg}, the whole manifold $(M,g)$ turns out to be isometric to the Euclidean space $(\mathbb R^n,g_0)$, which follows by the  characterization of the equality in the aforementioned isoperimetric inequality (see \cite{Brendle, Balogh-Kristaly});  the rest immediately follows from Ashbaugh and Benguria \cite{A-B} for $n\in \{2,3\}$, and  Ashbaugh and Laugesen \cite{A-L} for $n\geq 4$.

\textit{Structure of the paper.}  In \S \ref{section-2} we state/prove  those properties of the sphere $\mathbb S_\kappa^n$ and Gaussian hypergeometric functions that are used in the paper. In \S \ref{section-3} we discuss the sign-changing character of the first eigenfunctions on spherical belts (see Theorem \ref{theorem-belt}). In \S \ref{section-4} we perform the Ashbaugh--Benguria--Nadirashvili--Talenti nodal-decomposition on positively curved manifolds, providing a sharp estimate of  the  fundamental tone of a clamped plate by a coupled minimization expression involving spherical caps. In \S \ref{section-5} we prove  sharp spectral gaps for small and large spherical caps (see Theorem \ref{aszimptotak-CP-pozitiv}). Lord Rayleigh's conjecture on positively curved Riemannian manifolds (see Theorem \ref{fotetel-CP-pozitiv}) is proved in \S\ref{section-6}, while the limit case (see Theorem \ref{Huisken-CP})  is discussed in \S\ref{section-7}. Finally,  Appendix \ref{section-8} contains well-known properties of special functions that are collected for an easier reading of the proofs.

\section{Preliminaries}\label{section-2}

\subsection{The model space $\mathbb S_\kappa^n$} Let  $n\in \mathbb N_{\geq 2}$ and $\kappa>0$. The set $\mathbb S_\kappa^n\subset \mathbb R^{n+1}$ is the $n$-dimensional sphere with radius $1/\sqrt{\kappa}$ (i.e., with constant curvature $\kappa$), endowed with its natural Riemannian metric $g_\kappa$. Let $(\theta, \xi)$ be the spherical coordinates on  $\mathbb S_\kappa^n$ with respect to the North pole $N=(0,\ldots,0,1/\sqrt{\kappa})\in \mathbb S_\kappa^n$, where $\theta\in (0,\pi)$ represents the latitude measurement along a unit speed geodesic from $N,$ while $\xi\in \mathbb S_1^{n-1}\eqqcolon\mathbb S^{n-1}$ is a parameter representing the choice of `azimuthal' direction of the geodesic in $\mathbb S_\kappa^n$. The distance from  $x=x(\theta,\xi)\in \mathbb S_\kappa^n$ to the North pole is $d_{\kappa}(N,x)={\theta}/{\sqrt{\kappa}}\in (0,{\pi}/{\sqrt{\kappa}})$. The  set $$C_\kappa^n(R)=\{x\in \mathbb S_\kappa^n: d_\kappa(N,x)<R\}$$ denotes the $n$-dimensional spherical cap with center $N$ and radius $R\in \left(0,{\pi}/{\sqrt{\kappa}}\right)$. Its volume is 
\begin{equation}\label{terfogat-cap}
V_{\kappa}(C_\kappa^n(R))=\int_{C_\kappa^n(R)}{\rm d}v_\kappa = n\omega_n\int_0^{R}\left(\frac{\sin(\sqrt{\kappa} \rho)}{\sqrt{\kappa}}\right)^{n-1}{\rm d}\rho,
\end{equation}
%Csak magunknak	$$dv_\kappa(x)=\frac{(\sin \theta)^{n-1}}{\sqrt{\kappa}^n}d\theta d\xi,\ \ x=x(\theta,\xi)$$
where ${\rm d} v_\kappa$ is the canonical measure on $\mathbb S_\kappa^n$ and  $\omega_n$ stands for the volume of the unit $n$-dimensional Euclidean ball. Performing a change of variables and using the relation \eqref{terfogat-cap}, for  every integrable function $h:[0,L]\to \mathbb R$  with $L\in [0,V_\kappa(\mathbb S^n_\kappa)]$   we have that  
\begin{equation}\label{change-variables}
\int_0^{L}h(s){\rm d}s=\int_{C_\kappa^n(R_L)}h(V_\kappa(C_\kappa^n(d_\kappa(N,x))){\rm d} v_\kappa(x),
\end{equation}
where $R_L\geq 0$ is the unique number for which $V_\kappa(C_\kappa^n(R_L))=L$. 
%%CSAK MAGUNKNAK MAGYARAZAT:
%%  s=V_\kappa(C_\kappa^n(d_\kappa(N,x))
%%  egyreszt, innen 
%$$ds=n\omega_n\left(\frac{\sin(\sqrt{\kappa} d_\kappa(N,x))}{\sqrt{\kappa}}\right)^{n-1}d\theta/\sqrt{\kappa}$$
%MASRESZT a fetibb osszefuggesbol
%${\rm d}v_\kappa(x(\theta,\xi))=n\omega_n\left(\frac{\sin(\theta)}{\sqrt{\kappa}}\right)^{n-1}d\theta/\sqrt{\kappa}$.

The spherical Laplacian on $\mathbb S_\kappa^n$ is 
\begin{equation}\label{Laplace-operator}
\Delta_\kappa w(x)\coloneqq \Delta_{g_\kappa} w(x)=\kappa\left({\sin \theta}\right)^{1-n}\frac{\partial}{\partial \theta}\left({(\sin \theta)}^{n-1}\frac{\partial w}{\partial \theta}\right)+\frac{\kappa}{\sin^{2} \theta}\Delta_\xi w,
\end{equation}
where $\Delta_\xi$ is the Laplace--Beltrami operator on the usual  $(n - 1)$-dimensional unit sphere $\mathbb S^{n-1}$.

%https://en.wikipedia.org/wiki/Laplace%E2%80%93Beltrami_operator

\subsection{Properties of Gaussian hypergeometric functions}

Let $a,b,c\in \mathbb C$ ($c\notin \mathbb Z_-$) and  $(a)_m=a(a+1)\ldots(a+m-1)=\frac{\Gamma(a+m)}{\Gamma(a)}$ be the Pochhammer symbol, $m \in \mathbb N$.
The  \textit{Gaussian hypergeometric function} is  
 \begin{equation}\label{F-ertelmezes}
{_2F}_1(a,b;c;z)= \sum_{m\geq 0}\frac{(a)_m(b)_m}{(c)_m}\frac{z^m}{m!},\ |z|<1,
 \end{equation}
and extended by analytic continuation elsewhere.
The corresponding differential equation to $z\mapsto {_2F}_1(a,b;c;z)$ is 
\begin{equation}\label{hyper-ODE}
z(1-z)w''(z)+(c-(a+b+1)z)w'(z)-abw(z)=0.
\end{equation}

%For every $\Lambda,C,\kappa>0$, let 
%\begin{equation}\label{lambda-Gamma-0}
%\gamma^\pm_\Lambda(\kappa)\coloneqq \sqrt{C\pm\frac{\Lambda^2}{\kappa^2}}\in \mathbb C.
%\end{equation}

\begin{proposition}\label{hipergeometrikus-Bessel-0-1}
 If $t\mapsto \lambda(t)$ is a positive function and  $\lim_{t\to 0}\lambda(t)= \ell>0,$ then for every $C_1,C_2,C\in \mathbb R$, $\mu>-1$ and $x>0$ one has that
\begin{align*}
&~	\lim_{t\to 0} {_2F}_1\left(C_1+\sqrt{C\pm\frac{\lambda^2(t)}{t}},C_2-\sqrt{C\pm\frac{\lambda^2(t)}{t}};1+\mu;\sin^2\left(\frac{\sqrt{t} x}{2}\right)\right)\\  =&~\Gamma(1+\mu)\left(\frac{2}{\ell x}\right)^\mu\left\{ \begin{array}{@~rll@~}
		J_\mu(\ell x)   &\text{for}&  \text{`}+\text{'}; \\
		I_\mu(\ell x)    &\text{for}&  \text{`}-\text{'}.
	\end{array}\right.
\end{align*}	
\end{proposition}

\begin{proof}
% where ${_2F}_1$ denotes the Olver's scaled hypergeometric function. 
By definition, for every $m\in \mathbb N$ we have that
$$\lim_{t\to 0} t^\frac{m}{2}\left(C_1+\sqrt{C\pm\frac{\lambda^2(t)}{t}}\right)_m=(\pm 1)^\frac{m}{2}\ell ^m$$ and $$ \lim_{t\to 0} t^\frac{m}{2}\left(C_2-\sqrt{C\pm\frac{\lambda^2(t)}{t}}\right)_m=(-1)^m(\pm 1)^\frac{m}{2}\ell ^m.$$
Therefore, 
\begin{align*}
	A_{\lambda,\mu}^\pm(x)&\coloneqq \lim_{t\to 0}{_2F}_1\left(C_1+\sqrt{C\pm\frac{\lambda^2(t)}{t}},C_2-\sqrt{C\pm\frac{\lambda^2(t)}{t}};1+\mu;\sin^2\left(\frac{\sqrt{t} x}{2}\right)\right)\\
	&=\lim_{t\to 0}\sum_{m=0}^\infty\frac{\left(C_1+\sqrt{C\pm\frac{\lambda^2(t)}{t}}\right)_m\left(C_2-\sqrt{C\pm\frac{\lambda^2(t)}{t}}\right)_m}{(1+\mu)_m}\frac{\sin^{2m}\left(\frac{\sqrt{t} x}{2}\right)}{m!}\\
	&=\Gamma(1+\mu)\sum_{m=0}^\infty\frac{(-1)^m(\pm 1)^m\ell^{2m}}{\Gamma(1+m+\mu) m!}\left(\frac{x}{2}\right)^{2m}
	\\&=\Gamma(1+\mu)\left(\frac{2}{\ell x}\right)^\mu\left\{ \begin{array}{rll}
		J_\mu(\ell x) &{\rm for}& \text{`}+\text{'}; \\
		I_\mu(\ell x)  &{\rm for}& \text{`}-\text{'},
	\end{array}\right.\end{align*}
which concludes the proof. 
\end{proof}

For every  $\mu\geq 0$ and $n\geq 2$, we consider the number
\begin{equation}\label{Lambda-elso-definicio}
	\Lambda_\pm(\mu)\coloneqq \sqrt{\frac{(n-1)^2}{4}\pm \mu}\in \mathbb C,
\end{equation}
and  the specific  Gaussian hypergeometric function  \begin{equation}\label{Ferrers-fct0}
v_\pm(\mu,t)\coloneqq {_2F}_1\left(\frac{1}{2}-\Lambda_\pm(\mu),\frac{1}{2}+\Lambda_\pm(\mu);\frac{n}{2};t\right)=\sum_{m=0}^\infty \beta_m^\pm(\mu) t^m,\ \ t\in (0,1),
\end{equation}
where \begin{equation}\label{beta-sorozat}
	\beta_m^\pm(\mu)\coloneqq\frac{(\frac{1}{2}-\Lambda_{\pm}(\mu))_m(\frac{1}{2}+\Lambda_{\pm}(\mu))_m}{m!(\frac{n}{2})_m},\ m\in \mathbb N.
\end{equation}

%Hereafter,    $\Psi=(\ln \Gamma)'$ stands for the Digamma function.
	We now collect those properties of the Gaussian hypergeometric functions that are important in our further investigations,  most of which coming by  well-known properties listed in Olver,  Lozier,  Boisvert and Clark  \cite{Digital} and recalled in  Appendix \ref{section-8}. 

\begin{proposition}\label{Ferrers-basic-lemma-1} If  $n\in \mathbb N_{\geq 2}$, the following properties hold$:$ 
\begin{itemize}
	\item[(i)] $v_\pm(0,t)=(1-t)^{\frac{n}{2}-1}$ for every $t\in (0,1);$  
	\item[(ii)] $v_-(\mu,t)>0$  for every $\mu>0$ and $t\in (0,1);$  
	%integer next less than $a\in \mathbb R;$
	%, while $\sigma_2\in \{0,1\}$ and  $\sigma_n=0$ for $n\geq 3;$ 
	\item[(iii)] for every $t\in (0,1)$ the function $\mu\mapsto v_+(\mu,t)$ has infinitely many zeros in $[0,\infty);$ 
	\item[(iv)] for every $\mu>0$ the number of zeros of the mapping $t\mapsto v_+(\mu,t)$ in  $(0,1)$ is given by the {\rm Klein-number} $s_n^\mu\coloneqq \left\lfloor\Lambda_+(\mu) -\frac{n-3}{2}\right\rfloor+k_n^\mu \geq 1,$  where $\lfloor a\rfloor$ stands for the 
	greatest integer less than $a>0$ and $k_n^\mu\in \{0,1\}$ $($additionally, $k_n^\mu=0$ for every $n\geq 3$ and $\mu>0$$);$
		% ha a\in (n,n+1], akkor az ertek a. 
\item[(v)] for every $\mu>0$ with $\Lambda_{\pm}(\mu)-\frac{1}{2}\notin \mathbb Z$,	the function 	$t\mapsto \frac{w_\pm(\mu,t)}{v_\pm(\mu,t)}$ is decreasing between any two consecutive zeros of $v_\pm(\mu,\cdot)$, where  $$w_\pm(\mu,t)=\sum_{m=0}^\infty \beta_m^\pm(\mu) \frac{\Psi(m+\frac{1}{2}+\Lambda_{\pm}(\mu))-\Psi(m+\frac{1}{2}-\Lambda_{\pm}(\mu))}{\Lambda_{\pm}(\mu)} t^m,\ \ t\in (0,1),$$ 
$($with the convention that a limit is taken in $w_-(\mu,t)$ whenever $\Lambda_-(\mu)=0$$);$
\item[(vi)] $v_+\left(\mu,\frac{1}{2}\right)>0=v_+\left(n,\frac{1}{2}\right)$ for every $\mu\in [0,n)$.
\end{itemize} 

\end{proposition}
\begin{proof} (i) Using the Euler--Pfaff transformation \eqref{Olver-1},  we have that
	$$v_\pm(0,t)={_2F}_1\left(1-\frac{n}{2},\frac{n}{2};\frac{n}{2};t\right)=(1-t)^{\frac{n}{2}-1}{_2F}_1\left(1-\frac{n}{2},0;\frac{n}{2};\frac{t}{t-1}\right)=(1-t)^{\frac{n}{2}-1},~ \forall  t\in (0,1).$$

(ii) Assume that $\mu>0$ and let $A\coloneqq \frac{1}{2}-\Lambda_-(\mu)\in \mathbb C$ and 	$B\coloneqq \frac{1}{2}+\Lambda_-(\mu)\in \mathbb C.$ 
First, if $\mu>\frac{n(n-2)}{4}$, then it follows that $(A)_m(B)_m>0$ for every $m\in \mathbb N.$ Therefore, by the definition \eqref{F-ertelmezes} we obtain that $v_-(\mu,t)>0$ for every $t\in (0,1).$ Now, if $0<\mu\leq \frac{n(n-2)}{4}$, it turns out by \eqref{Olver-1} that  
\begin{align*}
v_-(\mu,t)=& ~ (1-t)^{\frac{n}{2}-A-B}{_2F}_1\left(\frac{n}{2}-A,\frac{n}{2}-B;\frac{n}{2};t\right)\\=& ~ (1-t)^{\frac{n}{2}-1}{_2F}_1\left(\frac{n-1}{2}+\Lambda_-(\mu),\frac{n-1}{2}-\Lambda_-(\mu);\frac{n}{2};t\right),~ \forall  t\in (0,1).
\end{align*}
Since $0<\mu\leq \frac{n(n-2)}{4}$, every parameter in the latter expression is real and positive,  implying again that $v_-(\mu,t)>0$ by means of \eqref{F-ertelmezes}.

	(iii) By formula \eqref{P-hypergoemetric}, the property that $\mu\mapsto v_+(\mu,t)$ has infinitely many zeros in $[0,\infty)$  for every fixed $t\in (0,1)$, is a consequence of the result of  MacDonald \cite{MacDonald}; see also Hobson \cite[p.\ 403--406]{Hobson} and Baginski \cite{Baginski}. 
	
	(iv) This property is attributed to Klein \cite{Klein} for $n\geq 3$ and Gormley \cite[p.\ 30]{Gormley} for $n=2$.

	(v)  Let $\mu>0$ with $\Lambda_{\pm}(\mu)-\frac{1}{2}\notin \mathbb Z$.  Due to the explicit forms of $w_\pm$ and $v_\pm$, and basic properties of the Digamma function $\Psi=(\ln \Gamma)'$, the monotonicity of 	$t\mapsto \frac{w_\pm(\mu,t)}{v_\pm(\mu,t)}$ follows by  direct computations that we illustrate in certain cases; similar arguments are provided by Yang,  Chua and Wang \cite{kinaiak} and Holtz and Tyaglov \cite{HT}. 
%	We assume first that $\Lambda_{\pm}(\mu)\neq l-\frac{1}{2}$ for any $l\in \mathbb Z.$ 
	
	Whenever $\Lambda_{\pm}(\mu)\neq 0,$ we will also use the notation
	 $$\alpha_m^\pm(\mu)\coloneqq \beta_m^\pm(\mu) \frac{\Psi(m+\frac{1}{2}+\Lambda_{\pm}(\mu))-\Psi(m+\frac{1}{2}-\Lambda_{\pm}(\mu))}{\Lambda_{\pm}(\mu)},\ m\in \mathbb N,$$
	whenever  $\Lambda_{\pm}(\mu)\neq 0$.
%	 $A_\pm^\mu(t)=\sum_{m=0}^\infty \alpha_m^\pm(\mu)t^m$ and $B_\pm^\mu(t)=\sum_{m=0}^\infty \beta_m^\pm(\mu)t^m$, where 
On account of relations \eqref{digamma-prop-11}--\eqref{digamma-prop-12},  if $m_\mu^\pm\coloneqq \left[\Re \Lambda_{\pm}(\mu)+\frac{1}{2}\right]\in \mathbb N$ (where $[a]$ is the integer part of $a\in \mathbb R$), then $\left(\frac{\alpha_m^\pm(\mu)}{\beta_m^\pm(\mu)}\right)_{m=\overline{0,m_\mu^\pm}}$ is increasing and $\left(\frac{\alpha_m^\pm(\mu)}{\beta_m^\pm(\mu)}\right)_{m\geq {m_\mu^\pm}}$ is decreasing, respectively.

For the case $\text{`}-\text{'}$, by (ii) we know that $v_-(\mu,\cdot)>0$	on $(0,1)$.  Let us consider 
	% In order to give a flavor of the proof, let us consider for exemplification the case when 
	$\mu>\frac{1}{4}{{n(n-2)}}$ (for every $n\geq 2$); thus $m_\mu^-=0$ and $\beta_m^-(\mu)>0$ for every $m\geq 1.$ Moreover, by the latter property it follows that  $\left(\frac{\alpha_m^-(\mu)}{\beta_m^-(\mu)}\right)_{m\geq 0}$ is positive and decreasing. A direct calculation or  the monotonicity result of Yang,  Chua and Wang \cite{kinaiak} (see also   Biernacki and  Krzyz \cite{BK}) implies that the function $(0,1)\ni t\mapsto \frac{w_-(\mu,t)}{v_-(\mu,t)}$ is also  decreasing. Analogously, when $\frac{1}{4}{{n(n-2)}}\geq \mu>\frac{1}{4}{{(n+2)(n-4)}}$
	%(n(n-2)-3)
	(for every $n\geq 4$), it follows that $m_\mu^-=1$, $\beta_0^-(\mu)=1$ and $\beta_m^-(\mu)<0$ for every $m\geq 1$, while $\alpha_m^-(\mu)<0$ for every $m\geq 0$ (see \eqref{digamma-prop-11}--\eqref{digamma-prop-12}), and  the sequence $\left(\frac{\alpha_m^-(\mu)}{\beta_m^-(\mu)}\right)_{m\geq 1}$ is  decreasing. Therefore, $v_-'(\mu,\cdot)<0$ and $v_-(\mu,\cdot)>0$ (due to (ii)), and in a similar way as above,  the function $\frac{w'_-(\mu,\cdot)}{v_-(\mu,\cdot)}$ is  decreasing  on $(0,1)$. If $H_\mu(t)\coloneqq \frac{w'_-(\mu,t)}{v'_-(\mu,t)}v_-(\mu,t)-w_-(\mu,t)=\left(\frac{w_-(\mu,t)}{v_-(\mu,t)}\right)'\frac{v_-^2(\mu,t)}{v'_-(\mu,t)}$, then it turns out that  $H_\mu'(t)=\left(\frac{w'_-(\mu,t)}{v'_-(\mu,t)}\right)'v_-(\mu,t)< 0$ for every $t\in (0,1)$, i.e., $H_\mu$ is decreasing on $(0,1)$. Thus,  one has $H_\mu(t)>\lim_{v\nearrow 1}H_\mu(v)=\frac{w'_-(\mu,1)}{v'_-(\mu,1)}v_-(\mu,1)-w_-(\mu,1)>0$ for every $t\in (0,1)$, which implies that $\left(\frac{w_-(\mu,\cdot)}{v_-(\mu,\cdot)}\right)'<0$ on $(0,1)$, concluding the proof. Generically, if $k\geq 1$ and  $\frac{1}{4}{{(n+2k-2)(n-2k)}}\geq \mu>\frac{1}{4}{{(n+2k)(n-2k-2)}}$
	%(n(n-2)-3)
	(for every $n\geq 2k+2$), then $m_\mu^-=k$
	and   a  similar argument as above implies  that $t\mapsto \frac{w_-(\mu,t)}{v_-(\mu,t)}$ is decreasing on $(0,1)$. When $\Lambda_-(\mu)= 0$, we consider the limit in $w_-$, replacing the expression
	$\frac{\Psi(m+\frac{1}{2}+\Lambda_-(\mu))-\Psi(m+\frac{1}{2}-\Lambda_-(\mu))}{\Lambda_-(\mu)}$ by $ 2\Psi\left(1,m+\frac{1}{2}\right)$ for all $ m\in \mathbb N,$
	where $\Psi(1, x) \coloneqq \frac{\rm d}{{\rm d}x}\Psi(x)$; thus $\alpha_m^-(\mu)=2\beta_m^-(\mu)\Psi\left(1,m+\frac{1}{2}\right)$ and the rest is similar as above. 
	
For the case $\text{`}+\text{'}$, let $t_1^\mu <\ldots<t_{s_n^\mu}^\mu$ be the zeros of $v_+(\mu,\cdot)$ in $(0,1)$, 
 where  $s_n^\mu\geq 1$ is the Klein-number  by (iii). 
 A similar procedure as before, or 
% Since $v_+(\mu,\cdot)$ and $w_+(\mu,\cdot)$ are interlacing, 
%  (i.e., between the zeros of  $v_+(\mu,\cdot)$, whose number is given by $s_n^\mu$,  there is exactly one zero of $w_+(\mu,\cdot)$), 
 by adapting the argument from Holtz and Tyaglov \cite
  %[Theorem 3.4]
  {HT} to our setting shows that $w_+'(\mu,\cdot)v_+(\mu,\cdot)-w_+(\mu,\cdot)v'_+(\mu,\cdot)<0$ on $(0,1)\setminus\{t_1^\mu,\ldots,t_{s_n^\mu}^\mu\}$, concluding the claim.  
 
% A similar argument as above, applied on the intervals $(0,t_1^\mu)$, $(t_k^\mu,t_{k+1}^\mu)$ for $k=1,\ldots,s_n^\mu-1$  and $(t_{s_n^\mu}^\mu,1)$, provides the desired statement. 

%  If  $\Lambda_-(\mu)= l-\frac{1}{2}$ or $\Lambda_+(\mu)= l-\frac{1}{2}$ for some $l\in \mathbb Z,$ the same monotonicity property holds, whenever the involved function series reduce to polynomials. 

(vi) By \eqref{Olver-1} and \eqref{F-ertelmezes}, one has    that
$$
	v_+\left(n,\frac{1}{2}\right)={_2F}_1\left(-\frac{n}{2},\frac{n}{2}+1;\frac{n}{2};\frac{1}{2}\right)=\frac{1}{2^{\frac{n}{2}-1}}{_2F}_1\left(n,-1;\frac{n}{2};\frac{1}{2}\right)=0.
$$
Moreover, if $0\leq\mu<n$, then $\Lambda_+(\mu)<\frac{n+1}{2}$, thus by \eqref{fel-terfogathoz} it follows that
$$v_+\left(\mu,\frac{1}{2}\right)=\frac{2^{1-\frac{n}{2}}\sqrt{\pi}\Gamma(\frac{n}{2})}{\Gamma\left(\frac{n+1}{4}+\frac{\Lambda_+(\mu)}{2}\right)\Gamma\left(\frac{n+1}{4}-\frac{\Lambda_+(\mu)}{2}\right)}>0,$$
which concludes the proof. \end{proof}

\section{Clamped spherical belts: proof of Theorem \ref{theorem-belt}}\label{section-3}

 Let $R>r>0$, $\lambda>0$ and fix $\kappa\in \left(0,(\pi/ R)^2\right)$. Particularizing  \eqref{CP-problem-0}, we consider the clamped plate problem on the spherical belt $B_\kappa(r,R)\subset \mathbb S_\kappa^2$, i.e., 
 \tagequation{CP-problem}
 {\left\{
 	\begin{array}
 		[c]{@{~}r@{~}l@{~}lll@{~}}
 		\Delta_{\kappa}^2 w & = & \lambda^4 w & \text{in} & B_\kappa(r,R),\\
 		w&= & \dfrac{\partial w}{\partial\mathbf{n}}  =  0 & \text{on} &\partial B_\kappa(r,R).
 	\end{array}
 	\right.}
 {P_{\kappa}}
%\begin{equation*}\label{CP-problem}
%\left\{
%\begin{array}
%	[c]{@{~}r@{~}l@{~}lll@{~}}
%	\Delta_{\kappa}^2 w & = & \lambda^4 w & \text{in} & B_\kappa(r,R),\\
%	w&= & \dfrac{\partial w}{\partial\mathbf{n}}  =  0 & \text{on} &\partial B_\kappa(r,R).
%\end{array}
%\right.\eqno{(P)_{\kappa}}
%\end{equation*}
Here, for any  $x=x\left(\theta,\xi\right)\in \mathbb S_\kappa^2$,  
the spherical Laplacian on $\mathbb S_\kappa^2$ from \eqref{Laplace-operator} reduces to 
\begin{equation}\label{Laplace-operator-1}
\Delta_\kappa w(x)=\frac{\kappa}{\sin \theta}\frac{\partial}{\partial \theta}\left({\sin \theta}\frac{\partial w}{\partial \theta}\right)+\frac{\kappa}{\sin^2 \theta}\frac{\partial^2w}{\partial\xi^2}.
\end{equation}

%	
%	
%%Let $\kappa>0$ be fixed. Then the following statements hold: 
%	\begin{itemize}
%		\item[{\rm (i)}] On \textbf{fat} spherical belts of $\mathbb S_\kappa^2,$ the first eigenfunctions in $(P)$ are \textbf{sign-changing}  with two opposite nodal geodesic segments$;$
%		
%		\item[{\rm (ii)}] On  \textbf{thin} spherical belts $\mathbb S_\kappa^2,$ the first eigenfunctions in $(P)$ are \textbf{sign-preserving}. 
%	\end{itemize}
% Moreover, the threshold value of  $c_0$ is approximately ${762.36}$. 

%The proof will be performed in three steps. In the first two steps we discuss the sign-preserving and sign-changing eigenfunctions of problem $(P)$ on generic spherical belts. The crucial point of the proof is a careful limiting and comparison procedure which provide the sharp transition from sign-preserving to sign-changing solutions for problem $(P)$.

%\textbf{Step 1}. \textit{Nodal eigenfunctions for \eqref{CP-problem}}.    

As the following subsections show. the proof of Theorem \ref{theorem-belt} is divided into three parts.  

\subsection{Narrow spherical belts:\ eigenfunctions of fixed sign}
%	of $(P)_\kappa$ and the limiting $\kappa\to 0$}  
We first observe that, if  $w:B_\kappa(r,R)\to \mathbb R$ is an eigenfunction of \eqref{CP-problem},  the same is true for its $\xi$-average  $$\widetilde w(x)=\frac{1}{2\pi}\int_0^{2\pi}w(x(\theta,\xi)){\rm d}\xi.$$
We also notice that $\widetilde w$ is azimuthally-invariant (or, spherical cap symmetric), i.e., $\widetilde w(x(\theta,\cdot))$ is constant for every fixed $\theta\in (\sqrt{\kappa} r,\sqrt{\kappa} R).$ In particular, by using 
\eqref{Laplace-operator}, $\widetilde w$ is an eigenfunction of the ordinary differential equation  
\begin{equation}\label{radialis-megoldas}
\left\{ \begin{array}{lll}
\dfrac{\kappa^2}{\sin \theta}\dfrac{{\rm d}}{{\rm d} \theta}\left(\sin \theta\dfrac{{\rm d} }{{\rm d} \theta}\left(\dfrac{1}{\sin\theta}\dfrac{{\rm d}}{{\rm d} \theta}\left(\sin \theta\dfrac{{\rm d} \widetilde w}{{\rm d} \theta}\right)\right)\right)=\lambda^4 \widetilde w, & &  \theta\in (\sqrt{\kappa} r,\sqrt{\kappa} R), \\
\\
% u\geq 0 &\mbox{in} &   \Omega;\\
\widetilde w(\sqrt{\kappa} r)=\dfrac{{\rm d} \widetilde w}{{\rm d} \theta}(\sqrt{\kappa} r)=\widetilde w(\sqrt{\kappa} R)=\dfrac{{\rm d} \widetilde w}{{\rm d} \theta}(\sqrt{\kappa} R)=0;
\end{array}\right.
\end{equation}
moreover $\widetilde{w}$ is of fixed sign  whenever it is not identically zero, see, e.g.\ Leighton and Nehari \cite{LN-TAMS}. 
 
For further use, by applying \eqref{Lambda-elso-definicio} for $n=2,$ we introduce the numbers  
\begin{equation}\label{lambda-Gamma}
\gamma^\pm_\lambda(\kappa)\coloneqq \Lambda_\pm\left(\frac{\lambda^2}{\kappa}\right)=\sqrt{\frac{1}{4}\pm\frac{\lambda^2}{\kappa}}\in \mathbb C,
\end{equation}
while for  $z\in \mathbb C$ and $ \theta\in (\sqrt{\kappa} r,\sqrt{\kappa} R)$ we also define the Gaussian hypergeometric function $$\mathcal P(z,\theta)\coloneqq {_2F}_1\left(\frac{1}{2}+z,\frac{1}{2}-z;1;\sin^2\left(\frac{ \theta}{2}\right)\right) $$  
as well as
\begin{align*}
	\mathcal Q(z,\theta)\coloneqq&\sum_{m=0}^\infty\frac{\left(\frac{1}{2}+z\right)_m\left(\frac{1}{2}-z\right)_m}{(m!)^2}\sin^{2m}\left(\frac{ \theta}{2}\right)  \left( \Psi\left(\frac{1}{2}+z+m\right)+\Psi\left(\frac{1}{2}-z+m\right)-2\Psi\left(1+m\right)\right) \\& +\mathcal P(z,\theta)\ln\left(\sin^2\left(\frac{ \theta}{2}\right)\right),
\end{align*}
% (15.10.8)
 whenever $\frac{1}{2}\pm z\neq 0,-1,-2,\dots$, see Olver, Lozier,  Boisvert and Clark \cite[rel.\ (15.10.8)]{Digital}.

By the factorization $(\Delta_{\kappa} w-\lambda^2 w)(\Delta_{\kappa} w+\lambda^2 w)=\Delta_{\kappa}^2 w-\lambda^4 w$ and \eqref{Laplace-operator-1}, we observe that  
the azimuthally-invariant function
\begin{equation}\label{w-eigenfunction}
w(x)=C_1\mathcal P(\gamma^+_\lambda(\kappa),\theta)  + C_2\mathcal Q(\gamma^+_\lambda(\kappa),\theta) +C_3\mathcal P(\gamma^-_\lambda(\kappa),\theta)  +C_4\mathcal Q(\gamma^-_\lambda(\kappa),\theta)
\end{equation}
verifies  the first equation of \eqref{radialis-megoldas}, where $x=x\left(\theta,\xi\right)\in B_\kappa(r,R)$ and the constants
 $\left\{C_i\right\}_{i=1}^{4}\subset   \mathbb R$  are not all zero; hereafter we consider the general case $\frac{1}{2}\pm \gamma^\pm_\lambda(\kappa)\neq 0,-1,-2,\dots$, as the complementary cases are obtained by  limits, see \cite[rel.\ (15.10.9)-(15.10.10))]{Digital}. 
% 
%  ${\bf P}_{\nu}$ and  ${\bf Q}_{\nu}$ being the  Ferrers functions of the first and second kind.  
  If $w$ satisfies  the boundary conditions in \eqref{radialis-megoldas}, then $w$ is of fixed sign (see \cite{LN-TAMS}) and we obtain four equations in $\left\{C_i\right\}_{i=1}^{4}$. Since some of these constants are  non-zero, we necessarily have that  
\begin{equation}\label{determinant-sign-preserving}
{\rm det}\left[\def\arraystretch{1.3} \begin{matrix}
\mathcal P(\gamma^+_\lambda(\kappa),\sqrt{\kappa} r) & \mathcal Q(\gamma^+_\lambda(\kappa),\sqrt{\kappa} r) &\mathcal P(\gamma^-_\lambda(\kappa),\sqrt{\kappa} r) & \mathcal Q(\gamma^-_\lambda(\kappa),\sqrt{\kappa} r)\\
\mathcal P'(\gamma^+_\lambda(\kappa),\sqrt{\kappa} r) & \mathcal Q'(\gamma^+_\lambda(\kappa),\sqrt{\kappa} r) &\mathcal P'(\gamma^-_\lambda(\kappa),\sqrt{\kappa} r) & \mathcal Q'(\gamma^-_\lambda(\kappa),\sqrt{\kappa} r)\\
\mathcal P(\gamma^+_\lambda(\kappa),\sqrt{\kappa} R) & \mathcal Q(\gamma^+_\lambda(\kappa),\sqrt{\kappa} R) &\mathcal P(\gamma^-_\lambda(\kappa),\sqrt{\kappa} R) & \mathcal Q(\gamma^-_\lambda(\kappa),\sqrt{\kappa} R)\\
\mathcal P'(\gamma^+_\lambda(\kappa),\sqrt{\kappa} R) & \mathcal Q'(\gamma^+_\lambda(\kappa),\sqrt{\kappa} R) &\mathcal P'(\gamma^-_\lambda(\kappa),\sqrt{\kappa} R) & \mathcal Q'(\gamma^-_\lambda(\kappa),\sqrt{\kappa} R)
\end{matrix}\right]=0,
\def\arraystretch{1}
\end{equation}
where $P'(z,\theta)\coloneqq\frac{\partial}{\partial \theta}P(z,\theta)$ and $Q'(z,\theta)\coloneqq\frac{\partial}{\partial \theta}Q(z,\theta).$

Let $\lambda\eqqcolon\lambda^{SP}_{r,R}(\kappa)>0$ in \eqref{lambda-Gamma} be the 
 smallest positive zero  of the equation  \eqref{determinant-sign-preserving}; it follows that any eigenvalue  that corresponds to an eigenfunction of fixed sign of \eqref{CP-problem} cannot be less than $\lambda^{SP}_{r,R}(\kappa)$. By analyticity, the function $\kappa\mapsto \lambda^{SP}_{r,R}(\kappa)$ is continuous and let 
 \begin{equation}\label{Lambda-0}
  \lambda_0\coloneqq \lambda^{SP}_{r,R}=\lim_{\kappa\to 0}\lambda^{SP}_{r,R}(\kappa).
 \end{equation} 
% In the sequel we are going to take the limit  $\kappa\to 0$  in \eqref{determinant-sign-preserving} for $\lambda\coloneqq \lambda^{SP}_{r,R}(\kappa)$. 
% Before explicitly doing this, we observe that letting $\kappa\to 0$ in the  first relation of 
% \eqref{radialis-megoldas}, 
% it follows
% $$
% \theta^{-1}\frac{{\rm d}}{{\rm d} \theta}\left(\theta\frac{{\rm d} }{{\rm d} \theta}\left(\theta^{-1}\frac{{\rm d}}{{\rm d} \theta}\left(\theta\frac{{\rm d} \widetilde w}{{\rm d} \theta}\right)\right)\right)=\lambda_0^4 \widetilde w\ \  \mbox{for}\ \    \theta\in (r,R),$$
% whose solutions -- after a natural factorization -- are the Bessel and modified Bessel  functions of first and second kinds $\theta\mapsto J_0(\lambda_0\theta)$, $I_0(\lambda_0\theta)$, $Y_0(\lambda_0\theta)$, 
% $K_0(\lambda_0\theta)$, respectively. Therefore, by the limiting procedure in \eqref{determinant-sign-preserving} we may expect the presence of certain terms involving Bessel functions. 

We  show that, for every $x>0$ one has the limit
 \begin{equation}\label{P-0-limit-0}
 \lim_{\kappa\to 0}{\mathcal  P}\big(\gamma^\pm_{\lambda^{SP}_{r,R}(\kappa)}(\kappa), \sqrt{\kappa} x\big) =\left\{ \begin{array}{rll}
 J_0(\lambda_0 x) &{\rm for}& \text{`}+\text{'}; \\
 I_0(\lambda_0 x)  &{\rm for}& \text{`}-\text{'},\end{array}\right.
 \end{equation}
 and 
  \begin{equation}\label{P-0-limit-1}
 \lim_{\kappa\to 0}{\mathcal  Q}(\gamma^\pm_{\lambda^{SP}_{r,R}(\kappa)}(\kappa), \sqrt{\kappa} x)=\left\{ \begin{array}{rll}
 \pi Y_0(\lambda_0 x)
% -2\ln(\lambda_0)J_0(\lambda_0 x) 
 &{\rm for}& \text{`}+\text{'}; \\
-2 K_0(\lambda_0 x)
%-2\ln(\lambda_0)I_0(\lambda_0 x)  
&{\rm for}& \text{`}-\text{'},\end{array}\right.
 \end{equation}
where $Y_\nu$ and $K_\nu$ are the Bessel and modified Bessel functions of the second kind  $(\nu\geq 0)$, respectively, see  Appendix \ref{section-8}.  We first observe that $\lambda_0>0;$ otherwise the terms involving the function $\mathcal Q$ blow up  whenever $\kappa\to 0,$ by reaching the branch point $0$ of both $Y_0$ and $K_0$. %%branching point

Relation \eqref{P-0-limit-0} immediately follows by Proposition \ref{hipergeometrikus-Bessel-0-1} (with $\mu=0$). To prove \eqref{P-0-limit-1}, by  Proposition \ref{hipergeometrikus-Bessel-0-1} (with $\mu=0$) and relations \eqref{Digamma-converges} and \eqref{Y_n-forma} (with $n=0$) we obtain that 
  %\begin{equation}\label{P-0-limit-2}
$$ \lim_{\kappa\to 0}{\mathcal  Q}(\gamma^+_{\lambda^{SP}_{r,R}(\kappa)}(\kappa), \sqrt{\kappa} x)=
 \pi Y_0(\lambda_0 x).
% -2\ln(\lambda_0)J_0(\lambda_0 x).
 $$
% \end{equation}
 Furthermore, by using the fact that $\ln \mathfrak{i}=\frac{\pi}{2}\mathfrak{i}$ (where $\mathfrak{i}=\sqrt{-1}$) and relations \eqref{Bessel-I}, \eqref{Y-I-K} and \eqref{Y_n-forma}, we also have that
 \begin{align}\label{P-0-limit-3}
 \nonumber \lim_{\kappa\to 0}{\mathcal  Q}(\gamma^-_{\lambda^{SP}_{r,R}(\kappa)}(\kappa), \sqrt{\kappa} x)&= 2
 I_0(\lambda_0 x)\ln\left(\frac{\lambda_0 x}{2}\right)-2\sum_{m=0}^\infty\frac{1}{(m!)^2}\left(\frac{\lambda_0 x}{2}\right)^{2m}\Psi(1+m)\\&=-2
 K_0(\lambda_0 x),\nonumber
 \end{align}
 which concludes the proof of \eqref{P-0-limit-1}. 
 
 Due to relations \eqref{P-0-limit-0}--\eqref{P-0-limit-1}, the analyticity of the aforementioned special functions, the limit argument in relation \eqref{determinant-sign-preserving} and simple properties of the determinants  imply that
 \begin{equation}\label{determinant-sign-preserving-1}
 {\rm det}\left[\begin{matrix}
 J_0(\lambda_0r) & Y_0(\lambda_0r) &I_0(\lambda_0r) & K_0(\lambda_0r)\\
 J_0'(\lambda_0r) & Y_0'(\lambda_0r) &I'_0(\lambda_0r) & K'_0(\lambda_0r)\\
 J_0(\lambda_0R) & Y_0(\lambda_0R) &I_0(\lambda_0R) & K_0(\lambda_0R)\\
 J_0'(\lambda_0R) &Y_0'(\lambda_0R) &I_0'(\lambda_0R) & K_0'(\lambda_0R)
 \end{matrix}\right]=0.
 \end{equation}
% where we applied some trivial properties of the determinant. 
% where $J_\mu,Y_\mu,I_\mu$ and $K_\mu$ denote the usual Bessel and modified Bessel functions of the first and second kind of degree $\mu\in \mathbb R$, respectively. 
% In fact, a precise argument shows below that 

%    provides the first eigenvalue in (\ref{radialis-megoldas}), associated with the sign-preserving eigenfunction $w$ in \eqref{w-eigenfunction}.

\subsection{Wide spherical belts:\ sign-changing eigenfunctions}
	%of $(P)_\kappa$ and the limiting $\kappa\to 0$}. 
For every $z\in \mathbb C$ with  $\frac{1}{2}\pm z\neq 0,-1,-2,\dots,$  and $ \theta\in (\sqrt{\kappa} r,\sqrt{\kappa} R)$ we consider the functions $$\mathcal F(z,\theta)\coloneqq {_2F}_1\left(\frac{3}{2}+z,\frac{3}{2}-z;2;\sin^2\left(\frac{ \theta}{2}\right)\right), $$  
and
\begin{align*}
\mathcal H(z,\theta)\coloneqq &~\mathcal F(z,\theta)\ln\left(\sin^2\left(\frac{ \theta}{2}\right)\right)+\frac{1}{\frac{1}{4}-z^2}\sin^{-2}\left(\frac{ \theta}{2}\right)\\&+ \sum_{m=0}^\infty\frac{\left(\frac{3}{2}+z\right)_m\left(\frac{3}{2}-z\right)_m}{(m+1)!m!}\sin^{2m}\left(\frac{ \theta}{2}\right)\cdot \\&\ \ \ \ \cdot \left( \Psi\left(\frac{3}{2}+z+m\right)+\Psi\left(\frac{3}{2}-z+m\right)-\Psi\left(1+m\right)-\Psi\left(2+m\right)\right).
\end{align*}
%
%
%For $\lambda>0$ fixed, let  
%\begin{equation}\label{lambda-Gamma-punctured}
%\gamma^\pm_\lambda(\kappa)\coloneqq \sqrt{\frac{1}{4}\pm\frac{\lambda^2}{\kappa^2}}=\lambda^\pm_\lambda(\kappa)+\frac{1}{2}\in \mathbb C,
%\end{equation} 
%$x=\left(\frac{\sin(\kappa \theta)}{\kappa}\cos\psi,\frac{\sin(\kappa \theta)}{\kappa}\sin\psi,\frac{\cos(\kappa \theta)}{\kappa}\right)\in {\rm SC}_\kappa(R)$,
% with $\theta\in [0,R)$ and $\psi\in [0,2\pi)$, and for some $C_1,C_2\in \mathbb R\setminus \{0\}$, we consider the non-vanishing function 
For every $x=x(\theta,\xi)\in B_\kappa(r,R)$, let 
\begin{equation}\label{nodal-function}
w(x)\coloneqq\left(D_1\mathcal F(\gamma^+_\lambda(\kappa),\theta)  + D_2\mathcal H(\gamma^+_\lambda(\kappa),\theta) + D_3\mathcal F(\gamma^-_\lambda(\kappa),\theta)  +D_4\mathcal H(\gamma^-_\lambda(\kappa),\theta) \right)\sin \theta\sin\xi,
\end{equation}
where the constants
$\left\{D_i\right\}_{i=1}^{4}$ are not all zero and  $\frac{1}{2}\pm \gamma^\pm_\lambda(\kappa)\neq 0,-1,-2,\dots$; in the complementary cases we consider the limiting values.  
We notice that $w$  changes its sign on ${B}_\kappa(r,R)$ as $$w(x(\theta,\xi))=-w(x(\theta,\pi+\xi));$$ in fact, $w$ has two azimuthally  opposite nodal  circular arcs,  corresponding to the values $\xi=0$ and $\xi=\pi$, respectively.  
   
Using the hypergeometric differential equation (15.10.1) and (15.10.8) from  \cite{Digital}, relation \eqref{Laplace-operator-1} shows that $w$ from \eqref{nodal-function}  verifies pointwisely the first equation of \eqref{CP-problem}, i.e.,   $\Delta_{\kappa}^2 w=\lambda^4 w$  in $ B_\kappa(r,R).$
Moreover, the factorized form of the latter equation -- which is relevant only in the variable $\theta$ after simplifying by $\sin\xi$ -- is equivalent to 
\begin{equation}\label{factorized-nodal}
\frac{1}{\sin \theta}\frac{{\rm d}}{{\rm d} \theta}\left({\sin \theta}\frac{{\rm d} \widetilde w}{{\rm d} \theta}\right)-\left(\frac{\kappa^2}{\sin^2 \theta}\pm\lambda^2\right)\widetilde w=0, \ \ \theta\in (\sqrt{\kappa} r,\sqrt{\kappa} R).
\end{equation}
In fact, the solutions $\widetilde w$ (both for $\text{`}+\text{'}$ and $\text{`}-\text{'}$) correspond to the four expressions involving the functions $\mathcal F$ and $\mathcal H$ in \eqref{nodal-function}. 
For abbreviation, let 
$$\widetilde {\mathcal F}_{\gamma^\pm_\lambda(\kappa)}(\theta)\coloneqq\mathcal F(\gamma^\pm_\lambda(\kappa),\theta))\sin  \theta\ \ {\rm and}\ \ \widetilde {\mathcal H}_{\gamma^\pm_\lambda(\kappa)}(\theta)\coloneqq\mathcal H(\gamma^\pm_\lambda(\kappa),\theta))\sin  \theta.$$
The  clamped boundary conditions  $w=\dfrac{\partial w}{\partial \textbf{n}}=0 $ on $\partial {B}_\kappa(r,R)$ provide four equations involving the constants $\left\{D_i\right\}_{i=1}^{4}$;  since these constants are not all zero, we necessarily obtain that 
 \begin{equation}\label{determinant-sign-changing}
 {\rm det}\left[\def\arraystretch{1.1}\begin{matrix}
 \widetilde {\mathcal F}_{\gamma^+_\lambda(\kappa)}(\sqrt{\kappa} r) & \widetilde {\mathcal H}_{\gamma^+_\lambda(\kappa)}(\sqrt{\kappa} r) &\widetilde {\mathcal F}_{\gamma^-_\lambda(\kappa)}(\sqrt{\kappa} r) & \widetilde {\mathcal H}_{\gamma^-_\lambda(\kappa)}(\sqrt{\kappa} r)\\
 \widetilde {\mathcal F}'_{\gamma^+_\lambda(\kappa)}(\sqrt{\kappa} r) & \widetilde {\mathcal H}'_{\gamma^+_\lambda(\kappa)}(\sqrt{\kappa} r) &\widetilde {\mathcal F}'_{\gamma^-_\lambda(\kappa)}(\sqrt{\kappa} r) & \widetilde {\mathcal H}'_{\gamma^-_\lambda(\kappa)}(\sqrt{\kappa} r)\\
 \widetilde {\mathcal F}_{\gamma^+_\lambda(\kappa)}(\sqrt{\kappa} R) & \widetilde {\mathcal H}_{\gamma^+_\lambda(\kappa)}(\sqrt{\kappa} R) &\widetilde {\mathcal F}_{\gamma^-_\lambda(\kappa)}(\kappa R) & \widetilde {\mathcal H}_{\gamma^-_\lambda(\kappa)}(\sqrt{\kappa} R)\\
 \widetilde {\mathcal F}'_{\gamma^+_\lambda(\kappa)}(\sqrt{\kappa} R) & \widetilde {\mathcal H}'_{\gamma^+_\lambda(\kappa)}(\sqrt{\kappa} R) &\widetilde {\mathcal F}'_{\gamma^-_\lambda(\kappa)}(\sqrt{\kappa} R) & \widetilde {\mathcal H}'_{\gamma^-_\lambda(\kappa)}(\sqrt{\kappa} R)
 \end{matrix}\right]=0.
\def\arraystretch{1}
 \end{equation}
 Let  $\lambda^{SC}_{r,R}(\kappa)>0$ be the 
 smallest positive zero  of  \eqref{determinant-sign-changing} and   consider 
 \begin{equation}\label{Lambda-1}
 \lambda_1\coloneqq \lambda^{SC}_{r,R}=\lim_{\kappa\to 0}\lambda^{SC}_{r,R}(\kappa).
 \end{equation} 
% By letting  $\kappa\to 0$  in \eqref{factorized-nodal} with $\lambda\coloneqq \lambda^{SC}_{r,R}(\kappa)$ and using \eqref{Lambda-1}, it turns out that 
% $$\theta^{-1}\frac{{\rm d}}{{\rm d} \theta}\left(\theta\frac{{\rm d} \widetilde w}{{\rm d} \theta}\right)-\left(\theta^{-2}\pm\lambda_1^2\right)\widetilde w=0, \ \ \theta\in (r,R),$$
% which have exactly as solutions the Bessel functions $\theta\mapsto J_1(\lambda_1\theta)$, $I_1(\lambda_1\theta)$, $Y_1(\lambda_1\theta)$, 
% $K_1(\lambda_1\theta)$. 
Similarly as before, we have that $\lambda_1>0$ and, 
 by using Proposition \ref{hipergeometrikus-Bessel-0-1} (with $\mu=1$)  and \eqref{Lambda-1}, it follows that for every $x>0$ we have 
 \begin{equation}\label{tilde-F-limit}
 \lim_{\kappa\to 0}\frac{1}{\sqrt{\kappa}}\widetilde {\mathcal F}_{\gamma^\pm_{\lambda^{SC}_{r,R}(\kappa)}(\kappa)}(\sqrt{\kappa} x)=\frac{2}{\lambda_1}\left\{ \begin{array}{rll}
 J_1(\lambda_1 x)   &{\rm for}&  \text{`}+\text{'}; \\
 I_1(\lambda_1 x)    &{\rm for}&  \text{`}-\text{'}.
 \end{array}\right.
 \end{equation}
 Moreover,  relations \eqref{Digamma-converges} and \eqref{Y_n-forma} (with $n=1$) imply that for every $x>0$ one has the limits
  \begin{equation}\label{tilde-H-limit}
  	\def\arraystretch{1.3}
 \lim_{\kappa\to 0}\frac{1}{\sqrt{\kappa}}\widetilde {\mathcal H}_{\gamma^\pm_{\lambda^{SC}_{r,R}(\kappa)}(\kappa)}(\sqrt{\kappa} x)=\left\{ \begin{array}{rll}
 \frac{2\pi}{\lambda_1}Y_1(\lambda_1 x)
 %-\frac{4}{\lambda_1}\ln(\lambda_1)J_1(\lambda_1 x)
    &{\rm for}&  \text{`}+\text{'}; 
 \\
 \frac{4}{\lambda_1}K_1(\lambda_1 x)
% -\frac{4}{\lambda_1}\ln(\lambda_1)I_1(\lambda_1 x) 
    &{\rm for}&  \text{`}-\text{'}.
 \end{array}\right.
\def\arraystretch{1}
 \end{equation}
 Thus, by taking the limit $\kappa\to 0$ in   
  \eqref{determinant-sign-changing} for $\lambda\coloneqq \lambda^{SC}_{r,R}(\kappa)$ and  by using basic properties of determinants, we obtain
  \begin{equation}\label{determinant-sign-changing-1}
 {\rm det}\left[\begin{matrix}
 J_1(\lambda_1r) & Y_1(\lambda_1r) &I_1(\lambda_1r) & K_1(\lambda_1r)\\
 J_1'(\lambda_1r) & Y_1'(\lambda_1r) &I'_1(\lambda_1r) & K'_1(\lambda_1r)\\
 J_1(\lambda_1R) & Y_1(\lambda_1R) &I_1(\lambda_1R) & K_1(\lambda_1R)\\
 J_1'(\lambda_1R) &Y_1'(\lambda_1R) &I_1'(\lambda_1R) & K_1'(\lambda_1R)
 \end{matrix}\right]=0.
 \end{equation}

%$$\left\{ \begin{array}{lll}
%C_1\mathcal F_\kappa(\lambda_+^\kappa(\Gamma),R)  +C_2\mathcal F_\kappa(\lambda_-^\kappa(\Gamma),R)=0; \\
%\\
%C_1\frac{\partial}{\partial \theta}(\sin(\kappa \theta)\mathcal F_\kappa(\lambda_+^\kappa(\Gamma),\theta))|_{\theta=R}  +C_2\frac{\partial}{\partial \theta}(\sin(\kappa \theta)\mathcal F_\kappa(\lambda_-^\kappa(\Gamma),\theta))|_{\theta=R}=0.
%\end{array}\right.$$
%In order to have non-zero values for $C_1$ and $C_2$, we necessarily have to guarantee that the determinant of the latter system vanishes, which is equivalent to 
%\begin{equation}\label{determinant=0}
%{\rm det}\left[\begin{matrix}
%\mathcal F_\kappa(\lambda_+^\kappa(\Gamma),R) & \mathcal F_\kappa(\lambda_-^\kappa(\Gamma),R) &\\
%\frac{\partial}{\partial \theta}\mathcal F_\kappa(\lambda_+^\kappa(\Gamma),\theta)|_{\theta=R} & \frac{\partial}{\partial \theta}\mathcal F_\kappa(\lambda_-^\kappa(\Gamma),\theta)|_{\theta=R}
%\end{matrix}\right]=0.
%\end{equation}
%This is also equivalent to saying that the function $\theta\mapsto \frac{\mathcal F(\lambda_-^\kappa(\Gamma),\theta)}{\mathcal F(\lambda_+^\kappa(\Gamma),\theta)}$ has a   critical point at $\theta=R.$ 
%On account of \eqref{lambda-Gamma-punctured}, let $\Gamma\coloneqq \Gamma_N(R,\kappa)>0$ be the smallest value verifying the latter equation. Since $C_2=-C_1\frac{\mathcal F_\kappa(\lambda_-^\kappa(\Gamma_N(R,\kappa)),R)}{\mathcal F_\kappa(\lambda_+^\kappa(\Gamma_N(R,\kappa)),R)}$, the above construction shows that the $\theta$-depending part of $w$ is sign-preserving on $[0,R)$.   

\subsection{Threshold spherical belts} Let $0<r<R.$ Using \eqref{Lambda-0} and \eqref{Lambda-1}, we recall that 
 \begin{equation}\label{ket-hatarertek}
\lambda_0\coloneqq\lambda^{SP}_{r,R}=\lim_{\kappa\to 0}\lambda^{SP}_{r,R}(\kappa)\ \  {\rm and }\ \ 
\lambda_1\coloneqq\lambda^{SC}_{r,R}=\lim_{\kappa\to 0}\lambda^{SC}_{r,R}(\kappa),
\end{equation} 
 are the smallest positive solutions to  equations \eqref{determinant-sign-preserving-1} and  \eqref{determinant-sign-changing-1}, respectively. The unit exterior radius case of the
  main result of Coffman, Duffin and Schaffer \cite[\S 4]{CDS} on planar annuli asserts (up to a scaling) that 
\begin{itemize}
	\item[(i)] $\lambda_0<\lambda_1$ whenever ${R}/{r}<{\sf C}_{CDS}$; and
	\item[(ii)]  $\lambda_0>\lambda_1$ whenever ${R}/{r}>{\sf C}_{CDS}$,
\end{itemize}
where ${\sf C}_{CDS}$ is the critical Coffman--Duffin--Schaffer constant. In fact, the value of the  constant ${\sf C}_{CDS}$ is determined whenever $\lambda^{SP}_{r,R}$ and $\lambda^{SC}_{r,R}$ coincide, which is based on certain properties of Bessel functions, see also Coffman and Duffin \cite{CD}.  In this critical case, one has that $\lambda^{SP}_{r,R}=\lambda^{SC}_{r,R}=\frac{\lambda^1_c}{r}=\frac{\lambda^2_c}{R}$ with $\lambda^1_c\approx 0.00062557144$ and $\lambda^2_c\approx 4.769102418$. Accordingly, $${\sf C}_{CDS}=\frac{R}{r}= \frac{\lambda^2_c}{\lambda^1_c}\approx762.3264.$$

It remains to combine the latter relations with the limits \eqref{ket-hatarertek} in order to conclude the existence of $\kappa_0\in (0,\pi/R)$ such that for every $\kappa\in (0,\kappa_0)$:
\begin{itemize}
	\item[(i)] $\lambda^{SP}_{r,R}(\kappa)<\lambda^{SC}_{r,R}(\kappa)$ whenever ${R}/{r}<{\sf C}_{CDS}$, i.e., the first eigenfunction in \eqref{CP-problem} is of fixed sign  (see Figure \ref{1-abra}/(a)); and  
	\item[(ii)] $\lambda^{SP}_{r,R}(\kappa)>\lambda^{SC}_{r,R}(\kappa)$ whenever ${R}/{r}>{\sf C}_{CDS}$, i.e., the first eigenfunction in \eqref{CP-problem} is sign-changing having two azimuthally  opposite nodal circular arcs (see Figure \ref{1-abra}/(b)). 
\end{itemize}
 The proof is concluded.  \hfill $\square$

%\noindent The properties from Theorem \ref{theorem-belt} are illustrated on Figure \ref{1-abra}:
	
		\begin{figure}[t!]
		\centering
		\includegraphics[scale=1]{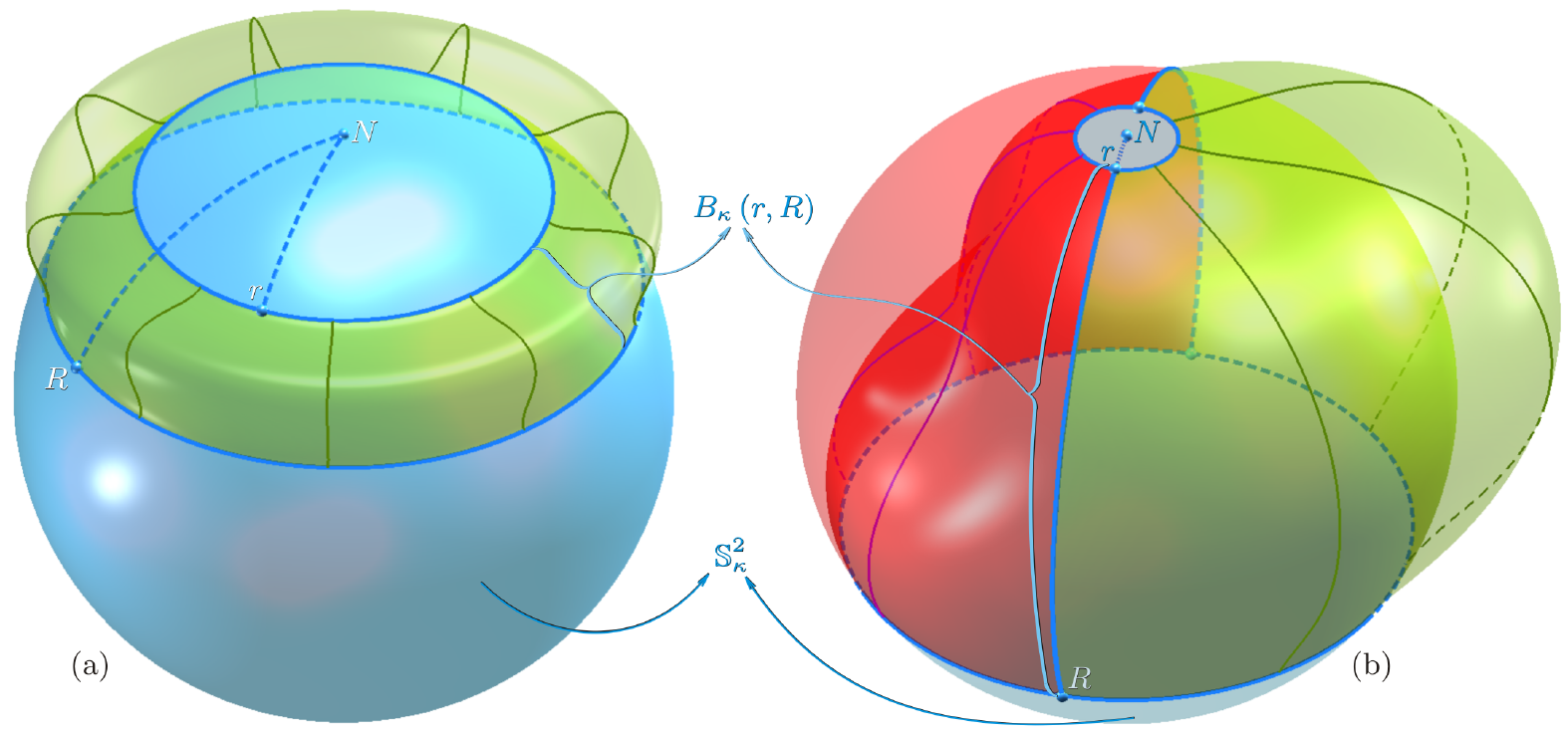}
		\caption{
%			Clamped spherical belts $B_\kappa(r,R)\subset \mathbb S_\kappa^2$ for:
%			 (a) narrow relative width having a sign-preserving first eigenfunction (light green); 
%				(b) wide relative width to whom the first eigenfunction is sign-changing, with  the positive (light green) and negative (dark red) values of the eigenfunction -- the zero altitude being represented by the sphere $\mathbb S_\kappa^2$ itself, -- 
%%				 being outside and inside of $\mathbb S_\kappa^2$, respectively, 
%				divided by   two  azimuthally  opposite nodal circular arcs (dark blue).
%			 When $\kappa\to 0$, the relative threshold width $R/r$ is given by the critical Coffman--Duffin--Schaffer constant  ${\sf C}_{CDS}\approx{762.3264}$. 
			%  PROBA: $\mathfrak  g_{{\nu},1}(\widetilde L_0)$ 
			Clamped spherical belts $B_\kappa(r,R)\subset \mathbb S_\kappa^2$ of (a) narrow and (b) wide relative widths, in the case of which a first eigenfunction is of fixed sign and sign-changing, respectively. In both cases the zero altitude is represented by the light blue sphere $\mathbb S_\kappa^2$. The positive eigenfunction (a) is rendered with a transparent light green material, while the positive and negative values of the sign-changing eigenfunction (b) are illustrated by light green and dark red transparent materials, respectively. In the latter case, the eigenfunction is divided by two dark blue azimuthally opposite nodal circular arcs, while the preimages of the positive and negative parts are rendered with transparent dark green and light red materials, respectively. When $\kappa \to 0$,  the relative threshold width $R/r$ is given by the critical Coffmann--Duffin--Schaffer constant ${\sf C}_{CDS}\approx{762.3264}$. 
%			(For interpretation of the references to color in this figure legend, the reader is referred to the web version of this article.)
		}\label{1-abra}
	\end{figure}

\begin{remark}\rm  	Let  $\kappa>0$, $R<\pi/\sqrt{\kappa}$ and consider the 2-dimensional spherical cap 
	$$C^2_\kappa(R)=\{x\in \mathbb S_\kappa^2:d_\kappa(N,x)<R\}.$$ Whenever $r\to 0$ (i.e., the interior radius of the spherical belt $B_\kappa(r,R)$ shrinks to the pole $N\in \mathbb S_\kappa^2$), the limit case of Theorem \ref{theorem-belt}/(ii) states that for every $0<R<{\pi}/{\sqrt{\kappa}}$, any first eigenfunction of the clamped problem \eqref{CP-problem} on the punctured spherical cap ${C}^2_\kappa(R)\setminus \{N\}\subset \mathbb S_\kappa^2$ is always sign-changing; this scenario clearly holds for any value of $\kappa$.\ In this particular case  of the punctured clamped spherical cap  ${C}^2_\kappa(R)\setminus \{N\}$  the boundary condition for $u$ reads as $$u=\frac{\partial u}{\partial \textbf{n}}=0\ \ {\rm on}\ \partial {C}^2_\kappa(R)\ \ {\rm and}\ \ u(N)=0,$$ 
	i.e., the tangential condition at the North pole $N$ vanishes. 
%	 In this particular case no restriction is needed for the value of $\kappa$, i.e., the first eigenfunction of \eqref{CP-problem} on the punctured spherical cap ${C}^2_\kappa(R)\setminus \{N\}\subset \mathbb S_\kappa^2$ is always sign-changing this scenario may occur even for punctured spherical caps having large curvature. 	
\end{remark}

\section{Reduction of Lord Rayleigh's conjecture on positively curved spaces}\label{section-4}

Let $n\in \mathbb N_{\geq 2}$ be fixed. In this section, we assume that $(M,g)$ is an  $n$-dimensional  compact Riemannian manifold  with Ricci curvature ${\sf Ric}_{(M,g)}\geq (n-1)\kappa>0$.  Rescaling the metric, we could consider ${\sf Ric}_{(M,g)}\geq n-1$ (and the unit sphere $\mathbb S^n\coloneqq\mathbb S_1^n)$, but as we have seen in \S \ref{section-3} the presence of $\kappa>0$ was crucial, thus for the sake of coherency we preserve the general form of the initial metric.   

%we will keep this general assumption on the curvature. 
%
% however, \S \ref{section-3} -- where the presence of $\kappa>0$ was crucial -- entitles us  to keep this consistency throughout the whole paper.   

	\subsection{Ashbaugh--Benguria--Nadirashvili--Talenti nodal-decomposition}\label{subsection-ABNT}
	
	%We are going to use certain symmetrization arguments \`a la Schwarz; 
	%
	%, starting with functions defined on $\Omega$ and transposing their level-structures to the space-form $N_\kappa^n$. 
If $S\subset M$ is a measurable set, its \textit{normalized rearrangement} $S^\star\subset\mathbb S_\kappa^n$ is an open spherical cap  centered at the North pole and $$\frac{V_{\kappa}(S^\star)}{V_{\kappa}(\mathbb S_\kappa^n)}=\frac{V_g(S)}{V_g(M)}.$$ In a similar way, if $U:S\to  (0,\infty)$ is a measurable function,  we introduce its \textit{normalized  rearrangement}   $U^\star:S^\star\to (0,\infty)$ such that
	\begin{equation}\label{U-definicio}
	\frac{V_{\kappa}(\{x\in S^\star:U^\star(x)>t \})}{V_{\kappa}(\mathbb S_\kappa^n)}=\frac{V_g(\{x\in S:U(x)>t \})}{V_g(M)},\ \forall t>0.
	\end{equation}
	By construction, $U^\star$ is a spherical cap symmetric function, i.e., $\xi\mapsto U^\star(x(\theta,\xi))$ is constant and
	 \begin{equation}\label{u-plusz-minusz}
	 U^\star(x)=\sup\{t>0:x\in \{U>t\}^\star\}.
	 \end{equation}
%The following result is needed in our arguments.  
	
	\begin{lemma}\label{lemma-1} Let $U:S\to  (0,\infty)$ be an integrable  function and $U^\star:S^\star\to (0,\infty)$ be its normalized  rearrangement. If $W\subseteq S$ is measurable and $W^*\subset \mathbb S_\kappa^n$ is its normalized  rearrangement, then  %$\frac{V_{\kappa}(W^\star)}{V_{\kappa}(\mathbb S_\kappa^n)}=\frac{V_g(W)}{V_g(M)},$ then 
		\begin{equation}\label{Riesz}
		\displaystyle\frac{\displaystyle\int_W U{\rm d}v_g}{V_g(M)} \leq \frac{\displaystyle\int_{W^\star}U^\star{\rm d}v_{\kappa}}{V_{\kappa}(\mathbb S_\kappa^n)}.
		\end{equation}
		In addition, if $W=S$, then we have equality in \eqref{Riesz}. 
	\end{lemma}

\begin{proof}
	Relation \eqref{Riesz} can be viewed as a normalized Hardy--Littlewood--P\'olya  inequality; although it can be verified by standard techniques, we provide its proof for completeness. 
%	To do this, we first observe that having $W_1,W_2\subseteq S$ and their relative rearrangements $W_1^\star,W_2^\star\subseteq \mathbb S_\kappa^n$, then 
%	\begin{equation}\label{rearrang}
%	\frac{V_g(W_1\cap W_2)}{V_g(M)}\leq \frac{V_{\kappa}(W_1^\star\cap W_2^\star)}{V_{\kappa}(\mathbb S_\kappa^n)}.
%	\end{equation}
%	 Indeed, if $V_g(W_1)\leq V_g(W_2)$, then $V_{\kappa}(W_1^\star)\leq V_{\kappa}(W_2^\star)$, thus $W_1^\star\subset  W_2^\star$. Consequently, $$V_g(W_1\cap W_2)\leq V_g(W_1)=\frac{V_g(M)}{V_{\kappa}(\mathbb S_\kappa^n)}V_{\kappa}(W_1^\star)\leq \frac{V_g(M)}{V_{\kappa}(\mathbb S_\kappa^n)}V_{\kappa}(W_1^\star\cap W_2^\star).$$
	 Using the layer cake representation and the fact that $\chi_{W^\star}=\chi_W^\star,$ where $\chi_Q$ denotes the indicator function of any non-empty set  $Q$, it turns out by \eqref{U-definicio} that 
	 \begin{align*}
	 \int_W U{\rm d}v_g&=\int_S \chi_W(x)U(x){\rm d}v_g(x)=\int_S \int_0^\infty\int_0^\infty \chi_{\{\chi_W>t\}}(x)\chi_{\{U>s\}}(x){\rm d}t{\rm d}s{\rm d}v_g(x)\\&=\int_0^\infty\int_0^\infty\int_S  \chi_{\{\chi_W>t\}\cap \{U>s\}}(x){\rm d}v_g(x){\rm d}t{\rm d}s=\int_0^\infty\int_0^\infty V_g({\{\chi_W>t\}\cap \{U>s\}}){\rm d}t{\rm d}s\\&\leq  \int_0^\infty\int_0^\infty \min\left\{V_g({\{\chi_W>t\}),V_g(\{U>s\}})\right\}{\rm d}t{\rm d}s\\&= 	 
	 \frac{V_g(M)}{V_{\kappa}(\mathbb S_\kappa^n)}\int_0^\infty\int_0^\infty \min\left\{V_{\kappa}({\{x\in S^\star:\chi_W^\star(x)>t\}),V_{\kappa}(\{x\in S^\star:U^\star(x)>s\}})\right\}{\rm d}t{\rm d}s %\ \ \ \ \ \ \ \ \ \ \ \ \  {\rm (see \ \eqref{rearrang})}
	 \\&=  
	 \frac{V_g(M)}{V_{\kappa}(\mathbb S_\kappa^n)}\int_0^\infty\int_0^\infty V_{\kappa}(\{x\in S^\star:\chi_W^\star(x)>t\} \cap \{x\in S^\star:U^\star(x)>s\}){\rm d}t{\rm d}s\\&=\frac{V_g(M)}{V_{\kappa}(\mathbb S_\kappa^n)}\int_0^\infty\int_0^\infty\int_{S^\star}  \chi_{\{\chi_W^\star>t\}\cap \{U^\star>s\}}(x){\rm d}v_{\kappa}(x){\rm d}t{\rm d}s=\frac{V_g(M)}{V_{\kappa}(\mathbb S_\kappa^n)}\int_{S^\star} \chi_{W^\star}(x)U^\star(x){\rm d}v_{\kappa}(x)\\&=\frac{V_g(M)}{V_{\kappa}(\mathbb S_\kappa^n)}\int_{W^\star} U^\star{\rm d}v_{\kappa}.
	 \end{align*}
 If $W=S$, equality clearly holds in the above relations, since  the set $\{\chi_W>t\}=\{\chi_S>t\}$  is either empty or coincides with the whole $S$ for every $t>0$. 
% since $V_g({\{\chi_S(x)>t\}\cap \{U(x)>s\}})=V_g( \{U(x)>s\})=\frac{V_g(M)}{V_{\kappa}(\mathbb S_\kappa^n)}V_{\kappa}( \{U^\star(x)>s\})=\frac{V_g(M)}{V_{\kappa}(\mathbb S_\kappa^n)}V_{\kappa}(\{\chi_S^\star(x)>t\}\cap \{U^\star(x)>s\}).$
\end{proof}

Let $\Omega\subset M$ be a fixed open set.	
Since ${\sf Ric}_{(M,g)}\geq (n-1)\kappa>0$, by the Bochner--Lichnerowicz--Weitzenb\"ock formula it follows that the Sobolev space  $H_0^{2}(\Omega)=W_0^{2,2}(\Omega)$ is a proper choice for problem \eqref{CP-problem-0}, see Hebey \cite[Proposition 3.3]{Hebey}.  Furthermore, on account of the compactness of $(M,g)$ (by the Bonnet--Myers theorem) and basic properties of  $W_0^{2,2}(\Omega)$, a similar argument as in Ashbaugh and Benguria \cite[Appendix 2]{A-B} shows that the infimum in (\ref{variational-charact}), i.e., the fundamental tone  	$$	\Lambda_g(\Omega)=\inf_{u\in W_0^{2,2}(\Omega)\setminus \{0\}}\frac{\displaystyle \int_{\Omega}(\Delta_g u)^2 {\rm d}v_g}{\displaystyle \int_{\Omega}u^2 {\rm d}v_g},$$
is achieved. Let $u\in W_0^{2,2}(\Omega)$ be such a minimizer; elliptic regularity results guarantee that $u\in \mathcal C^\infty(\Omega)$. 

 Since $u$ may change its sign in $\Omega$ (see e.g.\ Theorem \ref{theorem-belt}),  denote by $u_+\coloneqq\max(u,0)$ and $u_-\coloneqq-\min(u,0)$  the positive and negative parts of $u$, respectively, and consider also the corresponding preimages $$\Omega_+\coloneqq\{x\in \Omega: u_+(x)>0\}\ \ {\rm  and}\ \  \Omega_-\coloneqq\{x\in \Omega: u_-(x)>0\}.$$ 
In the rest of this section, we assume that 
$$V_g(\Omega_+)V_g(\Omega_-)>0,$$
otherwise the minimizer $u$ is of fixed sign and the subsequent argument becomes considerable simpler.
For further use, we also introduce the sets 
$$\Omega_+^\Delta\coloneqq\{x\in \Omega: (\Delta_gu)_+(x)>0\}\ \ {\rm  and}\ \  \Omega_-^\Delta\coloneqq\{x\in \Omega: (\Delta_gu)_-(x)>0\}.$$
Let $\Omega_\pm^\star\subset \mathbb S_\kappa^n$ and $\big(\Omega_\pm^\Delta\big)^\star\subset \mathbb S_\kappa^n$ be the normalized rearrangements of $\Omega_\pm\subset M$ and $\Omega_\pm^\Delta \subset M, $ respectively.  
	Since the subset of $\Omega$ where $u$ is either constant or harmonic is negligible,  
	%$\mathcal H^n(\Delta_g^{-1}(0))=0$, where $\mathcal H^n$ denotes the $n$-dimensional Hausdorff measure, 
	one has that   
\begin{equation}\label{terfogatok}
V_\kappa\big(\Omega_+^\star\big)+V_\kappa\big(\Omega_-^\star\big)=V_\kappa\left((\Omega_+^\Delta)^\star\right)+V_\kappa\left((\Omega_-^\Delta)^\star\right)=V_\kappa(\Omega^\star).
\end{equation}	
%	=V_\kappa(C_\kappa^n(L))
	%for some $L>0.$
	% $u\in W_0^{2,2}(\Omega)$ the minimizer in (\ref{variational-charact}) and its positive and negative parts $u_+$ and $u_-$, respectively, 
	Let  $u_\pm^\star:\Omega_\pm^\star\to (0,\infty)$ be the normalized rearrangement of $u_\pm:\Omega_\pm\to (0,\infty)$, i.e.,  for every $t>0,$ 
	\begin{equation}\label{szimmetrizacio-U}
	\frac{V_\kappa\big(\{x\in \Omega_+^\star:u_+^\star(x)>t \}\big)}{V_{\kappa}\big(\mathbb S_\kappa^n\big)}=\frac{V_g\big(\{x\in\Omega_+:u_+(x)>t \}\big)}{V_g(M)}=:j(t),
	\end{equation}
and
	\begin{equation}\label{szimmetrizacio-U-0}
	\frac{V_\kappa\big(\{x\in \Omega_-^\star:u_-^\star(x)>t \}\big)}{V_{\kappa}\big(\mathbb S_\kappa^n\big)}=\frac{V_g\big(\{x\in\Omega_-:u_-(x)>t \}\big)}{V_g(M)}=:h(t),
	\end{equation}
see Figure \ref{2-abra}. 
	\begin{figure}[h!]
	\centering
	\includegraphics[scale=1]{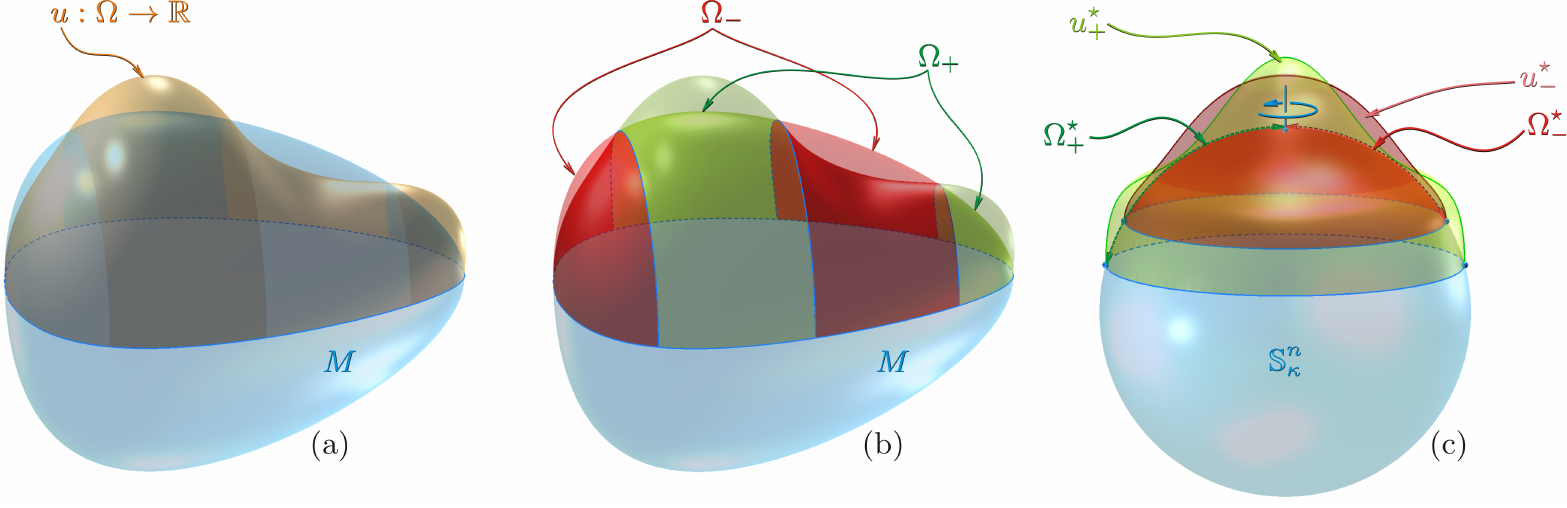}
	\caption{Illustration of the Ashbaugh--Benguria--Nadirashvili--Talenti nodal-decomposition argument.  
		(a) A first eigenfunction $u:\Omega\to \mathbb R$ of the clamped plate problem  \eqref{CP-problem-0} having nodal domains (the zero altitude is represented by the transparent light blue manifold $M$). (b) Rendering the segments of the null-measure set $u^{-1}(0)$ with dark blue lines, the image shows the nodal-decomposition of the open set $\Omega\subset M$ into the transparent light red and dark green preimages   $\Omega_-=\{u<0\}$ and  $\Omega_+=\{u>0\}$, respectively.  (c) The normalized rearrangements  $u_\pm^\star: \Omega_\pm^\star \to (0,\infty)$ of  $u_\pm:\Omega_\pm\to (0,\infty)$  on the spherical caps $\Omega_\pm^\star\subset \mathbb S^n_\kappa$, cf.\ relations \eqref{szimmetrizacio-U} and \eqref{szimmetrizacio-U-0}.
%		Nodal-decomposition of  $\Omega\subset M$ to $\Omega_-=\{u<0\}$ and $\Omega_+=\{u>0\}$ w.r.t. the first eigenfunction $u:\Omega\to \mathbb R$ of the clamped plate problem  \eqref{CP-problem-0} -- having the null measure set $u^{-1}(0)$, -- resulting the normalized rearrangements  $u_\pm^\star: \Omega_\pm^\star \to (0,\infty)$ of  $u_\pm:\Omega_\pm\to (0,\infty)$  on the spherical caps $\Omega_\pm^\star\subset \mathbb S^n_\kappa$, cf. \eqref{szimmetrizacio-U} and \eqref{szimmetrizacio-U-0}.
	}\label{2-abra}
\end{figure}

%	The functions $u_+^\star$ and $u_-^\star$ are well-defined and radially symmetric, verifying the property that  for some $r_t>0$ and $\rho_t>0$ one has 
%	\begin{equation}\label{alfa-r-t}
%	\{x\in N_\kappa^n:u_+^\star(x)>t \}=B_\kappa(r_t)\ \ {\rm and}\ \ \{x\in N_\kappa^n:u_-^\star(x)>t \}=B_\kappa(\rho_t),\end{equation}
%	with $V_\kappa(r_t)=\alpha(t)$ and $V_\kappa(\rho_t)=\beta(t)$, respectively.  

	% Namely, given a measurable function $U:\Omega\to [0,\infty)$, we define a function $U^\star:N_\kappa^n\to [0,\infty)$ such that for every $t>0$ one has 

Let $T_u^\pm\coloneqq \sup_{x\in \Omega_\pm}u_\pm(x)\geq 0;$ clearly, by definition, one has that $j(t)=0$ for every $t\geq T_u^+$ and  $h(t)=0$ for every $t\geq T_u^-$.

In the same way as in \eqref{szimmetrizacio-U} and \eqref{szimmetrizacio-U-0}, we  introduce the normalized rearrangements $(\Delta_gu)_\pm^\star$ of $(\Delta_gu)_\pm:\Omega_\pm^\Delta\to (0,\infty)$.  We extend $u_\pm^\star:\Omega_\pm^\star\to (0,\infty)$ and $(\Delta_gu)_\pm^\star:(\Omega_\pm^\Delta)^\star\to (0,\infty)$ by zero  to the whole $\Omega^\star$ outside of  $\Omega_\pm^\star$ and $\big(\Omega_\pm^\Delta\big)^\star$, respectively. 	
Our further analysis is based on fine properties of the functions $$\mathcal J(s)\coloneqq(\Delta_g u)^*_-(s)-(\Delta_g u)^*_+(V_\kappa(\Omega^\star)-s)$$ {\rm and}\ \ $$\mathcal H(s)\coloneqq-\mathcal J(V_\kappa(\Omega^\star)-s),\ \ s\in [0,V_\kappa(\Omega^\star)],$$	
	where we will use the notation 
\begin{equation}\label{J-J-notation}
(\Delta_g u)^*_\pm(s)\coloneqq (\Delta_g u)^\star_\pm(x),\ \ {\rm whenever}\ \  s=V_\kappa(C_\kappa^n(d_\kappa(N,x))),\ x\in \Omega^\star.
\end{equation}	
Some useful properties of the functions $\mathcal J$ and $\mathcal H$ are summarized in the sequel. 

\begin{lemma}\label{basic=J-H}
For every $ \sigma\in [0,V_\kappa(\Omega^\star)],$ one has that$:$ 
\begin{itemize}
	\item[(i)]  $\ds\int_0^\sigma \mathcal J(s){\rm d}s\geq \ds\int_0^{V_\kappa(\Omega^\star)} \mathcal J(s){\rm d}s=0;$
	\item[(ii)] $ \ds \int_0^\sigma \mathcal H(s){\rm d}s\geq \ds\int_0^{V_\kappa(\Omega^\star)} \mathcal H(s){\rm d}s=0;$ 
	\item [(iii)] $\ds\int_{\Omega_+^\star} \mathcal J(V_\kappa(C_\kappa^n(d_\kappa(N,x)))){\rm d}v_\kappa(x)= \int_{\Omega_-^\star} \mathcal H(V_\kappa(C_\kappa^n(d_\kappa(N,x)))){\rm d}v_\kappa(x).$
\end{itemize}
\end{lemma}
\noindent {\it Proof.} By the divergence theorem and the boundary condition 	$\dfrac{\partial u}{\partial \textbf{n}}=0 $ on $\partial \Omega$, it turns out that $\ds\int_\Omega \Delta_gu{\rm d}v_g=0$. Therefore, by applying Lemma \ref{lemma-1} (with  $S=W\coloneqq \Omega_\pm^\Delta$ and $U\coloneqq (\Delta_g u)_\pm$) and by using the change of variables \eqref{change-variables}, it turns out that
\begin{align*}
	0&=\frac{V_{\kappa}(\mathbb S_\kappa^n)}{V_g(M)}\int_\Omega \Delta_gu{\rm d}v_g=-\frac{V_{\kappa}(\mathbb S_\kappa^n)}{V_g(M)}\left(\int_{\Omega_-^\Delta} (\Delta_gu)_-{\rm d}v_g-\int_{\Omega_+^\Delta} (\Delta_gu)_+{\rm d}v_g\right)\\&=
	\int_{(\Omega_-^\Delta)^\star} (\Delta_gu)_-^\star{\rm d}v_\kappa-\int_{(\Omega_+^\Delta)^\star} (\Delta_gu)_+^\star{\rm d}v_\kappa=\int_{\Omega^\star} (\Delta_gu)_-^\star(x){\rm d}v_\kappa(x)-\int_{\Omega^\star} (\Delta_gu)_+^\star(x){\rm d}v_\kappa(x)\\&=\int_{\Omega^\star} (\Delta_gu)_-^*(V_\kappa(C_\kappa^n(d_\kappa(N,x)))){\rm d}v_\kappa(x)-\int_{\Omega^\star} (\Delta_gu)_+^*(V_\kappa(C_\kappa^n(d_\kappa(N,x)))){\rm d}v_\kappa(x)\\&=
	\int_0^{V_\kappa(\Omega^\star)}\left((\Delta_gu)_-^*(s)-(\Delta_gu)_+^*(s)\right){\rm d}s=\int_0^{V_\kappa(\Omega^\star)}\left((\Delta_gu)_-^*(s)-(\Delta_gu)_+^*(V_\kappa(\Omega^\star)-s)\right){\rm d}s\\&=\int_0^{V_\kappa(\Omega^\star)}\mathcal J(s){\rm d}s.	
\end{align*}
In a similar way, one obtains that $\ds \int_0^{V_\kappa(\Omega^\star)}\mathcal H(s){\rm d}s=0$. Now, 
by the definition of the normalized rearrangement, the functions $s\mapsto \mathcal J(s)$ and $s\mapsto \mathcal H(s)$ are decreasing on $[0,V_\kappa(\Omega^\star)];$ therefore 
\begin{equation}\label{basic-2-j-h}
\int_0^{\sigma} \mathcal J(s){\rm d}s\geq 0\ {\rm and}\ \int_0^\sigma \mathcal H(s){\rm d}s\geq 0,\ \ \sigma\in [0,V_\kappa(\Omega^\star)],
\end{equation}
which conclude the proof of (i) and (ii).

(iii) By the first part of the proof, relations \eqref{terfogatok} and  \eqref{change-variables} we have that   
\begin{align*}
0&=\int_0^{V_\kappa(\Omega^\star)}\mathcal J(s){\rm d}s=\int_0^{V_\kappa(\Omega^\star_+)}\mathcal J(s){\rm d}s-\int_0^{V_\kappa(\Omega^\star_-)}\mathcal H(s){\rm d}s\\&=\ds\int_{\Omega_+^\star} \mathcal J(V_\kappa(C_\kappa^n(d_\kappa(N,x)))){\rm d}v_\kappa(x)- \int_{\Omega_-^\star} \mathcal H(V_\kappa(C_\kappa^n(d_\kappa(N,x)))){\rm d}v_\kappa(x),
\end{align*}
 which completes the proof. 
\hfill $\square$	\\

Besides Lemma \ref{basic=J-H}, we need more precise estimates for the quantities  in \eqref{basic-2-j-h} as follows.

%	where $\cdot^*$ stands for the notation 
	%$h:\Omega^\star\to [0,\infty)$ being any function. 
	
	%\subsection{Talenti-type comparison on curved spaces}

	\begin{proposition} \label{f-hasonlitas}
		The following statements hold$:$ 
		\begin{itemize}
			\item[(i)] $\ds\displaystyle\int_0^{V_{\kappa}(\mathbb S_\kappa^n)j(t)} \mathcal J(s){\rm d}s\geq -\frac{V_{\kappa}(\mathbb S_\kappa^n)}{V_g(M)}\int_{\{u>t \}}\Delta_g u(x){\rm d}v_g(x)$ for every $t\in [0,T_u^+];$
			\item[(ii)] $\ds\displaystyle\int_0^{V_{\kappa}(\mathbb S_\kappa^n)h(t)} \mathcal H(s){\rm d}s\geq -\frac{V_{\kappa}(\mathbb S_\kappa^n)}{V_g(M)}\int_{\{u<-t \}}\Delta_g u(x){\rm d}v_g(x)$ for every $t\in [0,T_u^-];$
			\item[(iii)] either $(\Delta_g u)^*_-(s)=0$ or $ (\Delta_g u)^*_+(V_\kappa(\Omega^\star)-s) = 0$ for every $s\in [0,V_\kappa(\Omega^\star)]$.
		\end{itemize}
		%	\begin{equation}\label{kell-becsles}
		%	,
		%	\end{equation}
		%	and
		%		\begin{equation}\label{kell-becsles-22}
		%	
		%	\end{equation}
	\end{proposition}

\noindent 	{\it Proof.} (i)  Let $t\in [0,T_u^+]$ be fixed; due to \eqref{szimmetrizacio-U}, one can find a unique number  $a_t\geq 0$ such that $V_\kappa(C_\kappa^n(a_t))=V_{\kappa}(\mathbb S_\kappa^n)j(t).$ Moreover, since $\{u>t \}^\star=\{u_+>t \}^\star=C_\kappa^n(a_t)$, it follows that
	$\left(\{u>t \}\cap \Omega_-^\Delta\right)^\star\subseteq C_\kappa^n(a_t)\cap (\Omega_-^\Delta)^\star.$
	Using the latter relation together with the change of variables \eqref{change-variables} and Lemma \ref{lemma-1} (with $S\coloneqq \Omega_-^\Delta$, $W\coloneqq \{u>t \}\cap \Omega_-^\Delta$ and $U\coloneqq (\Delta_g u)_-$), we have that
	\begin{align}\label{I-becsles}
	I&\coloneqq  \nonumber	\displaystyle\int_0^{V_{\kappa}(\mathbb S_\kappa^n)j(t)} (\Delta_g u)^*_-(s){\rm d}s=\displaystyle\int_0^{V_\kappa(C_\kappa^n(a_t))} (\Delta_g u)^*_-(s){\rm d}s=\int_{C_\kappa^n(a_t)}(\Delta_g u)^*_-(V_\kappa(C_\kappa^n(d_\kappa(N,x))){\rm d} v_\kappa(x)\\&= \nonumber \int_{C_\kappa^n(a_t)}(\Delta_g u)^\star_-(x){\rm d} v_\kappa(x)=\int_{C_\kappa^n(a_t)\cap (\Omega_-^\Delta)^\star}(\Delta_g u)^\star_-(x){\rm d} v_\kappa(x)
	\\&\geq   \int_{\left(\{u>t \}\cap \Omega_-^\Delta\right)^\star}(\Delta_g u)^\star_-(x){\rm d} v_\kappa(x)\nonumber	\\&\geq \nonumber \frac{V_{\kappa}(\mathbb S_\kappa^n)}{V_g(M)}\int_{\{u>t \}\cap \Omega_-^\Delta}(\Delta_g u)_-(x){\rm d}v_g(x)
	\\&=  \frac{V_{\kappa}(\mathbb S_\kappa^n)}{V_g(M)}\int_{\{u>t \}}(\Delta_g u)_-(x){\rm d}v_g(x).
	\end{align}
 	Let $\widetilde a_t>0$ be the unique value such that $V_\kappa(C_\kappa^n(\widetilde a_t))=V_\kappa(\Omega^\star)-V_{\kappa}(\mathbb S_\kappa^n)j(t)$.
In a similar manner as above, by a change of variables and \eqref{change-variables},  we infer that 
		\begin{align}\label{II-becsles-0}
\nonumber 	II&\coloneqq 	\displaystyle \int_0^{V_{\kappa}(\mathbb S_\kappa^n)j(t)} (\Delta_g u)^*_+(V_\kappa(\Omega^\star)-s){\rm d}s\\&= \nonumber\int_0^{V_\kappa(\Omega^\star)} (\Delta_g u)^*_+(s){\rm d}s- \int_0^{V_\kappa(\Omega^\star)-V_{\kappa}(\mathbb S_\kappa^n)j(t)} (\Delta_g u)^*_+(s){\rm d}s \\&=  \nonumber\int_{\Omega^\star}(\Delta_g u)^*_+(V_\kappa(C_\kappa^n(d_\kappa(N,x)))){\rm d} v_\kappa(x)-\int_{C_\kappa^n(\widetilde a_t)}(\Delta_g u)^*_+(V_\kappa(C_\kappa^n(d_\kappa(N,x)))){\rm d} v_\kappa(x)\\&=  \int_{(\Omega_+^\Delta)^\star}(\Delta_g u)^\star_+(x){\rm d} v_\kappa(x)-\int_{C_\kappa^n(\widetilde a_t)\cap (\Omega_+^\Delta)^\star}(\Delta_g u)^\star_+(x){\rm d} v_\kappa(x).
	\end{align}
Applying Lemma \ref{lemma-1} (with $S=W\coloneqq \Omega_+^\Delta$  and $U\coloneqq (\Delta_g u)_+$), one has that
\begin{eqnarray}\label{II-becsles-1}
\int_{(\Omega_+^\Delta)^\star}(\Delta_g u)^\star_+(x){\rm d} v_\kappa(x)=\frac{V_{\kappa}(\mathbb S_\kappa^n)}{V_g(M)}\int_{\Omega_+^\Delta}(\Delta_g u)_+(x){\rm d} v_g(x).
\end{eqnarray}
Furthermore, 
if we consider the set $S_t=\{u\leq t\}$, it turns out that $S_t^\star=C_\kappa^n(\widetilde a_t)$; indeed, the latter fact follows from the definition of normalized rearrangement combined with  
relations  $V_g(S_t)=V_g(\Omega)-V_g(M)j(t)$ and $V_\kappa(C_\kappa^n(\widetilde a_t))=V_\kappa(\Omega^\star)-V_{\kappa}(\mathbb S_\kappa^n)j(t)$. Accordingly, due to Lemma \ref{lemma-1} (with $S\coloneqq \Omega_+^\Delta$, $W\coloneqq S_t\cap \Omega_+^\Delta$ and $U\coloneqq (\Delta_g u)_+$), it follows that
\begin{align}\label{II-becsles-2}
\nonumber \int_{C_\kappa^n(\widetilde a_t)\cap (\Omega_+^\Delta)^\star}(\Delta_g u)^\star_+(x){\rm d} v_\kappa(x)&=\int_{S_t^\star\cap (\Omega_+^\Delta)^\star}(\Delta_g u)^\star_+(x){\rm d} v_\kappa(x)\geq \int_{(S_t\cap \Omega_+^\Delta)^\star}(\Delta_g u)^\star_+(x){\rm d} v_\kappa(x)\\&\geq  \frac{V_{\kappa}(\mathbb S_\kappa^n)}{V_g(M)}\int_{S_t\cap \Omega_+^\Delta}(\Delta_g u)_+(x){\rm d} v_g(x).
%\\&=&\frac{V_{\kappa}(\mathbb S_\kappa^n)}{V_g(M)}\int_{S_t}(\Delta_g u)_+(x){\rm d} v_g(x).
\end{align}
Using relations  \eqref{II-becsles-0}--\eqref{II-becsles-2}, it follows that 
\begin{align}\label{II-becsles-3}
	\nonumber II&\leq   \frac{V_{\kappa}(\mathbb S_\kappa^n)}{V_g(M)}\int_{\Omega_+^\Delta}(\Delta_g u)_+(x){\rm d} v_g(x)-\frac{V_{\kappa}(\mathbb S_\kappa^n)}{V_g(M)}\int_{S_t\cap \Omega_+^\Delta}(\Delta_g u)_+(x){\rm d} v_g(x)\\&=	\nonumber\frac{V_{\kappa}(\mathbb S_\kappa^n)}{V_g(M)}\int_{\Omega_+^\Delta\setminus S_t}(\Delta_g u)_+(x){\rm d} v_g(x)=\frac{V_{\kappa}(\mathbb S_\kappa^n)}{V_g(M)}\int_{\Omega_+^\Delta\cap  \{u>t\}}(\Delta_g u)_+(x){\rm d} v_g(x)\\&=\frac{V_{\kappa}(\mathbb S_\kappa^n)}{V_g(M)}\int_{  \{u>t\}}(\Delta_g u)_+(x){\rm d} v_g(x).
\end{align}
The estimates \eqref{I-becsles} and \eqref{II-becsles-3} imply that
\begin{align*}
\ds\displaystyle\int_0^{V_{\kappa}(\mathbb S_\kappa^n)j(t)} \mathcal J(s){\rm d}s&=I-II\\&\geq\frac{V_{\kappa}(\mathbb S_\kappa^n)}{V_g(M)}\int_{\{u>t \}}(\Delta_g u)_-(x){\rm d}v_g(x) -\frac{V_{\kappa}(\mathbb S_\kappa^n)}{V_g(M)}\int_{  \{u>t\}}(\Delta_g u)_+(x){\rm d} v_g(x)\\&=-\frac{V_{\kappa}(\mathbb S_\kappa^n)}{V_g(M)}\int_{\{u>t \}}(\Delta_g u)(x){\rm d}v_g(x),
\end{align*}
	which is precisely the required inequality.  

(ii) The proof is similar to (i). 

(iii) Let us fix the parameter $s\in [0,V_\kappa(\Omega^\star)]$; if either $s=0$ or $s=V_\kappa(\Omega^\star)$, the claim trivially holds. Otherwise,  pick $x,\widetilde x\in \Omega^\star$ such that $s=V_\kappa(C_\kappa^n(d_\kappa(N,x)))$ and $V_\kappa(\Omega^\star)-s=V_\kappa(C_\kappa^n(d_\kappa(N,\widetilde x)))$. 

On one hand, if $x\notin (\Omega_-^\Delta)^\star$, then $(\Delta_g u)^*_-(s)=(\Delta_g u)^*_-(V_\kappa(C_\kappa^n(d_\kappa(N,x))))=(\Delta_g u)^\star_-(x)=0$. On the other hand, if $x\in (\Omega_-^\Delta)^\star$, then $s=V_\kappa(C_\kappa^n(d_\kappa(N,x)))< V_\kappa((\Omega_-^\Delta)^\star)$. Consequently, $$V_\kappa(C_\kappa^n(d_\kappa(N,\widetilde x)))=V_\kappa(\Omega^\star)-s> V_\kappa(\Omega^\star)-V_\kappa((\Omega_-^\Delta)^\star)=V_\kappa((\Omega_+^\Delta)^\star),$$ i.e., $(\Omega_+^\Delta)^\star\subset C_\kappa^n(d_\kappa(N,\widetilde x))$ with strict inclusion. Therefore, one has $\widetilde x\notin (\Omega_+^\Delta)^\star$, i.e.,  $0=(\Delta_g u)^\star_+(\widetilde x)=(\Delta_g u)^*_+(V_\kappa(C_\kappa^n(d_\kappa(N,\widetilde x))))=(\Delta_g u)^*_+(V_\kappa(\Omega^\star)-s)$, which concludes the proof.  
	\hfill $\square$\\
	
	%We now split our argument.
	
	For further use, let  $a,b\geq 0$ and $L>0$ be such that 
	\begin{equation}\label{a-b}
	C_\kappa^n(a)=\Omega_+^\star,\ \   C_\kappa^n(b)=\Omega_-^\star \ \ {\rm and} \ \ C_\kappa^n(L)=\Omega^\star.
	\end{equation}
%	We consider the functions   $U_a,U_b:\to \mathbb R$ defined by 
%
%	
%	By keeping the above notations, 
	The main result of this subsection reads as follows. 
		
	\begin{theorem} \label{talenti-result} The real functions 
			\begin{equation}\label{v-definit}
			U_a(x)\coloneqq\frac{1}{n\omega_n}\int_{d_\kappa(N,x)}^a \left(\frac{\sin(\sqrt{\kappa}\rho)}{\sqrt{\kappa}}\right)^{1-n}\left(\int_0^{V_\kappa(C_\kappa^n(\rho))}\mathcal J(s){\rm d}s\right){\rm d}\rho,\ x\in \Omega^\star,
			%
			%\left\{ \begin{array}{lll}
			%\ds\frac{1}{n\omega_n}\int_{d_\kappa(x)}^a \left(\frac{\sinh(\kappa\rho)}{\kappa}\right)^{1-n}\left(\int_0^{V_\kappa(\rho)}F(s){\rm d}s\right){\rm d}\rho &\mbox{if} &  \kappa>0; \\
			%% u\geq 0 &\mbox{in} &   \Omega;\\
			%\ds\frac{1}{n\omega_n}\int_{d_0(x)}^a \rho^{1-n}\left(\int_0^{V_0(\rho)}F(s){\rm d}s\right){\rm d}\rho  &\mbox{if} &  \kappa=0.
			%\end{array}\right.
		\end{equation}
		and
		\begin{equation}\label{w-definit}
			U_b(x)\coloneqq\frac{1}{n\omega_n}\int_{d_\kappa(N,x)}^b \left(\frac{\sin(\sqrt{\kappa}\rho)}{\sqrt{\kappa}}\right)^{1-n}\left(\int_0^{V_\kappa(C_\kappa^n(\rho))}\mathcal H(s){\rm d}s\right){\rm d}\rho,\ x\in \Omega^\star,
		\end{equation}
	satisfy the following 	statements$:$ 
		\begin{itemize}
			\item[(i)] $\displaystyle\frac{V_{\kappa}(\mathbb S_\kappa^n)}{V_g(M)} \int_{\Omega} (\Delta_g u)^2{\rm d}v_g=\int_{C_\kappa^n(a)} (\Delta_\kappa U_a)^2{\rm d}v_\kappa+ \int_{C_\kappa^n(b)} (\Delta_\kappa U_b)^2{\rm d}v_\kappa;$
			\item[(ii)] $\displaystyle\frac{V_{\kappa}(\mathbb S_\kappa^n)}{V_g(M)}\int_{\Omega}u^2 {\rm d}v_g\leq \int_{C_\kappa^n(a)}U_a^2 {\rm d}v_\kappa+\int_{C_\kappa^n(b)}U_b^2 {\rm d}v_\kappa.$
		\end{itemize}
		
%		In particular, one has
%		\begin{equation}\label{u-v-w}
%		\int_{\Omega}u^2 {\rm d}v_g\leq \int_{B_\kappa(a)}v^2 {\rm d}v_\kappa+\int_{B_\kappa(b)}w^2 {\rm d}v_\kappa.
%		\end{equation}
%		In addition, 
%		\begin{equation}\label{laplace-hasonlitas}
%		\int_{\Omega} (\Delta_g u)^2{\rm d}v_g=\int_{B_\kappa(a)} (\Delta_\kappa v)^2{\rm d}v_\kappa+ \int_{B_\kappa(b)} (\Delta_\kappa w)^2{\rm d}v_\kappa.
%		\end{equation}
	\end{theorem}
	 
\noindent 	{\it Proof.} (i) By construction, both $U_a$ and $U_b$ in \eqref{v-definit} and \eqref{w-definit} are spherical cap symmetric functions; according to  \eqref{Laplace-operator}, a direct computation shows that $U_a$ and $U_b$ solve the problems 
 \begin{equation}\label{CP-problem-v}
	\left\{ \begin{array}{rclll}
	-\Delta_\kappa U_a(x)&~=&~\mathcal J(V_\kappa(C_\kappa^n((d_\kappa(N,x)))) &\mbox{in} &  C_\kappa^n(a)=\Omega_+^\star; \\
	% u\geq 0 &\mbox{in} &   \Omega;\\
	U_a&~=&~0  &\mbox{on} &  \partial C_\kappa^n(a),
	\end{array}\right.
	\end{equation}
%	In a similar way, the function  $w:\Omega^\star\to \mathbb R$ given by 
%	\begin{equation}\label{w-definit}
%	w(x)=\frac{1}{n\omega_n}\int_{d_\kappa(x)}^b {\mathbf s}_\kappa(\rho)^{1-n}\left(\int_0^{V_\kappa(\rho)}H(s){\rm d}s\right){\rm d}\rho
%	%
%	%\left\{ \begin{array}{lll}
%	%\ds\frac{1}{n\omega_n}\int_{d_\kappa(x)}^b \left(\frac{\sinh(\kappa\rho)}{\kappa}\right)^{1-n}\left(\int_0^{V_\kappa(\rho)}G(s){\rm d}s\right){\rm d}\rho &\mbox{if} &  \kappa>0; \\
%	%% u\geq 0 &\mbox{in} &   \Omega;\\
%	%\ds\frac{1}{n\omega_n}\int_{d_0(x)}^b \rho^{1-n}\left(\int_0^{V_0(\rho)}G(s){\rm d}s\right){\rm d}\rho  &\mbox{if} &  \kappa=0.
%	%\end{array}\right.
%	\end{equation}
	and 
	\begin{equation}\label{CP-problem-w}
	\left\{ \begin{array}{rllll}
	-\Delta_\kappa U_b(x)&~=&~\mathcal H(V_\kappa(C_\kappa^n((d_\kappa(N,x)))) &\mbox{in} &  C_\kappa^n(b)=\Omega_-^\star; \\
	% u\geq 0 &\mbox{in} &   \Omega;\\
	U_b&~=&~0  &\mbox{on} &  \partial C_\kappa^n(b),
	\end{array}\right.
	\end{equation}
 respectively. On one hand, by (\ref{CP-problem-v}), (\ref{CP-problem-w}), \eqref{change-variables}, the definition of  $\mathcal H$ and by the fact that $V_\kappa(\Omega^\star)=V_\kappa(\Omega^\star_+)+V_\kappa(\Omega^\star_-)$, we obtain that
	\begin{align*}
		I&\coloneqq \int_{C^n_\kappa(a)} (\Delta_\kappa U_a)^2{\rm d}v_\kappa+ \int_{C_\kappa^n(b)} (\Delta_\kappa U_b)^2{\rm d}v_\kappa\\&=\int_{C_\kappa^n(a)} \mathcal J^2(V_\kappa(C_\kappa^n(d_\kappa(N,x)))){\rm d}v_\kappa(x)+ \int_{C^n_\kappa(b)} \mathcal H^2(V_\kappa(C_\kappa^n(d_\kappa(N,x)))){\rm d}v_\kappa(x)\\&= \int_{0}^{V_\kappa(\Omega_+^\star)}\mathcal  J^2(s){\rm d}s+ \int_{0}^{V_\kappa(\Omega_-^\star)}\mathcal  H^2(s){\rm d}s\\&=\int_{0}^{V_\kappa(\Omega^\star)}\mathcal  J^2(s){\rm d}s.
	\end{align*}
	On the other hand, by the definition of the function $\mathcal J$, Proposition \ref{f-hasonlitas}/(iii), relation \eqref{change-variables} and Lemma \ref{lemma-1} (with  $S=W\coloneqq \Omega_\pm^\Delta$ and $U\coloneqq (\Delta_g u)_\pm$), we have that
	%\\
	%$\ \ \ \ \ \ds	\int_{0}^{V_g(\Omega)} F(s)^2{\rm d}s=$
	%\begin{eqnarray*}
	%&=&\int_{0}^{V_g(\Omega)} \left[(\Delta_g u)^\#_-(s)^2+(\Delta_g u)^\#_+(V_g(\Omega)-s)^2-2(\Delta_g u)^\#_-(s)(\Delta_g u)^\#_+(V_g(\Omega)-s)\right]{\rm d}s\\&=&\int_{0}^{V_g(\Omega)} \left[(\Delta_g u)^\#_-(s)^2+(\Delta_g u)^\#_+(s)^2-2(\Delta_g u)^\#_-(s)(\Delta_g u)^\#_+(V_g(\Omega)-s)\right]{\rm d}s.
	%\end{eqnarray*}
	\begin{align*}
	I&=\int_{0}^{V_\kappa(\Omega^\star)} \mathcal J^2(s){\rm d}s
	\\&=\int_{0}^{V_\kappa(\Omega^\star)} \left[(\Delta_g u)^*_-(s)^2+(\Delta_g u)^*_+(V_\kappa(\Omega^\star)-s)^2 -2(\Delta_g u)^*_-(s)(\Delta_g u)^*_+(V_\kappa(\Omega^\star)-s)\right]{\rm d}s
	\\
	&=\int_{0}^{V_\kappa(\Omega^\star)} \left[(\Delta_g u)^*_-(s)^2+(\Delta_g u)^*_+(V_\kappa(\Omega^\star)-s)^2\right]{\rm d}s\\&=
	\int_{0}^{V_\kappa(\Omega^\star)} \left[(\Delta_g u)^*_-(s)^2+(\Delta_g u)^*_+(s)^2\right]{\rm d}s\\&=
		\int_{\Omega^\star} \left[(\Delta_g u)^*_-(V_\kappa(C_\kappa^n(d_\kappa(N,x))))^2+(\Delta_g u)^*_+(V_\kappa(C_\kappa^n(d_\kappa(N,x))))^2\right]{\rm d}v_\kappa(x)\\&=\int_{\Omega^\star} \left[(\Delta_g u)^\star_-(x)^2+(\Delta_g u)^\star_+(x)^2\right]{\rm d}v_\kappa(x)\\&=\int_{{(\Omega_-^\Delta)}^\star} (\Delta_g u)^\star_-(x)^2{\rm d}v_\kappa(x)+\int_{{(\Omega_+^\Delta)}^\star}(\Delta_g u)^\star_+(x)^2{\rm d}v_\kappa(x)
		\\&=\frac{V_{\kappa}(\mathbb S_\kappa^n)}{V_g(M)}\int_{{\Omega_-^\Delta}} (\Delta_g u)^2_-(x){\rm d}v_g(x)+\frac{V_{\kappa}(\mathbb S_\kappa^n)}{V_g(M)}\int_{{\Omega_+^\Delta}} (\Delta_g u)^2_+(x){\rm d}v_g(x)
		\\&=\frac{V_{\kappa}(\mathbb S_\kappa^n)}{V_g(M)}\int_{\Omega} (\Delta_g u)^2(x){\rm d}v_g(x),
	\end{align*}
	which concludes the proof of (i). 

	(ii) For every $t\in [0,T_u^+]$, with $$T_u^+=\sup_{x\in \Omega_+}u_+(x),$$ we  introduce  the sets 
	$$
	\Lambda_t\coloneqq\partial(\{x\in \Omega: u_+(x)>t \})\ \  {\rm and} \  \Lambda_t^\star\coloneqq\partial(\{x\in \mathbb S_\kappa^n: u_+^\star(x)>t \}).  
	$$
	For any fixed $\varepsilon>0$, Cauchy's inequality implies that
	\begin{align*}
	\left(\frac{1}{\varepsilon}\int_{t<u(x)\leq t+\varepsilon}|\nabla_g u(x)|_g{\rm d}v_g(x)\right)^2&\leq \frac{V_g(u^{-1}((t,t+\varepsilon]))}{\varepsilon}\frac{1}{\varepsilon}\int_{t<u(x)\leq t+\varepsilon}|\nabla_g u(x)|_g^2{\rm d}v_g(x)\\&=V_g(M)\frac{j(t)-j(t+\varepsilon)}{\varepsilon}\frac{1}{\varepsilon}\int_{t<u(x)\leq t+\varepsilon}|\nabla_g u(x)|_g^2{\rm d}v_g(x).
	\end{align*}
Since $j$ is non-increasing, by letting $\varepsilon\to 0$, the co-area formula (see Chavel \cite[p.86]{Chavel}) and the latter inequality imply that 
$$
	\ds{\mathcal P}_g^2(\Lambda_t)\leq -j'(t)V_g(M)\int_{\Lambda_t}|\nabla_g u|{\rm d}\mathcal H_{n-1}\ \ {\rm for\ a.e.}\ 	t\in [0,T_u^+],
$$
	where ${\mathcal P}_g(\Lambda_t)$ denotes the induced perimeter of the set $\{u_+>t\}\subset M$, while $\mathcal H_{n-1}$ stands for the $(n-1)$-dimensional Hausdorff measure. Since the outward normal vector at  $x\in \Lambda_t$ to  $\{u>t\}$ is $-\frac{\nabla_g u(x)}{|\nabla_g u(x)|_g},$  the divergence theorem implies up to a negligible set  that 
	$$\int_{\Lambda_t}|\nabla_g u|_g{\rm d}\mathcal H_{n-1}=-\int_{\{ u_+>t\}}\Delta_gu{\rm d}v_g=-\int_{\{u>t\}}\Delta_gu{\rm d}v_g.$$
	Therefore,  for a.e.\ $t\in [0,T_u^+]$, we obtain the inequality 
	\begin{equation}\label{peri-1}
	\ds{\mathcal P}^2_g(\Lambda_t)\leq V_g(M)j'(t)\int_{\{u>t\}}\Delta_gu{\rm d}v_g.
	\end{equation}
	By relation \eqref{szimmetrizacio-U}, the L\'evy--Gromov isoperimetric inequality implies that for every $t\in [0,T_u^+]$ one has
\begin{equation}\label{Levy--Gromov-inequal}
	\frac{{\mathcal P}_\kappa(\Lambda_t^\star)}{V_\kappa(\mathbb S_\kappa^n)}\leq \frac{{\mathcal P}_g(\Lambda_t)}{V_g(M)},\ \forall 	t\in [0,T_u^+],
\end{equation}
	where $\mathcal P_\kappa(\Lambda_t^\star)$ denotes the  perimeter of the set $\{u_+^\star>t\}\subset \mathbb S_\kappa^n$. 
	
	Combining inequalities \eqref{peri-1}-\eqref{Levy--Gromov-inequal}  with  Proposition \ref{f-hasonlitas}/(i), we have  that
	\begin{equation}\label{ppp}
	{\mathcal P}^2_\kappa(\Lambda_t^\star)\leq -{V_{\kappa}(\mathbb S_\kappa^n)}j'(t)\displaystyle\int_0^{V_{\kappa}(\mathbb S_\kappa^n)j(t)} \mathcal J(s){\rm d}s\ \ {\rm for\ a.e.}\ t\in [0,T_u^+].
	\end{equation}
We have seen in the proof of 	Proposition \ref{f-hasonlitas} that $$\{u_+>t \}^\star=C_\kappa^n(a_t) \ \ {\rm and}\ \   V_\kappa(C_\kappa^n(a_t))=V_{\kappa}(\mathbb S_\kappa^n)j(t)$$ for some $a_t\geq 0$; note that $t\mapsto a_t$ is decreasing on $[0,T_u^+]$. Since for sufficiently small $\varepsilon>0$ one has that $$\{x:\in \mathbb S_\kappa^n: d_\kappa (x,C_\kappa^n(a_t))<\varepsilon\}=C_\kappa^n(a_t+\varepsilon),$$ the Minkowski content of $C_\kappa^n(a_t)$ and  \eqref{terfogat-cap} ensure that 
$$V_{\kappa}(\mathbb S_\kappa^n)j'(t)=n\omega_n\left(\frac{\sin(\sqrt{\kappa}a_t)}{\sqrt{\kappa}}\right)^{n-1}a_t'={\mathcal P}_\kappa(\Lambda_t^\star)a_t'\ \ {\rm for\ a.e.}\ t\in [0,T_u^+].$$
Combining this relation with \eqref{ppp}, we infer that 
	$$1\leq -\frac{1}{n\omega_n}\left(\frac{\sin(\sqrt{\kappa}a_t)}{\sqrt{\kappa}}\right)^{1-n}a_t'\int_0^{V_\kappa(C_\kappa^n(a_t))} \mathcal J(s){\rm d}s\ \ {\rm for\ a.e.}\ t\in [0,T_u^+].$$
For every fixed $\eta\in[0,T_u^+]$, an integration 	
	of the latter inequality yields that
	\begin{equation}\label{utolso-inequality}
		\eta\leq -\frac{1}{n\omega_n}\int_0^\eta\left(\frac{\sin(\sqrt{\kappa}a_t)}{\sqrt{\kappa}}\right)^{1-n}a_t'\left(\int_0^{V_\kappa(C_\kappa^n(a_t))} \mathcal J(s){\rm d}s\right) {\rm d}t.
	\end{equation}
Arbitrary fixing the point $x\in C_\kappa^n(a)=\Omega_+^\star$, there exists a unique $\eta\in [0,T_u^+]$ such that $d_\kappa(N,x)=a_\eta$. By \eqref{u-plusz-minusz} and the monotonicity of $t\mapsto a_t$ on $[0,T_u^+]$, it follows that 
	\begin{align*}
	u_+^\star(x)&=\sup\{t>0:x\in \{u_+>t\}^\star\}=\sup\{t>0:x\in C_\kappa^n(a_t)\}\\&=\sup\{t>0:d_\kappa(N,x)<a_t\}=\sup\{t>0:a_\eta<a_t\}\\&=\eta.
	\end{align*}
Performing the change $\rho=a_t$ of variables at the right hand side of \eqref{utolso-inequality}, and taking into account that $\lim_{t\to 0}a_t=a$ (with $a\geq 0$  by \eqref{a-b}),
the latter inequality together with (\ref{v-definit}) implies that 	
%Furthermore,  thus  is precisely 
%$$\alpha'(t)={A}_\kappa(\Lambda_t^\star)r_t'=n\omega_n{\mathbf s}_\kappa(r_t)^{n-1}r_t'.$$
%and by the definition of , it turns out that $${\mathcal P}_\kappa(\Lambda_t^\star)=\frac{\rm d}{{\rm d}t}V_\kappa(C_\kappa^n(a_t))=V_{\kappa}(\mathbb S_\kappa^n)j'(t)\ \ {\rm for\ a.e.}\ t>0;$$ here we observe that On the other hand,  \eqref{terfogat-cap}  yields 
%$${\mathcal P}_\kappa(\Lambda_t^\star)=\frac{\rm d}{{\rm d}t}V_\kappa(C_\kappa^n(a_t))=n\omega_n\left(\frac{\sin(\sqrt{\kappa}a_t)}{\sqrt{\kappa}}\right)^{n-1}\ \ {\rm for\ a.e.}\ t>0.$$
	\begin{equation}\label{1-compar}
	u_+^\star\leq U_a\ \ {\rm in}\ \  C_\kappa^n(a)=\Omega_+^\star;
	\end{equation}
see also Figure \ref{3-abra}. In a similar manner, we have that	\begin{equation}\label{2-compar}
u_-^\star\leq U_b\ \ {\rm in}\ \  C_\kappa^n(b)=\Omega_-^\star,
\end{equation}
by applying analogously to \eqref{peri-1} that
\begin{equation}\label{peri-2}
\ds{\mathcal P}^2_g(\Pi_t)\leq V_g(M)h'(t)\int_{\{u<-t\}}\Delta_gu{\rm d}v_g \ \ {\rm for\ a.e.}\ t\in [0,T_u^-],
\end{equation}
where		$
	\Pi_t\coloneqq\partial(\{x\in \Omega: -u(x)>t \})$.
	%  and $  \Pi_t^\star\coloneqq\partial(\{x\in \mathbb S_\kappa^n: u_-^\star(x)>t \}). 
%	$

		\begin{figure}[h!]
		\centering
		\includegraphics[scale=1.0]{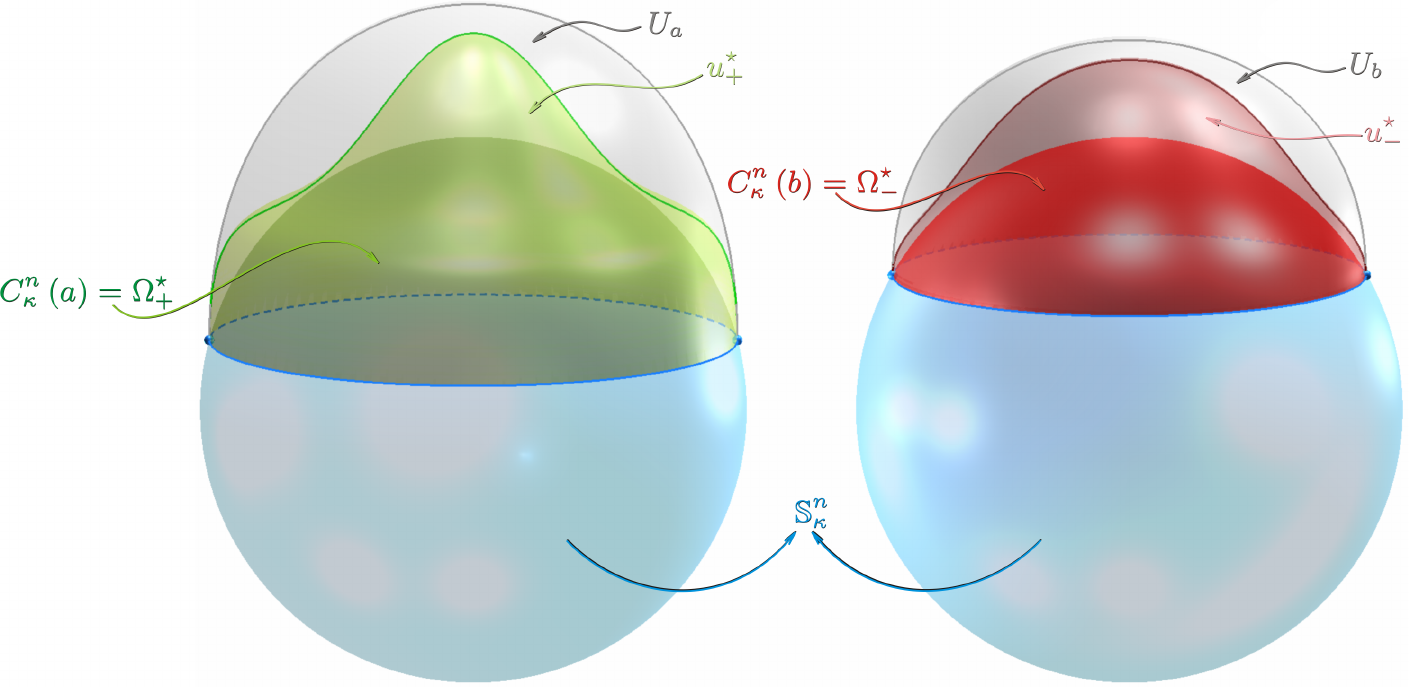}
		\caption{Spherical cap symmetric functions $U_a,u_+^\star:\Omega_+^\star\to (0,\infty)$ and  $U_b,u_-^\star:\Omega_-^\star\to (0,\infty)$ which verify the inequalities \eqref{1-compar}--\eqref{2-compar}. The functions  $u_\pm^\star: \Omega_\pm^\star \to (0,\infty)$ arise from Figure  \ref{2-abra}/(c), which are the normalized rearrangements of $u_\pm:\Omega_\pm\to (0,\infty)$  on the spherical caps $\Omega_\pm^\star\subset \mathbb S^n_\kappa$.
			%  and  the boundary condition \eqref{b-v-c}.
			%  PROBA: $\mathfrak  g_{{\nu},1}(\widetilde L_0)$ 
		}\label{3-abra}
	\end{figure}

 Inequalities \eqref{1-compar}--\eqref{2-compar}  and Lemma \ref{lemma-1} (with $S=W\coloneqq  \Omega_\pm$ and $U\coloneqq u_\pm^2$) imply that
	\begin{align*}
		\int_{\Omega}u^2 {\rm d}v_g&= \int_{\Omega_+}u_+^2 {\rm d}v_g+\int_{\Omega_-}u_-^2 {\rm d}v_g\\&=\frac{V_g(M)}{V_{\kappa}(\mathbb S_\kappa^n)}\int_{\Omega_+^\star}(u_+^\star)^2 {\rm d}v_\kappa+\frac{V_g(M)}{V_{\kappa}(\mathbb S_\kappa^n)}\int_{\Omega_-^\star}(u_-^\star)^2 {\rm d}v_\kappa\\&\leq\frac{V_g(M)}{V_{\kappa}(\mathbb S_\kappa^n)}\left(
		\int_{C_\kappa^n(a)}U_a^2 {\rm d}v_\kappa+\int_{C_\kappa^n(b)}U_b^2 {\rm d}v_\kappa\right),
	\end{align*}
	which completes the proof. 	\hfill $\square$
	
	\begin{remark}\rm We conclude this subsection with an observation concerning a joint boundary condition involving the functions $U_a$ and $U_b$.  In fact, we have that \begin{equation}\label{b-v-c-0}
			(\sin(\sqrt{\kappa}a))^{n-1}\frac{{\rm d} U_a}{{\rm d} \theta}(\sqrt{\kappa}a)=(\sin(\sqrt{\kappa}b))^{n-1}\frac{{\rm d} U_b}{{\rm d} \theta}(\sqrt{\kappa}b),
		\end{equation}
		where we explored the  spherical cap symmetry of $U_a$ and $U_b,$ by identifying the point $x=x(\theta,\xi)\in \mathbb S_n^\kappa$ with the angle $\theta=\sqrt{\kappa}d_\kappa(N,x)\in (0,\pi).$ Applying Lemma \ref{basic=J-H}/(iii) to problems \eqref{CP-problem-v}--\eqref{CP-problem-w}, we first obtain that
		\begin{equation}\label{utolso-ebben-a-fejezetben}
			\int_{C_\kappa^n(a)} \Delta_\kappa U_a(x){\rm d}v_\kappa(x)= \int_{C_\kappa^n(b)} \Delta_\kappa U_b(x){\rm d}v_\kappa(x).
		\end{equation}
		Moreover, by \eqref{Laplace-operator} and the density of the measure ${\rm d}v_\kappa$ given by relation  \eqref{terfogat-cap}, the spherical cap symmetry of $U_a$ and $U_b$ implies the required relation \eqref{b-v-c-0}. 
		
	\end{remark}	
	\vspace*{0.5cm}
	
	\subsection{Coupled minimization on spherical caps}

%	Since $U_a$ and $U_b$ are spherical cap symmetric functions,  it follows that
%	$$\Delta_\kappa U_l(x)=\kappa\left({\sin \theta}\right)^{1-n}\frac{{\rm d}}{{\rm d} \theta}\left({(\sin \theta)}^{n-1}\frac{{\rm d} U_l}{{\rm d} \theta}\right),\ \ x=x(\theta,\xi),\  l\in \{a,b\}.$$ Furthermore, by the density of the measure ${\rm d}v_\kappa$ given by relation  \eqref{terfogat-cap}, 
%%Csak magunknak	$$dv_\kappa(x)=\frac{(\sin \theta)^{n-1}}{\sqrt{\kappa}^n}d\theta d\xi,\ \ x=x(\theta,\xi)$$
%	the latter expression and relation \eqref{utolso-ebben-a-fejezetben} imply that 

	Given an open set $S\subset \mathbb S_\kappa^n$, we introduce the space $$\mathcal W(S)\coloneqq W_0^{1,2}(S)\cap W^{2,2}(S).$$ 
	Based on  \eqref{b-v-c-0}, we will consider the boundary condition 
	  \begin{equation}\label{b-v-c}
	  	(\sin(\sqrt{\kappa}a))^{n-1}\frac{{\rm d} u_a}{{\rm d} \theta}(\sqrt{\kappa}a)=(\sin(\sqrt{\kappa}b))^{n-1}\frac{{\rm d} u_b}{{\rm d} \theta}(\sqrt{\kappa}b),
	  \end{equation}
  for two spherical cap symmetric functions $u_a:C_\kappa^n(a)\to (0,\infty)$ and $u_b:C_\kappa^n(b)\to (0,\infty)$. 
	For further use, let 
%	$$\mathcal W_{a,b}(\Omega^\star)\coloneqq \{(U_a,U_b)\in\mathcal W(C_\kappa^n(a))\times \mathcal W(C_\kappa^n(b)): U_a,U_b\ {\rm are\ spherical \ cap\ symmetric\ and} \eqref{b-v-c}\ {\rm holds}\},$$
%	
	$$\mathcal W_{a,b}(\Omega^\star)\coloneqq \left\{ (u_a,u_b)\in\mathcal W(C_\kappa^n(a))\times \mathcal W(C_\kappa^n(b)):\begin{array}{lll}
	u_a,u_b\ {\rm are\ spherical \ cap\ symmetric}  \\
{\rm functions\ that\ verify\ relation}\ \eqref{b-v-c}
	\end{array}\right\},
	$$	where 
	\begin{equation}\label{terfogatok-ujra}
	V_\kappa(C_\kappa^n(a))+V_\kappa(C_\kappa^n(b))=V_\kappa(\Omega^\star).
	\end{equation}	
%	\newpage

 Using the aforementioned notations and Theorem \ref{talenti-result}, one can establish the connection between the fundamental tone of a set $\Omega\subset M$ and a coupled minimization problem on $\mathbb S_\kappa^n;$  namely, if  $a,b\geq 0$ are numbers such that \eqref{terfogatok-ujra} holds, then  
 \begin{equation}\label{atteres-sokasagrol-gombre-0}
 \Gamma_g(\Omega)\geq \inf_{ (u_a,u_b)\in\mathcal W_{a,b}(\Omega^\star)\setminus\{(0,0)\}}\frac{\ds\int_{C_\kappa^n(a)} (\Delta_\kappa u_a)^2{\rm d}v_\kappa+ \int_{C_\kappa^n(b)} (\Delta_\kappa u_b)^2{\rm d}v_\kappa}{\ds\int_{C_\kappa^n(a)}u_a^2 {\rm d}v_\kappa+\int_{C_\kappa^n(b)}u_b^2 {\rm d}v_\kappa}.
 \end{equation}

%	\begin{corollary}\label{atteres-sokasagrol-gombre}
%		Let $(M,g)$ be a compact $n$-dimensional Riemannian manifold with ${\rm Ric}(M,g)\geq (n-1)\kappa>0$, the open domain $\Omega\subset M$ and $\Omega^\star\subset \mathbb S^n_\kappa$ its relative rearrangement. 
%	\end{corollary}
	
\noindent 
For further use, we recall the Gaussian hypergeometric functions 
\begin{equation}\label{Ferrers-fct}
\mathcal F_{\pm}(t,\lambda,\kappa,n)\coloneqq {_2F}_1\left(\frac{1}{2}-\Lambda_{\pm}(\lambda),\frac{1}{2}+\Lambda_{\pm}(\lambda);\frac{n}{2};t\right), \ t\in (0,1),
\end{equation}
where
\begin{equation}\label{lambda-pm}
\Lambda_{\pm}(\lambda)=\Lambda_{\pm}(\lambda,\kappa,n)\coloneqq \sqrt{\frac{(n-1)^2}{4}\pm \frac{{\lambda^2}}{\kappa}}\in \mathbb C.
\end{equation}
Given $\lambda>0$, we also need the cross-product  of the Gaussian hypergeometric functions, i.e., 
\begin{equation}\label{K-definicio}
 \mathcal K_{\kappa,n}(t,\lambda)\coloneqq\frac{\mathcal F_{-}'(t,\lambda,\kappa,n)}{\mathcal F_{-}(t,\lambda,\kappa,n)}-\frac{\mathcal F_{+}'(t,\lambda,\kappa,n)}{\mathcal F_{+}(t,\lambda,\kappa,n)},
\end{equation}
defined outside of the zeros of ${\mathcal F_{+}(\cdot,\cdot,\kappa,n)}$, see Proposition \ref{Ferrers-basic-lemma-1}/(iii)--(iv); here we use  the notation $\mathcal F_{\pm}'(t,\lambda,\kappa,n)\coloneqq\dfrac{\partial}{\partial t}\mathcal F_{\pm}(t,\lambda,\kappa,n).$
Based on \eqref{atteres-sokasagrol-gombre-0}, a crucial result in our  investigation can be formulated as follows. 
	
\begin{theorem}\label{fundametal-tones-double}
		Let $(M,g)$ be a compact $n$-dimensional Riemannian manifold of Ricci curvature  ${\sf Ric}_{(M,g)}\geq (n-1)\kappa>0$, the open domain $\Omega\subset M$ with its normalized rearrangement $\Omega^\star\subset \mathbb S^n_\kappa$. If  $a,b\geq 0$ are numbers such that \eqref{terfogatok-ujra} holds, then  
		\begin{equation}\label{atteres-sokasagrol-gombre-1}
		\Lambda_g(\Omega)\geq \Lambda(\kappa,n,a,b)=:\lambda_{\kappa,n}(\alpha,\beta)^4,
		\end{equation}
			 where 
			 \begin{equation}\label{alfa-beta}
			 \alpha\coloneqq \sin^2\left(\frac{\sqrt{\kappa} a}{2}\right),\ \beta\coloneqq \sin^2\left(\frac{\sqrt{\kappa} b}{2}\right),
			 \end{equation}
			 and 			 
			 $\lambda\coloneqq \lambda_{\kappa,n}(\alpha,\beta)>0$ is the smallest positive zero of the equation 
			 \begin{equation}\label{cross-product-equation}
			 		 (1-\alpha)^{\frac{n}{2}}\alpha^{\frac{n}{2}}\mathcal K_{\kappa,n}(\alpha,\lambda)+(1-\beta)^{\frac{n}{2}}\beta^{\frac{n}{2}}\mathcal K_{\kappa,n}(\beta,\lambda)=0.
			 \end{equation}	 
\end{theorem}	
	
\noindent	{\it Proof.}  
	Let $a,b\geq 0$ be real numbers  that verify \eqref{terfogatok-ujra}. We note that the infimum in \eqref{atteres-sokasagrol-gombre-0} is a minimum; this fact can be stated, by using a similar argument as in the flat and negatively curved cases studied by Ashbaugh and Benguria \cite{A-B} and Krist\'aly \cite{Kristaly-Adv-math}, respectively. 
	 Accordingly, the value  
	\begin{equation}\label{Lambda-ertelmezes}
		\Lambda(\kappa,n,a,b)\coloneqq \min_{ (u_a,u_b)\in\mathcal W_{a,b}(\Omega^\star)\setminus\{(0,0)\}}\frac{\ds\int_{C_\kappa^n(a)} (\Delta_\kappa u_a)^2{\rm d}v_\kappa+ \int_{C_\kappa^n(b)} (\Delta_\kappa u_b)^2{\rm d}v_\kappa}{\ds\int_{C_\kappa^n(a)}u_a^2 {\rm d}v_\kappa+\int_{C_\kappa^n(b)}u_b^2 {\rm d}v_\kappa}
	\end{equation}
	is achieved by a pair of functions $(u_a,u_b)\in\mathcal W_{a,b}(\Omega^\star)\setminus\{(0,0)\}$. For simplicity of notation, let $\lambda\coloneqq \lambda_{\kappa,n}(\alpha,\beta)>0$ with $\lambda_{\kappa,n}(\alpha,\beta)^4=\Lambda(\kappa,n,a,b).$
	The Euler--Lagrange equation implies that  
	\begin{align}\label{Euler-Lagrange}
	0&=\int_{C_\kappa^n(a)}(\Delta_\kappa u_a\Delta_\kappa \phi -\lambda^4  u_a \phi){\rm d}v_\kappa + \int_{C_\kappa^n(b)}(\Delta_\kappa u_b\Delta_\kappa \psi -\lambda^4  u_b \psi){\rm d}v_\kappa,
	\end{align}
	for every pair of functions $(\phi,\psi)\in\mathcal W_{a,b}(\Omega^\star)$; in particular, they verify the boundary condition
	\begin{equation}\label{b-v-c-1}
	(\sin(\sqrt{\kappa}a))^{n-1}\frac{{\rm d} \phi}{{\rm d} \theta}(\sqrt{\kappa}a)=(\sin(\sqrt{\kappa}b))^{n-1}\frac{{\rm d} \psi}{{\rm d} \theta}(\sqrt{\kappa}b).
	\end{equation}
	Using the particular form of the measure ${\rm d}v_\kappa$ in  \eqref{terfogat-cap} and integrating by parts together with the fact that $$\phi(\sqrt{\kappa}a)=\psi(\sqrt{\kappa}b)=0,$$ we have that 
	\begin{align*}
	\int_{C_\kappa^n(a)}\Delta_\kappa u_a\Delta_\kappa \phi {\rm d}v_\kappa&= n\omega_n\kappa^{1-\frac{n}{2}}\Delta_\kappa u_a(\sqrt{\kappa}a)(\sin(\sqrt{\kappa}a))^{n-1}\frac{{\rm d} \phi}{{\rm d} \theta}(\sqrt{\kappa}a)+\int_{C_\kappa^n(a)}\Delta_\kappa^2 u_a \phi {\rm d}v_\kappa
	\end{align*}
	and 
	\begin{align*}
		\int_{C_\kappa^n(b)}\Delta_\kappa u_b\Delta_\kappa \psi {\rm d}v_\kappa&= n\omega_n\kappa^{1-\frac{n}{2}}\Delta_\kappa u_b(\sqrt{\kappa}b)(\sin(\sqrt{\kappa}b))^{n-1}\frac{{\rm d} \psi}{{\rm d} \theta}(\sqrt{\kappa}b)+\int_{C_\kappa^n(b)}\Delta_\kappa^2 u_b \psi {\rm d}v_\kappa.
	\end{align*}
	Due to the latter relations, equation \eqref{Euler-Lagrange} transforms into 
	\begin{align}\label{Euler-Lagrange-1}
	0=&\nonumber ~ n\omega_n\kappa^{1-\frac{n}{2}}\left(\Delta_\kappa u_a(\sqrt{\kappa}a)(\sin(\sqrt{\kappa}a))^{n-1}\frac{{\rm d} \phi}{{\rm d} \theta}(\sqrt{\kappa}a) + \Delta_\kappa u_b(\sqrt{\kappa}b)(\sin(\sqrt{\kappa}b))^{n-1}\frac{{\rm d} \psi}{{\rm d} \theta}(\sqrt{\kappa}b)\right)\\&+\int_{C_\kappa^n(a)}(\Delta_\kappa^2 u_a  -\lambda^4  u_a )\phi{\rm d}v_\kappa + \int_{C_\kappa^n(b)}(\Delta_\kappa^2 u_b  -\lambda^4  u_b )\psi{\rm d}v_\kappa.
	\end{align}
	In (\ref{Euler-Lagrange-1}) we may choose either $\psi=0$ and  $\phi\in \mathcal C^2_0(C_\kappa^n(a))$, or  $\phi=0$ and  $\psi\in \mathcal C^2_0(C_\kappa^n(b))$,  obtaining that 
	\begin{equation}\label{1-4-rendu}
	\Delta^2_\kappa u_a=\lambda^4 u_a \  \ {\rm in}\ \ C_\kappa^n(a),
	\end{equation}
	and
	\begin{equation}\label{2-4-rendu}
	\Delta^2_\kappa u_b=\lambda^4 u_b \  \ {\rm in}\ \ C_\kappa^n(b),
	\end{equation}
	respectively. Elliptic regularity theory shows that $u_a\in \mathcal C^\infty(C_\kappa^n(a))$ and $u_b\in \mathcal C^\infty(C_\kappa^n(b))$.
	Using again \eqref{Euler-Lagrange-1} and the boundary condition \eqref{b-v-c-1}, it turns out that 
	\begin{equation}\label{Delta-relation}
	\Delta_\kappa u_a(\sqrt{\kappa}a)+\Delta_\kappa u_b(\sqrt{\kappa}b)=0.
	\end{equation}
The standard theory of ordinary differential equations shows that two of the four linearly independent solutions to the fourth-order equation \eqref{1-4-rendu} have singularities at the North pole $N\in \mathbb S_\kappa^n$; thus, the general non-singular solution to \eqref{1-4-rendu}  has the form 
\begin{align*}
u_a(x)&\coloneqq u_a(\theta)\\&=\cos^{2-n}\left(\frac{\theta}{2}\right)\left[A_1\mathcal F_{+}\left(\sin^2\left(\frac{\theta}{2}\right),\lambda,\kappa,n\right) +A_2\mathcal F_{-}\left(\sin^2\left(\frac{\theta}{2}\right),\lambda,\kappa,n\right)\right],
\end{align*}
 for every $x=x(\theta,\xi)\in C_\kappa^n(a)$ and some $A_1,A_2\in \mathbb R$, where  we have used the notations \eqref{Ferrers-fct}. 
%We notice that these functions appear in the proof of Theorem \ref{theorem-belt}/Step 1 for $n=2$. 
In a similar way, the non-singular solution to \eqref{2-4-rendu} in general form  is 
\begin{align*}
	u_b(x)&\coloneqq u_b(\theta)\\&=\cos^{2-n}\left(\frac{\theta}{2}\right)\left[B_1\mathcal F_{+}\left(\sin^2\left(\frac{\theta}{2}\right),\lambda,\kappa,n\right) +B_2\mathcal F_{-}\left(\sin^2\left(\frac{\theta}{2}\right),\lambda,\kappa,n\right)\right],
\end{align*}
for every $x=x(\theta,\xi)\in C_\kappa^n(b)$ and  some $B_1,B_2\in \mathbb R$.

Since $(u_a,u_b)\in\mathcal W_{a,b}(\Omega^\star)\setminus\{(0,0)\}$,  we have  that $u_a(\sqrt{\kappa}a)=u_b(\sqrt{\kappa}b)=0$; therefore, by \eqref{alfa-beta} we infer that
\begin{equation}\label{1-bvc}
A_1\mathcal F_{+}\left(\alpha,\lambda,\kappa,n\right) +A_2\mathcal F_{-}\left(\alpha,\lambda,\kappa,n\right)=0
\end{equation}
and
\begin{equation}\label{2-bvc}
B_1\mathcal F_{+}\left(\beta,\lambda,\kappa,n\right) +B_2\mathcal F_{-}\left(\beta,\lambda,\kappa,n\right)=0.
\end{equation}
Combining the boundary condition \eqref{b-v-c}  with \eqref{1-bvc}--\eqref{2-bvc}, it follows  that 
\begin{align}\label{3-bvc}
0=&~\nonumber (1-\alpha)\alpha^{\frac{n}{2}} \left(A_1\mathcal F_{+}'\left(\alpha,\lambda,\kappa,n\right) +A_2\mathcal F_{-}'\left(\alpha,\lambda,\kappa,n\right)\right)\\&- (1-\beta)\beta^{\frac{n}{2}}\left(B_1\mathcal F_{+}'\left(\beta,\lambda,\kappa,n\right) +B_2\mathcal F_{-}'\left(\beta,\lambda,\kappa,n\right)\right).
\end{align}
Since we have the pointwise equality 
$$\Delta_\kappa \left(\cos^{2-n}\left(\frac{\theta}{2}\right)\mathcal F_{\pm}\left(\sin^2\left(\frac{\theta}{2}\right),\lambda,\kappa,n)\right)\right)=\mp\lambda^2 \cos^{2-n}\left(\frac{\theta}{2}\right)\mathcal F_{\pm}\left(\sin^2\left(\frac{\theta}{2}\right),\lambda,\kappa,n\right),\ \theta\in (0,\pi),$$
by \eqref{Delta-relation} it follows that  
\begin{align}\label{4-bvc}
0=&~\nonumber (1-\alpha)^{1-\frac{n}{2}} \left(-A_1\mathcal F_{+}(\alpha,\lambda,\kappa,n) +A_2\mathcal F_{-}(\alpha,\lambda,\kappa,n)\right)\\&+ (1-\beta)^{1-\frac{n}{2}} \left(-B_1\mathcal F_{+}(\beta,\lambda,\kappa,n) +B_2\mathcal F_{-}(\beta,\lambda,\kappa,n)\right).
\end{align}
%
%	Since by the symmetrization construction $u_i$ turns to be cap-symmetric, i.e., $u_i$ depends only in $\theta$,  the spherical Laplacian of $u_i$ reduces to 
%$$\Delta_{g_0}u_i=(\sin\theta)^{1-n}\frac{d}{d \theta}\left((\sin\theta)^{n-1}\frac{d}{d \theta}u_i\right).$$ Accordingly, the solution of the fourth-order
%equation   \eqref{elso}, having no singularity at the origin, is 
%
%where $\nu=\frac{n}{2}-1$, $A_i,B_i\in \mathbb R$ are some constants, $i=1,2$, 
Since $A_1,A_2,B_1,B_2$ cannot be simultaneously zero, by using the notation $\mathcal F_\pm(\cdot)\coloneqq \mathcal F_{\pm}(\cdot,\lambda,\kappa,n)$,
  equations \eqref{1-bvc}--\eqref{4-bvc} necessarily imply that	
\begin{equation*}\label{determinant-sign-preserving-2}
{\rm det}\left[\begin{matrix}
\mathcal F_+(\alpha) & \mathcal F_-(\alpha) &0 & 0\\
0 & 0 &\mathcal F_+(\beta) & \mathcal F_-(\beta)\\
(1-\alpha)\alpha^{\frac{n}{2}}\mathcal F_+'(\alpha) & (1-\alpha)\alpha^{\frac{n}{2}}\mathcal F_-'(\alpha) &-(1-\beta)\beta^{\frac{n}{2}}\mathcal F_+'(\beta) & -(1-\beta)\beta^{\frac{n}{2}}\mathcal F_-'(\beta)\\
-(1-\alpha)^{1-\frac{n}{2}}\mathcal F_+(\alpha) & (1-\alpha)^{1-\frac{n}{2}}\mathcal F_-(\alpha) &-(1-\beta)^{1-\frac{n}{2}}\mathcal F_+(\beta) & (1-\beta)^{1-\frac{n}{2}}\mathcal F_-(\beta)
\end{matrix}\right]=0.
\end{equation*}
Expanding the determinant, we equivalently have	that 
	$$
	(1-\alpha)^{\frac{n}{2}}\alpha^{\frac{n}{2}}\left(\frac{\mathcal F'_-(\alpha)}{\mathcal F_-(\alpha)}-\frac{\mathcal F'_{+}(\alpha)}{\mathcal F_{+}(\alpha)}\right)+(1-\beta)^{\frac{n}{2}}\beta^{\frac{n}{2}}\left(\frac{\mathcal F'_-(\beta)}{\mathcal F_-(\beta)}-\frac{\mathcal F'_+(\beta)}{\mathcal F_+(\beta)}\right)=0,$$ 
	which is precisely equation \eqref{cross-product-equation}.
	 	\hfill $\square$\\

\section{Sharp spectral gaps on clamped spherical caps: proof of Theorem \ref{aszimptotak-CP-pozitiv}}\label{section-5}

In this section we prove Theorem \ref{aszimptotak-CP-pozitiv}, by establishing sharp growth estimates of the fundamental tone on the spherical cap $C_\kappa^n(L)$
in the two limit cases, i.e., when $L\to 0$ and $L\to \pi/\sqrt{\kappa},$  respectively. Before providing explicitly these 
estimates, we notice that the eigenfunctions on any spherical cap $C_\kappa^n(L)$ for the initial clamped problem \eqref{CP-problem-0} is of fixed sign, which follows by a Krein--Rutman argument and the sign-definite character of the solution to the Poisson-type biharmonic equation on $C_\kappa^n(L)$. 
%MAGYARAZAT MAGAMNAK: ez azert van igy, mert vennek a biharmonikus Poisson-Dirichlet egyenletet es ha vennenk a polar koordinatakat, megtudnank effektiv oldani ezt (valami specialis fgvek lennenek benne). Ez a fuggveny pozitiv vagy negativ lenne. Most alkalmaznank a Krein-Rutman elvet, amit Sweers-ek is hasznalnak es belatnank, hogy valojaban a clamped plate problemanak is pozitiv sajatfuggvenye lenne. 
%GREEN function:
%
%\textbf{http://www.physics.miami.edu/~curtright/650sphere.pdf}
%
%\textbf{http://www.physics.usu.edu/Wheeler/EM/Notes/EMNotes06GreenFunctionsFormal.pdf}
%We first deal with Theorem \ref{aszimptotak-CP-pozitiv} whose statements are strongly explored in the proof of Theorem \ref{fotetel-CP-pozitiv}. Accordingly, 
By the proof of  Theorem \ref{fundametal-tones-double},
these spherical cap symmetric eigenfunctions on $C_\kappa^n(L)$  are of the form 
$$
	U_{c,L}(x)=c\cos^{2-n}\left(\frac{\theta}{2}\right)\left(\mathcal F_{+}\left(\sin^2\left(\frac{\theta}{2}\right),\lambda_L,\kappa,n\right) -\frac{{\mathcal F_{+}(\alpha_L,\lambda_L,\kappa,n)}}{{\mathcal F_{-}(\alpha_L,\lambda_L,\kappa,n)}}\mathcal F_{-}\left(\sin^2\left(\frac{\theta}{2}\right),\lambda_L,\kappa,n\right)\right),
$$
for every $ x=x(\theta,\xi)\in C_\kappa^n(L)$ and  $c\in \mathbb R\setminus \{0\}$, see Figure \ref{4-abra},  the value $\lambda_L\coloneqq\Lambda^{\frac{1}{4}}_{\kappa}(C_\kappa^n(L))$ being the first positive zero  of the equation
\begin{equation}\label{K-egyenlo-nulla}
	\mathcal K_{\kappa,n}(\alpha_L,\lambda)\coloneqq \frac{\mathcal F_{-}'(\alpha_L,\lambda,\kappa,n)}{\mathcal F_{-}(\alpha_L,\lambda,\kappa,n)}-\frac{\mathcal F_{+}'(\alpha_L,\lambda,\kappa,n)}{\mathcal F_{+}(\alpha_L,\lambda,\kappa,n)}=0,
\end{equation}
  where $\alpha_L\coloneqq\sin^2\left(\frac{\sqrt{\kappa} L}{2}\right)$, while  $\mathcal F_{\pm}$ is defined in \eqref{Ferrers-fct}. 
%$$\mathcal F_{\pm}(t,\lambda_L,\kappa,n)\coloneqq {_2F}_1\left(\frac{1}{2}-\Lambda_\pm(\mu_L),\frac{1}{2}+\Lambda_\pm(\mu_L);\frac{n}{2};t\right),\ t\in (0,1),$$
%with
%\begin{equation}\label{notations}
%	\Lambda_\pm(\mu)=\sqrt{\frac{(n-1)^2}{4}\pm \mu} \ \ {\rm and}\ \ \mu_L=\frac{{\lambda_L^2}}{\kappa},	
%\end{equation}
%cf.\  relations \eqref{Ferrers-fct} and \eqref{lambda-pm}, respectively.

\subsection{Small spherical caps}\label{section-small-caps}  In the infinitesimal case $L\ll 1$, we assume that 
\begin{equation}\label{aszimptota-1}
	\lambda_L\sim \frac{C}{L}\ \ {\rm as}\ \  L\to 0,
%		\lambda_L\sim \sqrt{\frac{(n-1)^2}{4}\kappa+\frac{C^2}{L^2}}\ \ {\rm as}\ \  L\to 0.
\end{equation}
for some $C>0.$ 
By Proposition \ref{hipergeometrikus-Bessel-0-1} (with  settings $t={L}^2$, $x=\sqrt{\kappa}$ and $\mu=\frac{n}{2}-1$), on one hand, we obtain that
$$\lim_{L\to 0}\mathcal F_{\pm}(\alpha_L,\lambda_L,\kappa,n)=\Gamma\left(\frac{n}{2}\right)\left(\frac{2}{ C }\right)^{\frac{n}{2}-1}\left\{ \begin{array}{rll}
	J_{\frac{n}{2}-1}( C)   &{\rm for}&  \text{`}+\text{'}; \\
	I_{\frac{n}{2}-1}( C)    &{\rm for}&  \text{`}-\text{'}.
\end{array}\right.$$
On the other hand, the differentiation formula (\ref{F-differential}) and Proposition \ref{hipergeometrikus-Bessel-0-1} (with the choices $t={L}^2$, $x=\sqrt{\kappa}$ and $\mu=\frac{n}{2}$)  imply that 
$$\lim_{L\to 0} L^2 \mathcal F_{\pm}'(\alpha_L,\lambda_L,\kappa,n)=\Gamma\left(1+\frac{n}{2}\right)\left(\frac{2}{ C }\right)^{\frac{n}{2}}\left\{ \begin{array}{rll}
	J_{\frac{n}{2}}( C)   &{\rm for}&  \text{`}+\text{'}; \\
	-I_{\frac{n}{2}}( C)    &{\rm for}&  \text{`}-\text{'}.
\end{array}\right.$$
Since $\lambda_L$ satisfies equation  \eqref{K-egyenlo-nulla}, the above limits immediately imply that 
$\frac{J_{\frac{n}{2}}(C)}{J_{\frac{n}{2}-1}(C)}+\frac{I_\frac{n}{2}(C)}{I_{\frac{n}{2}-1}(C)}=0.$
Due to \eqref{Bessel-recurrence}, the latter equation is equivalent to 
$$\frac{J'_{\frac{n}{2}-1}(C)}{J_{\frac{n}{2}-1}(C)}-\frac{I'_{\frac{n}{2}-1}(C)}{I_{\frac{n}{2}-1}(C)}=0,$$
thus $C$ coincides with  the first positive zero $\mathfrak  h_{\frac{n}{2}-1}$ of the cross-product of Bessel functions.   According to  \eqref{aszimptota-1}, on has  the estimate 
%\begin{equation}\label{kicsi-eukl}
$	\Lambda_{\kappa}(C_\kappa^n(L))\sim\frac{\mathfrak  h_{\frac{n}{2}-1}^4}{L^4}\ \ {\rm as}\ \ L\to 0,$
%\end{equation}
which concludes
the proof of  \eqref{small-caps}.\\

The asymptotic estimate \eqref{small-caps} shows that the fundamental tone $\Lambda_{\kappa}(C_\kappa^n(L))$  has an Euclidean character over small scales, since 	$\Lambda_0(B_0(L))={\mathfrak  h_{\frac{n}{2}-1}^4}/{L^4}$ for every $L>0$, see e.g.\ \cite{A-B}. Table \ref{table-11} provides an insight into the accuracy of the   estimate \eqref{small-caps}  in  a few  dimensions. 
%; by scaling reasons, we present the values $\Lambda_{\kappa}(C_\kappa^n(L))^{1/4}$ by considering $\kappa=1$.  

\begin{table}[h!]
	\renewcommand{\arraystretch}{1.4}
\begin{tabular}{>{\centering}p{1.0cm}||>{\centering}p{2.1cm}|>{\centering}p{2.1cm}||>{\centering}p{2.1cm}|>{\centering}p{2.1cm}||>{\centering}p{2.1cm}|>{\centering}p{2.1cm}|}
	\hhline{~|-|-|-|-|-|-|}
	&
	\multicolumn{2}{c||}{\cellcolor{red!25}$n=2$}  & \multicolumn{2}{c||}{\cellcolor{green!25}$n=3$} & \multicolumn{2}{c|}{\cellcolor{teal!25}$n=4$} \tabularnewline
	\hline
	\multicolumn{1}{|p{1.07cm}||}
	{
		\cellcolor{gray!35}\phantom{$L$}   \phantom{$L$} \ \ \ $L$
	}
	&
	\cellcolor{red!5}Algebraic value of $\Lambda_{\kappa}^{\frac{1}{4}}(C_\kappa^2(L))$ & \cellcolor{red!10}Asymptotic estimate of $\Lambda_{\kappa}^{\frac{1}{4}}(C_\kappa^2(L))$ & \cellcolor{green!5}Algebraic value of $\Lambda_{\kappa}^{\frac{1}{4}}(C_\kappa^3(L))$ & \cellcolor{green!10}Asymptotic estimate of $\Lambda_{\kappa}^{\frac{1}{4}}(C_\kappa^3(L))$ & \cellcolor{teal!5}Algebraic value of $\Lambda_{\kappa}^{\frac{1}{4}}(C_\kappa^4(L))$ & \cellcolor{teal!10}Asymptotic estimate of $\Lambda_{\kappa}^{\frac{1}{4}}(C_\kappa^4(L))$\tabularnewline
	\hline
	\hline
	\multicolumn{1}{|c||}{\cellcolor{gray!30}$0.4$} & \cellcolor{red!5}7.9764 & \cellcolor{red!10}7.9906 & \cellcolor{green!5}9.7785 &\cellcolor{green!10}9.8165 &\cellcolor{teal!5}11.4588 & \cellcolor{teal!10}11.5272\tabularnewline
	\hline
	\multicolumn{1}{|c||}{\cellcolor{gray!25}$0.03$} & \cellcolor{red!5}106.5396 &\cellcolor{red!10}106.5406 &\cellcolor{green!5}130.8839 &\cellcolor{green!10}130.8867 & \cellcolor{teal!5}153.6915&\cellcolor{teal!10}153.6966 \tabularnewline
	\hline
	\multicolumn{1}{|c||}{\cellcolor{gray!20}$0.002$} & \cellcolor{red!5}1598.1102 & \cellcolor{red!10}1598.1103 & \cellcolor{green!5}1963.30097 &\cellcolor{green!10}1963.3011 & \cellcolor{teal!5}2305.4496 &\cellcolor{teal!10}2305.4501  \tabularnewline
	\hline
	\multicolumn{1}{|c||}{\cellcolor{gray!15}$ 0.0001$} & \cellcolor{red!5}31962.205 & \cellcolor{red!10}31962.206 & \cellcolor{green!5}39266.0231 & \cellcolor{green!10}39266.0232 & \cellcolor{teal!5}46108.9987 & \cellcolor{teal!10}46108.9988 \tabularnewline
	\hline
		\multicolumn{1}{|c||}{\cellcolor{gray!15}	$\vdots$} & \cellcolor{red!5}	$\vdots$ & \cellcolor{red!10}	$\vdots$ & \cellcolor{green!5}	$\vdots$ & \cellcolor{green!10}	$\vdots$ & \cellcolor{teal!5}	$\vdots$ & \cellcolor{teal!10}	$\vdots$ \tabularnewline
	\hline
	\multicolumn{1}{|c||}{\cellcolor{gray!10}$ L\to 0$} & \cellcolor{red!5}$\to +\infty$ & \cellcolor{red!10}$\to +\infty$ & \cellcolor{green!5}$\to +\infty$ & \cellcolor{green!10}$\to +\infty$ & \cellcolor{teal!5}$\to +\infty$ & \cellcolor{teal!10}$\to +\infty$ \tabularnewline
	\hline
\end{tabular}
	\renewcommand{\arraystretch}{1.4}
	\vspace{0.2cm}
\caption{Algebraic values and asymptotic estimates of the fourth root of the fundamental tone $\Lambda_{\kappa}(C_\kappa^n(L))$  for small clamped  spherical caps in dimensions 2, 3 and 4 ($\kappa=1$, for simplicity). The algebraic value  $\Lambda^\frac{1}{4}_{\kappa}(C_\kappa^n(L))$ is  the first positive zero  of equation  \eqref{K-egyenlo-nulla}, while the asymptotic estimate is given by \eqref{small-caps}. 
	%	The numbers are rounded up to four decimals.
}
\label{table-11}
\end{table}

\subsection{Large spherical caps}\label{subsection-5-2} In the sequel we will investigate the behavior of $\lambda_L>0$ as $L\to \pi/\sqrt{\kappa}$, see \eqref{K-egyenlo-nulla}, which has a  dimension-depending character.  

\subsubsection{The case $n\geq 4$.}
 On account of 
\eqref{F-differential} and \eqref{K-egyenlo-nulla}, 
the identity $\mathcal K_{\kappa,n}(\alpha_L,\lambda_L)=0$ can be rewritten into the equivalent form  
\begin{align}\label{egyenlet-atirva}
\nonumber 0=&~	\left(\frac{1}{4}-\Lambda_-^2(\lambda_L)\right)\frac{{_2F}_1\left(\frac{3}{2}-\Lambda_-(\lambda_L),\frac{3}{2}+\Lambda_-(\lambda_L);\frac{n+2}{2};\sin^2\left(\frac{\sqrt{\kappa} L}{2}\right)\right)}{{_2F}_1\left(\frac{1}{2}-\Lambda_-(\lambda_L),\frac{1}{2}+\Lambda_-(\lambda_L);\frac{n}{2};\sin^2\left(\frac{\sqrt{\kappa} L}{2}\right)\right)}\\&-\left(\frac{1}{4}-\Lambda_+^2(\lambda_L)\right)\frac{{_2F}_1\left(\frac{3}{2}-\Lambda_+(\lambda_L),\frac{3}{2}+\Lambda_+(\lambda_L);\frac{n+2}{2};\sin^2\left(\frac{\sqrt{\kappa} L}{2}\right)\right)}{{_2F}_1\left(\frac{1}{2}-\Lambda_+(\lambda_L),\frac{1}{2}+\Lambda_+(\lambda_L);\frac{n}{2};\sin^2\left(\frac{\sqrt{\kappa} L}{2}\right)\right)},
\end{align}
where we have used the notation  \eqref{lambda-pm}. 
We assume that $\lambda_L\to \lambda_0$ for some $\lambda_0\geq 0$ as $L\to \pi/\sqrt{\kappa}.$ If $n\geq 5$, we have that  $\frac{n+2}{2}-\left(\frac{3}{2}-\Lambda_\pm(\lambda_L)\right)-\left(\frac{3}{2}+\Lambda_\pm(\lambda_L)\right)=\frac{n-4}{2}>0$ and $\frac{n}{2}-\left(\frac{1}{2}-\Lambda_\pm(\lambda_L)\right)-\left(\frac{1}{2}+\Lambda_\pm(\lambda_L)\right)=\frac{n-2}{2}>0$, thus the asymptotic formula \eqref{singularity-1} applied to \eqref{egyenlet-atirva}  implies that $\lambda_0=0$. If $n=4$, by using a similar argument as above, the asymptotic formulas \eqref{singularity-1} and \eqref{singularity-2} applied to \eqref{egyenlet-atirva}  imply again that $\lambda_0=0$. Consequently,  $\Lambda_{\kappa}(C_\kappa^n(L))\to 0$ as $L\to \pi/\sqrt{\kappa}$ for every $n\geq 4.$ 

%\textbf{TABLAZAT (Algebrai es approximacios) n=4,5,6,10,15}

\begin{remark}\rm
	We note that when $n\in \{2,3\}$, a similar argument as in the case $n\geq 4$ can be formally applied to \eqref{egyenlet-atirva} via the asymptotic formula \eqref{singularity-3}; however, in both cases the asymptotic arguments lead us to an identity which looses any information on the behavior of $\Lambda_{\kappa}(C_\kappa^n(L))$ whenever $L\to \pi/\sqrt{\kappa}$. This phenomenon turns out to be  unsurprising, since the low-dimensional cases behave in a different manner with respect to the higher dimensional counterparts. 
\end{remark}

%Therefore, $\Lambda_{\kappa}(C_\kappa^n(L))\to 0$ as $L\to \pi/\sqrt{\kappa}$ and $n\geq 5.$

\subsubsection{The case $n=3$.} We first establish an elementary form of  \eqref{K-egyenlo-nulla} which is valid only in the 3-dimensional case. By relation (15.4.16) of Olver, Lozier,  Boisvert  and Clark \cite{Digital}, we have that  
\begin{equation}\label{zart-formula}
	{_2F}_1\left(\frac{1}{2}-C,\frac{1}{2}+C;\frac{3}{2};t\right)=\frac{\sin\left(2C\arcsin(\sqrt{t})\right)}{2C\sqrt{t}},\ \forall C\in \mathbb C\setminus \{0\}, t\in (0,1).
\end{equation}
Therefore, if we use the notations $$\Lambda_\pm\coloneqq \Lambda_\pm(\lambda,\kappa,3)=\sqrt{1\pm\frac{\lambda^2}{\kappa}}\ \ {\rm and}\ \  \widetilde \Lambda_-\coloneqq \mathfrak{i}\Lambda_-(\lambda,\kappa,3)= \sqrt{\frac{\lambda^2}{\kappa}-1},$$
see \eqref{lambda-pm}, one has that
$$\mathcal F_{-}(t,\lambda,\kappa,3)=\left\{ \begin{array}{@{\,}l@{~}l@{~}l@{\,}}\def\arraystretch{1.3}
	\dfrac{\sinh\left(2\widetilde \Lambda_- \arcsin (\sqrt{t})\right)}{2\widetilde \Lambda_- \sqrt{t}} &\mbox{if} &  \lambda>\sqrt{\kappa}; \\
	\dfrac{\arcsin \left(\sqrt{t}\right)}{\sqrt{t}} &\mbox{if} &  \lambda=\sqrt{\kappa}; \\
	\dfrac{\sin\left(2 \Lambda_- \arcsin (\sqrt{t})\right)}{2 \Lambda_- \sqrt{t}} &\mbox{if} &  \lambda<\sqrt{\kappa}; \\
\end{array}\right.\def\arraystretch{1}
\ \ {\rm and}\ \ \mathcal F_{+}(t,\lambda,\kappa,3)=\frac{\sin\left(2\Lambda_+ \arcsin (\sqrt{t})\right)}{2 \Lambda_+ \sqrt{t}}.$$
%where 
%\begin{equation}\label{lambdak}
%\widetilde \Lambda_-\coloneqq i\Lambda_-= 2\sqrt{\frac{\lambda^2}{\kappa^2}-1}\ \ {\rm and}\ \ \widetilde \Lambda_+\coloneqq \Lambda_+=2\sqrt{\frac{\lambda^2}{\kappa^2}+1}.
%\end{equation}
%Due to (\ref{lambda-minel-nagyobb}), it turns out that $\widetilde \Lambda_\pm\in \mathbb R.$
Accordingly, by using (\ref{K-definicio}),   we have for every $ t\in (0,1)$ that
\begin{equation}\label{explicit3dimenzioban}\def\arraystretch{1.5}
	\mathcal K_{\kappa,3}(t,\lambda)={ \dfrac{1}{\sqrt{t(1-t)}}\cdot
		\left\{ \begin{array}{@{\,}l@{~}l@{~}l@{\,}}
			{\widetilde \Lambda_-\coth\left(2\widetilde \Lambda_-\arcsin (\sqrt{t})\right)- \Lambda_+\cot\left(2 \Lambda_+\arcsin (\sqrt{t})\right)} &\mbox{if} &  \lambda>\sqrt{\kappa}; \\
			\dfrac{1}{2\arcsin \left(\sqrt{t}\right)}-\sqrt{2}\cot\left(2\sqrt{2}\arcsin (\sqrt{t})\right) &\mbox{if} &  \lambda=\sqrt{\kappa}; \\
			{\Lambda_-\cot\left(2 \Lambda_-\arcsin (\sqrt{t})\right)- \Lambda_+\cot\left(2 \Lambda_+\arcsin (\sqrt{t})\right)}  &\mbox{if} &  \lambda<\sqrt{\kappa}. \\
		\end{array}\right.}
	\def\arraystretch{1}
\end{equation}

We are ready to investigate the behavior of $\Lambda_{\kappa}(C_\kappa^3(L))$ as $L\to \pi/\sqrt{\kappa}$. By contradiction, we assume that $\lambda_L\leq  \sqrt{\kappa}$, thus $\mu_L\coloneqq{\lambda_L^2}/{\kappa}\in (0,1].$  First, when $\lambda_L<  \sqrt{\kappa}$, by \eqref{explicit3dimenzioban} 
the identity $\mathcal K_{\kappa,3}(\alpha_L,\lambda_L)=0$ is equivalent to 
\begin{equation}\label{egyenlet-nincs-megoldas}
	\sqrt{1-\mu_L}\cot \big(\sqrt{1-\mu_L}\sqrt{\kappa} L\big)-\sqrt{1+\mu_L}\cot \big(\sqrt{1+\mu_L}\sqrt{\kappa} L\big)=0.
\end{equation}
We claim that  \eqref{egyenlet-nincs-megoldas} has no solution in $\mu_L$ for any $L\in (0,\pi/\sqrt{\kappa})$. On one hand, if $\sqrt{1+\mu_L}\sqrt{\kappa} L < \pi$, then by the monotonicity of $s\mapsto s\cot(s)$ on $(0,\pi)$, we have that $$\sqrt{1-\mu_L}\cot \big(\sqrt{1-\mu_L}\sqrt{\kappa} L\big)>\sqrt{1+\mu_L}\cot \big(\sqrt{1+\mu_L}\sqrt{\kappa} L\big).$$ On the other hand, if $\sqrt{1+\mu_L}\sqrt{\kappa} L >\pi$, since $\mu_L<1$ and $\sqrt{\kappa} L<\pi$,  the monotonicity of $s\mapsto s\cot(s)$ on $(\pi,2\pi)$ and on $(0,\pi)$, respectively, implies that 
\begin{align*}
	\sqrt{1+\mu_L}\sqrt{\kappa} L\cot \big(\sqrt{1+\mu_L}\sqrt{\kappa} L\big)&\geq \sqrt{2}\pi\cot \big(\sqrt{2}\pi\big)\approx 1.2272\\&>1=\lim_{s\to 0}\sqrt{1-\mu_L}s\cot\big (\sqrt{1-\mu_L}s\big)\\&\geq  \sqrt{1-\mu_L}\sqrt{\kappa} L\cot \big(\sqrt{1-\mu_L}\sqrt{\kappa} L\big).
\end{align*}
The above estimates conclude the claim together with the  limit cases, i.e., 
\begin{itemize}
	\item[$\bullet$] $\sqrt{1+\mu_L}\sqrt{\kappa} L =\pi$, when the left hand side of \eqref{egyenlet-nincs-megoldas} blows up; and 
	\item[$\bullet$] $\lambda_L=\sqrt{\kappa}$, when $\mathcal K_{\kappa,3}(\alpha_L,\lambda_L)=0$ reduces (due to \eqref{explicit3dimenzioban}) to an incompatible relation. 
\end{itemize}
Consequently, the only possible case when  $\mathcal K_{\kappa,3}(\alpha_L,\lambda_L)=0$ might hold is when $\lambda_L>  \sqrt{\kappa}$, obtaining by \eqref{explicit3dimenzioban} that
$$\sqrt{\mu_L-1}\coth \big(\sqrt{\mu_L-1}\sqrt{\kappa} L\big)-\sqrt{1+\mu_L}\cot \big(\sqrt{1+\mu_L}\sqrt{\kappa} L\big)=0,$$
where $\mu_L={\lambda_L^2}/{\kappa}>1$. Since $\sqrt{\kappa} L\to \pi$ and we may consider $\mu_L\to\mu$ for some $\mu\geq 1$, the latter equation reduces to 
$$\sqrt{\mu-1}\coth \big(\pi\sqrt{\mu-1} \big)-\sqrt{1+\mu}\cot \big(\pi\sqrt{1+\mu} \big)=0,$$
whose first positive zero is $\mu_3\coloneqq \mu\approx 1.0277$. Therefore, $$\Lambda_{\kappa}(C_\kappa^3(L))=\lambda_L^4=\mu_L^2\kappa^2\to \mu_3^2\kappa^2 \ \ {\rm as}\ \ L \to \frac{\pi}{\sqrt{\kappa}}.$$

\subsubsection{The case $n=2$.} In this special case, by using a simple relationship between the Gaussian hypergeometric and Legendre functions 
%are connected as $${\bf P}_\nu^0(t)={_2F}_1\left(1+\nu,-\nu;1;\frac{1-t}{2}\right)\in \mathbb R,\ \ t\in (-1,1),$$
(see \eqref{P-hypergoemetric} for $\mu=0$), the identity  $\mathcal K_{\kappa,2}(\alpha_L,\lambda_L)=0$ is equivalent to 
\begin{equation}\label{egyenlet-2-dim-a}
\left.\frac{{\rm d}}{{\rm d} t}\ln\frac{\LP_{\nu_-(\lambda_L)}^{0}(1-2t)}{\LP_{\nu_+(\lambda_L)}^{0}(1-2t)}\right|_{t=\alpha_L}=0,
\end{equation}
where  $\nu_\pm(\lambda)=\Lambda_\pm(\lambda)-\frac{1}{2}=\sqrt{\frac{1}{4}\pm \frac{\lambda^2}{\kappa}}-\frac{1}{2}$.  
The derivation formula  \eqref{derivation-legendre} transforms the equation \eqref{egyenlet-2-dim-a} into 
\begin{align}\label{poz-neg}
	\nonumber 0=&~\nu_-(\lambda_L)\left(\frac{\LP_{\nu_-(\lambda_L)-1}^{0}\left(\cos\left(\sqrt{\kappa}L\right)\right)}{\LP_{\nu_-(\lambda_L)}^{0}\left(\cos(\sqrt{\kappa}L)\right)}-\cos\left(\sqrt{\kappa}L\right)\right)\\&-\nu_+(\lambda_L)\left(\frac{\LP_{\nu_+(\lambda_L)-1}^{0}\left(\cos(\sqrt{\kappa}L)\right)}{\LP_{\nu_+(\lambda_L)}^{0}\left(\cos(\sqrt{\kappa}L)\right)}-\cos\left(\sqrt{\kappa}L\right)\right).
\end{align}
By \eqref{derivative-Legendre} and \eqref{A-expression}, it turns out that for every fixed $s\in (-1,1)$ the function $t\mapsto t\left(\frac{\LP_{t-1}^{0}(s)}{\LP_{t}^{0}(s)}-s\right)$ is increasing on $\big(-\frac{1}{2},\frac{\sqrt{2}-1}{2}\big)$; in particular, one has the inequality
$$\nu_-(\lambda_L)\left(\frac{\LP_{\nu_-(\lambda_L)-1}^{0}\left(\cos(\sqrt{\kappa}L)\right)}{\LP_{\nu_-(\lambda_L)}^{0}\left(\cos(\sqrt{\kappa}L)\right)}-\cos\left(\sqrt{\kappa}L\right)\right)<\nu_+(\lambda_L)\left(\frac{\LP_{\nu_+(\lambda_L)-1}^{0}\left(\cos(\sqrt{\kappa}L)\right)}{\LP_{\nu_+(\lambda_L)}^{0}\left(\cos(\sqrt{\kappa}L)\right)}-\cos\left(\sqrt{\kappa}L\right)\right)$$
for every $L\in \big(0,{\pi}/{\sqrt{\kappa}}\big)$ and $\lambda_L\in \big(0,{\sqrt{\kappa}}/{2}\big).$
Thus, equation \eqref{poz-neg} has no solution whenever $(L,\mu_L)\in  \big(0,{\pi}/{\sqrt{\kappa}}\big)\times \big(0,{\sqrt{\kappa}}/{2}\big).$ In particular, we necessarily have $\lambda_L\geq {\sqrt{\kappa}}/{2}$ for every  $L\in  \big(0,{\pi}/{\sqrt{\kappa}}\big)$, and  $\nu_-(\lambda_L)\in \mathbb C\setminus \mathbb R$ when  $\lambda_L> {\sqrt{\kappa}}/{2}$.
% We are going to compute  the limit of $\mu_L$ whenever $L\to \pi/\sqrt{\kappa}$.

\begin{figure}[t!]
	\centering
	\includegraphics[scale=1]{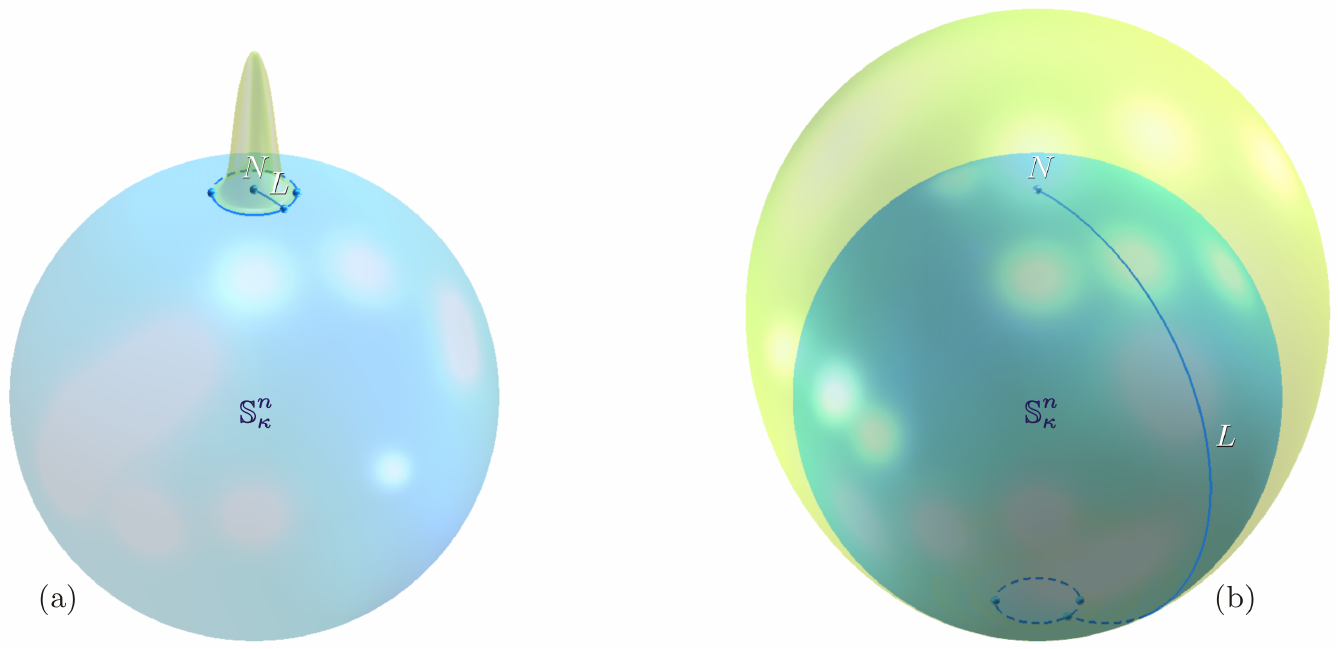}
	\caption{The shape of the first eigenfunction (light green) on  (a) small and (b) large  vibrating clamped spherical caps. The fundamental tone in (a) behaves as in the Euclidean case (cf.\ \S \ref{section-small-caps}), while in (b) it is dimensional-dependent (cf.\ \S \ref{subsection-5-2}).  
		%  PROBA: $\mathfrak  g_{{\nu},1}(\widetilde L_0)$ 
	}\label{4-abra}
\end{figure}

As $-1$ is a singularity  in \eqref{poz-neg} whenever $L\to \pi/\sqrt{\kappa}$, the symmetrization formula \eqref{inversion-P} and the behavior at the singularity $1$ of the Legendre functions  \eqref{singular-Q} yield -- after an asymptotic argument in \eqref{poz-neg} -- that 
$$\sin((\nu_-(\lambda_0)-\nu_+(\lambda_0))\pi)+\frac{2}{\pi}\sin(\nu_-(\lambda_0)\pi)\sin(\nu_+(\lambda_0)\pi)\left(\Psi(\nu_+(\lambda_0)+1)-\Psi(\nu_-(\lambda_0)+1)\right)=0,$$
where $\lambda_L\to \lambda_0$ as $L\to \pi/\sqrt{\kappa}$  for some $\lambda_0\geq {\sqrt{\kappa}}/{2}.$ We equivalently  transform the latter equation into 
\begin{equation}\label{tan-2-dim}
\tan\big(\Lambda_-(\lambda_0)\pi\big)-\tan\big(\Lambda_+(\lambda_0)\pi\big)+\frac{2}{\pi}\left(\Psi\left(\Lambda_+(\lambda_0)+\frac{1}{2}\right)-\Psi\left(\Lambda_-(\lambda_0)+\frac{1}{2}\right)\right)=0.	
\end{equation}
%Since the function $t\mapsto \tan(t\pi)-\frac{2}{\pi}\Psi\left(t+\frac{1}{2}\right)$ is increasing on $(0,\frac{1}{2})$, it turns out that \eqref{tan-2-dim} has no solution on $(0,\frac{1}{4}).$
Due to \eqref{digamma-prop-2}, the imaginary part of \eqref{tan-2-dim} vanishes; thus,  $\mu_2\coloneqq{\lambda_0^2}/{\kappa}\geq \frac{1}{4}$ is the first positive zero of 
$$	\frac{\pi}{2}\tan\left(\pi\sqrt{\frac{1}{4}+\mu}\right)-\Psi\left(\sqrt{\frac{1}{4}+\mu}+\frac{1}{2}\right)+\Re \Psi\left(\sqrt{\frac{1}{4}-\mu}+\frac{1}{2}\right)=0,\ \ \mu\geq \frac{1}{4}.$$
In fact, one has that $\mu_2\approx 0.9125$, thus $$\Lambda_{\kappa}(C_\kappa^2(L))=\lambda_L^4\to \lambda_0^4=\mu_2^2\kappa^2\ \ {\rm as}\ \ L \to \frac{\pi}{\sqrt{\kappa}},$$
which concludes the proof of \eqref{large-caps}. \\

 Table \ref{table-1}  presents the numerical  behavior of the fundamental tone $\Lambda_{\kappa}(C_\kappa^n(L))$ in some dimensions whenever $L\to \pi/\sqrt{\kappa}$.

	\renewcommand{\arraystretch}{1.3}	
\providecommand{\tabularnewline}{\\}
%	\begin{center}
\begin{table}[H]
	\centering
	\begin{tabular}{|c||>{\centering}p{2.7cm}|>{\centering}p{2.5cm}|>{\centering}p{1.0cm}|>{\centering}p{1.7cm}|>{\centering}p{1.7cm}|>{\centering}p{1.9cm}|>{\centering}p{1.0cm}|>{\centering}p{1.0cm}|>{\centering}p{1.0cm}|>{\centering}p{1.0cm}|>{\centering}p{1.0cm}|}
		\hline
		\cellcolor{teal!35}
		%\cline{2-5}
%		\diagbox[height=0.9cm]{$~L~~$}{$~n$}
		            \diagbox[width=\dimexpr \textwidth/8+2\tabcolsep\relax, height=0.8cm]{\ \ \ \ \\ \ \ \  $L$ }{ $n$\ \ \ }
		&
		\cellcolor{green!30}2
		&
		\cellcolor{green!25}
		3
		&
		\cellcolor{green!20}
		4
		&
		\cellcolor{green!15}
		5
		&
		\cellcolor{green!10}
		6
		&
		\cellcolor{green!5}
		
		7
		
		\tabularnewline
		\hline\hline
		\rowcolor{gray!30}
		\cellcolor{red!5}
		%        $0.9\pi$ & 1.1929 & 1.7907 & 2.1308&2.1503&1.876&1.4237\tabularnewline
		%        \hline
		$0.99\pi$ & 0.8437 &1.1091& 0.8018 & $2.32\cdot 10^{-1}$ &$2.91\cdot 10^{-2}$ & $2.35\cdot 10^{-3}$
		\tabularnewline
		\hline
		\rowcolor{gray!25}
		\cellcolor{red!10}
		$0.999\pi$ & 0.8332 &1.0612& 0.4978&$2.39\cdot 10^{-2}$&$2.96\cdot 10^{-4}$&$2.36\cdot 10^{-6}$
		\tabularnewline
		\hline
		\rowcolor{gray!20}
		\cellcolor{red!15}
		$0.9999\pi$ & 0.8328 &1.05662& 0.3605 &$2.3\cdot 10^{-3}$ & $2.95\cdot 10^{-7}$& $2.35\cdot 10^{-9}$
		\tabularnewline
		\hline
		\rowcolor{gray!15}
		\cellcolor{red!20}
		$0.99999\pi$ & 0.83277 &1.05661& 0.2798& $2.16\cdot 10^{-4}$ &$2.4\cdot 10^{-8}$& $1.72\cdot 10^{-12}$
		\tabularnewline
		\hline
		\rowcolor{gray!10}
		\cellcolor{red!25}
		$\vdots$ & $\vdots$ &$\vdots$& $\vdots$ & $\vdots$ & $\vdots$ & $\vdots$
		\tabularnewline
		\hline
		\rowcolor{gray!5}
		\cellcolor{red!30}
		$\to \pi$ & $\to\mu_2^2\approx 0.83274$ &$\to\mu_3^2\approx 1.0561$& $\to 0$ & $\to 0$ & $\to 0$ & $\to 0$
		\tabularnewline
		\hline
	\end{tabular}
	\vspace*{0.2cm}
	\caption{Behavior of the fundamental tone $\Lambda_{\kappa}(C_\kappa^n(L))$ in dimensions $n\in \{2,\ldots,7\}$ for large  spherical caps (i.e.,\ $L\to \pi$ and $\kappa =1$, for simplicity);  $\mu_2$ and $\mu_3$ are the smallest positive zeros of  equations \eqref{transzdendental-2-dimenzio} and \eqref{transzdendental-3-dimenzio}, respectively. 
		%The numbers are rounded to maximum five decimals.
		%		Lord Rayleigh's conjecture holds on any compact $n$-dimensional Riemannian manifold with ${\rm Ric}(M,g)\geq (n-1)\kappa>0$  for any clamped domain $\Omega\subset M$ verifying $V_g(\Omega)>v_{n,\kappa}V_g(M)$.} 
	}	\label{table-1}
\end{table}

\vspace{-0.5cm}

\section{Lord Rayleigh's conjecture: proof of Theorem \ref{fotetel-CP-pozitiv}}\label{section-6}

This section is devoted to the proof of Lord Rayleigh's conjecture on positively curved spaces. Based on  Theorem \ref{fundametal-tones-double}, first we further reduce the conjecture to the validity of an algebraic inequality (see \S \ref{subsection-reduction}),  then we conclude the proof (see \S \ref{high-dimension})  by using the  sharp growth estimates of the fundamental tone of  spherical caps (see \S \ref{section-5}). 
% Using Theorem \ref{fundametal-tones-double}, in the next subsection we first present a preparatory part to do this.  
% zeros of Ferrers functions and their cross-products

\subsection{Reduction to eigenvalue comparison on `half-caps'}\label{subsection-reduction} We first need a monotonicity result that plays an important role in the proof of Theorem \ref{fotetel-CP-pozitiv}; the notations are the same as before. 

\begin{proposition}\label{monotonicity-K}
	If $\kappa>0,$ $n\in \mathbb N_{\geq 2}$ and $t\in (0,1)$ are fixed,  the function $\lambda\mapsto \mathcal K_{\kappa,n}(t,\lambda)$ is increasing on $(0,\infty)$ between any two consecutive zeros of $\mathcal F_{+}(t,\cdot,\kappa,n)$. Moreover, $\lim_{\lambda\to 0}\mathcal K_{\kappa,n}(t,\lambda)=0.$
\end{proposition}

\begin{proof}
Let $\kappa>0$ and $t\in (0,1)$ be fixed. For $n=3$, the proof is trivial due to relation \eqref{explicit3dimenzioban}; indeed,  a direct computation yields that   $\lambda\mapsto \mathcal K_{\kappa,3}(t,\lambda)$ is increasing on $(0,\infty)$ between any two consecutive zeros of $\mathcal F_{+}(t,\cdot,\kappa,3),$ which have the explicit closed form
	\begin{equation}\label{g-sajatertekek}
		\mathfrak  f_{\kappa,3,m}(t)=\sqrt{\kappa\left(\left(\frac{m\pi}{2\arcsin (\sqrt{t})}\right)^2-1\right)},\ \ m\in \mathbb N_{\geq 1}.
	\end{equation}
	
When $n\neq 3$, 
%although a similar monotonicity property holds as above, 
another approach is needed as no closed formula is available similar to \eqref{zart-formula} (and \eqref{explicit3dimenzioban}). In fact, the strategy is to `replace' the monotonicity of $\lambda\mapsto \mathcal K_{\kappa,n}(t,\lambda)$ with that of the ratio of hypergeometric-type functions with respect to the variable $t\in (0,1)$. 
	%, reducing in this way the degree of complexity of our task. 
	In the sequel we  consider those pairs $(t,\lambda)$  in the open domain for which $\mathcal F_{+}(t,\lambda,\kappa,n)\neq 0$.
	In addition, we first assume that $\Lambda_{\pm}-\frac{1}{2}\coloneqq\Lambda_{\pm}(\lambda,\kappa,n)-\frac{1}{2}\notin  \mathbb Z$ and $\Lambda_-\neq 0,$  see \eqref{lambda-pm} for $\Lambda_{\pm}$. The analyticity of  Gaussian hypergeometric functions together with relations   (\ref{P-hypergoemetric}) and \eqref{derivative-Legendre} imply that 
	\begin{align*}
		\frac{\partial}{\partial \lambda}\mathcal K_{\kappa,n}(t,\lambda)&=\frac{\partial}{\partial \lambda}\left(\frac{\partial}{\partial t}\ln\frac{\mathcal F_{-}(t,\lambda,\kappa,n)}{\mathcal F_{+}(t,\lambda,\kappa,n)}\right)=\frac{\partial}{\partial t}\left(\frac{\partial}{\partial \lambda}\ln\frac{\mathcal F_{-}(t,\lambda,\kappa,n)}{\mathcal F_{+}(t,\lambda,\kappa,n)}\right)\\&=\frac{\partial}{\partial t}\frac{\partial}{\partial \lambda}\ln\frac{\LP_{\Lambda_--\frac{1}{2}}^{1-\frac{n}{2}}(1-2t)}{\LP_{\Lambda_+-\frac{1}{2}}^{1-\frac{n}{2}}(1-2t)}=\frac{\partial}{\partial t}\frac{\frac{\partial}{\partial \lambda}\LP_{\Lambda_--\frac{1}{2}}^{1-\frac{n}{2}}(1-2t)}{\LP_{\Lambda_--\frac{1}{2}}^{1-\frac{n}{2}}(1-2t)}-\frac{\partial}{\partial t}\frac{\frac{\partial}{\partial \lambda}\LP_{\Lambda_+-\frac{1}{2}}^{1-\frac{n}{2}}(1-2t)}{\LP_{\Lambda_+-\frac{1}{2}}^{1-\frac{n}{2}}(1-2t)}\\&=\frac{\lambda}{\kappa \pi}\left(\frac{1}{\Lambda_-}\frac{\partial}{\partial t}\frac{\LA_{\Lambda_--\frac{1}{2}}^{1-\frac{n}{2}}(1-2t)}{\LP_{\Lambda_--\frac{1}{2}}^{1-\frac{n}{2}}(1-2t)}+\frac{1}{\Lambda_+}\frac{\partial}{\partial t}\frac{\LA_{\Lambda_+-\frac{1}{2}}^{1-\frac{n}{2}}(1-2t)}{\LP_{\Lambda_+-\frac{1}{2}}^{1-\frac{n}{2}}(1-2t)}\right).
	\end{align*}
	Therefore,  it is enough to prove that the function $t\mapsto \frac{1}{\Lambda_\pm}\frac{\LA_{\Lambda_\pm-\frac{1}{2}}^{1-\frac{n}{2}}(1-2t)}{\LP_{\Lambda_\pm-\frac{1}{2}}^{1-\frac{n}{2}}(1-2t)},$ $t\in (0,1)$ is increasing in  the aforementioned domain. Using the formula $\sin(\nu\pi)\Gamma(-\nu)\Gamma(\nu+1)=-\pi$ for every $\nu\in \mathbb C\setminus \mathbb Z$ (see \eqref{Gamma-special} with suitable choices),  relations \eqref{P-hypergoemetric} and \eqref{A-expression} imply that
	$$\frac{1}{\Lambda_\pm}\frac{\LA_{\Lambda_\pm-\frac{1}{2}}^{1-\frac{n}{2}}(1-2t)}{\LP_{\Lambda_\pm-\frac{1}{2}}^{1-\frac{n}{2}}(1-2t)}=-\pi \cdot\frac{\displaystyle\sum_{m=0}^\infty \alpha_m^\pm t^m}{\displaystyle\sum_{m=0}^\infty \beta_m^\pm t^m}=:-\pi\cdot\frac{v_\pm(t)}{w_\pm(t)},$$
	where the coefficients are
	$$\alpha_m^\pm =\beta_m^\pm \frac{\Psi(m+\frac{1}{2}+\Lambda_{\pm})-\Psi(m+\frac{1}{2}-\Lambda_{\pm})}{\Lambda_{\pm}} \ \ {\rm and }\ \  \beta_m^\pm =\frac{(\frac{1}{2}-\Lambda_{\pm})_m(\frac{1}{2}+\Lambda_{\pm})_m}{m!(\frac{n}{2})_m},\ m\in \mathbb N,$$
	 are well-defined. 
Consequently, the claimed monotonicity of  $\lambda\mapsto \mathcal K_{\kappa,n}(t,\lambda)$ reduces to the decreasing character of  $t\mapsto \frac{w_\pm(t)}{v_\pm(t)}$, $t\in (0,1)$,  %  between any two consecutive zeros of $v_+$, 
which follows from Proposition \ref{Ferrers-basic-lemma-1}/(v). When $\Lambda_-=0$, a limit is considered, by obtaining the equality $\alpha_m^\pm=2\beta_m^\pm\Psi\left(1,m+\frac{1}{2}\right)$, and a similar proof applies as above. Finally, when  $\Lambda_--\frac{1}{2}\in  \mathbb Z$ or $\Lambda_+-\frac{1}{2}\in  \mathbb Z$, another  argument is needed, where the discussion is even simpler, since  the corresponding series can be reduced to some polynomials. 

%If either $\Lambda_+=\Lambda_+(\lambda,\kappa,n)= l-\frac{1}{2}$ or $\Lambda_-=\Lambda_-(\lambda,\kappa,n)= l-\frac{1}{2}$  for some $l\in \mathbb Z,$ a limiting  argument applies; in this case, the expressions in $v_\pm$ and $w_\pm$ reduce to polynomials instead of function series. 

The fact that $\lim_{\lambda\to 0}\mathcal K_{\kappa,n}(t,\lambda)=0$ directly follows from \eqref{lambda-pm}--\eqref{K-definicio}, which ends the proof.  
\end{proof}

From now on, we focus on the proof of Lord Rayleigh's conjecture for positively curved vibrating clamped plates. To this end, let  $(M,g)$ be a compact $n$-dimensional Riemannian manifold with ${\sf Ric}_{(M,g)}\geq (n-1)\kappa>0$ and  consider the non-empty smooth domain  $\Omega\subset M$ with its normalized rearrangement $\Omega^\star\subset \mathbb S^n_\kappa$, i.e.,     
\begin{equation}\label{terfogatok-egyenloek}
	\frac{V_g(\Omega)}{V_g(M)}=\frac{V_{\kappa}(\Omega^\star)}{V_{\kappa}(\mathbb S^n_\kappa)}.
\end{equation}
Using the notation  \eqref{Lambda-ertelmezes},  Lord Rayleigh's conjecture is proved once we show that
\begin{equation}\label{amit-kell-igazolni}
	\Lambda(\kappa,n,a,b)\geq \Lambda(\kappa,n,0,L)
\end{equation}
 for every $a,b\geq 0$ with $	V_\kappa(C_\kappa^n(a))+V_\kappa(C_\kappa^n(b))=V_\kappa(\Omega^\star)=V_\kappa(C_\kappa^n(L)),$ see \eqref{terfogatok-ujra}. Indeed, if \eqref{amit-kell-igazolni} holds, by Theorem \ref{fundametal-tones-double} we have that
\begin{equation}\label{atirva-amit-kell}
\Lambda_g(\Omega)\geq \Lambda(\kappa,n,a,b)\geq \Lambda(\kappa,n,0,L)=\Lambda_\kappa(C_\kappa^n(L))=\Lambda_\kappa(\Omega^\star),	
\end{equation}
which is precisely the required inequality \eqref{foegyenlotlenseg}.

According to the statement of Theorem \ref{fundametal-tones-double},  inequality \eqref{amit-kell-igazolni} is equivalent to
 \begin{equation}\label{amit-kell-igazolni-2}
 	\lambda_{\kappa,n}(\alpha,\beta)\geq \lambda_{\kappa,n}(0,\alpha_L),
 \end{equation}
where $\lambda_{\kappa,n}(\alpha,\beta)>0$ is the smallest positive zero of the equation  \eqref{cross-product-equation},  $\alpha_L=\sin^2\left(\frac{\sqrt{\kappa} L}{2}\right)$, while $\alpha,\beta\in (0,1)$ are arbitrarily chosen such that  
\begin{equation}\label{alfbet}
	\int_0^{\frac{2}{\sqrt{\kappa}}\sin^{-1}(\sqrt{\alpha})} \sin(\sqrt{\kappa}\rho)^{n-1}{\rm d}\rho + \int_0^{\frac{2}{\sqrt{\kappa}}\sin^{-1}(\sqrt{\beta})} \sin(\sqrt{\kappa}\rho)^{n-1}{\rm d}\rho =\int_0^L \sin(\sqrt{\kappa}\rho)^{n-1}{\rm d}\rho,
\end{equation}
see \eqref{terfogatok-ujra} and \eqref{alfa-beta}.

Without loss of generality, we may choose $\alpha\leq \beta$ that verify \eqref{alfbet}. In view of Proposition \ref{monotonicity-K}, the function 
$$\mathcal S_{\kappa,n}(\alpha,\beta,\lambda)=  (1-\alpha)^{\frac{n}{2}}\alpha^{\frac{n}{2}}\mathcal K_{\kappa,n}(\alpha,\lambda)+(1-\beta)^{\frac{n}{2}}\beta^{\frac{n}{2}}\mathcal K_{\kappa,n}(\beta,\lambda)$$ inherits the properties of $\mathcal K_{\kappa,n}$, i.e., $\lambda\mapsto \mathcal S_{\kappa,n}(\alpha,\beta,\lambda)$ is increasing on $(0,\infty)$ between any two consecutive poles of $\mathcal S_{\kappa,n}(\alpha,\beta,\cdot)$ and $\lim_{\lambda\to 0}\mathcal S_{\kappa,n}(\alpha,\beta,\lambda)=0.$ In particular, if we denote the sequence of zeros of the Gaussian hypergeometric function $\mathcal F_{+}(t,\cdot,\kappa,n)$ by $(\mathfrak  f_{\kappa,n,m}(t))_m$ (cf.\ Proposition \ref{Ferrers-basic-lemma-1}/(iii)), it turns out that the first positive zero $\lambda_{\kappa,n}(\alpha,\beta)$  of $\mathcal S_{\kappa,n}(\alpha,\beta,\cdot)$ will be situated between $\mathfrak  f_{\kappa,n,1}(\beta)$ and $\mathfrak  f_{\kappa,n,1}(\alpha)$; more precisely, 
\begin{equation}\label{fogo-alfa-beta}
	\mathfrak  f_{\kappa,n,1}(\beta)\leq \lambda_{\kappa,n}(\alpha,\beta)\leq \min\{\mathfrak  f_{\kappa,n,1}(\alpha),\mathfrak  f_{\kappa,n,2}(\beta)\}.
\end{equation}

We associate with $L\in (\pi/\sqrt{\kappa})$ (arising from  $V_\kappa(C_\kappa^n(L))=V_\kappa(\Omega^\star)$) the {\it half-cap radius} $L_0\in \left(0,\frac{\pi}{2\sqrt{\kappa}}\right)$ defined by 
\begin{equation}\label{half-cap-radius}
2V_\kappa(C_\kappa^n(L_0))=V_\kappa(C_\kappa^n(L)),	
\end{equation}
and we also introduce the notation $\alpha_{L_0}\coloneqq\sin^2\left(\frac{\sqrt{\kappa} L_0}{2}\right).$ Letting $\alpha\to \alpha_{L_0}$ and $\beta\to \alpha_{L_0}$ in \eqref{fogo-alfa-beta},  we obtain that
\begin{equation}\label{fogo-hatarertek}
	 \lambda_{\kappa,n}(\alpha_{L_0},\alpha_{L_0})=\mathfrak  f_{\kappa,n,1}(\alpha_{L_0}),
\end{equation}
which corresponds to $a=b={L_0}$. Due to \eqref{fogo-hatarertek}, a \textit{necessary} condition for the validity of \eqref{amit-kell-igazolni-2} is 
   \begin{equation}\label{necessarily-1}
	\mathfrak  f_{\kappa,n,1}(\alpha_{L_0})\geq \lambda_{\kappa,n}(0,\alpha_{L}).
\end{equation}
 Postponing the study of \eqref{necessarily-1} (see  \S \ref{high-dimension}), we show in the sequel that \eqref{necessarily-1} is also \textit{sufficient} to prove Rayleigh's conjecture that will be done in two steps.
%; namely, for $\alpha$'s close to ${\sf L}_0$ a continuity reason applies, while for small values of $\alpha$ a monotonicity argument provides the required conclusion. 
%once  is valid for certain $n\geq 2$ and $L\in (0,\pi/\sqrt{\kappa})$, we can prove the validity of Lord Raighley's conjecture. 

First, if strict inequality occurs in \eqref{necessarily-1}, by using a continuity argument in the transcendental equation $\mathcal S_{\kappa,n}(\alpha,\beta,\lambda)= 0$ together with equality \eqref{fogo-hatarertek}, it follows that there is $\alpha>0$ sufficiently close to $\alpha_{L_0}$ such that  $\lambda_{\kappa,n}(\alpha,\beta(\alpha))>\lambda_{\kappa,n}(0,{\alpha_{L}}),$ where $\beta=\beta(\alpha)$ is from \eqref{alfbet}. Quantitatively, the last statement implies that one can find the unique minimal $\alpha_0\in (0,\alpha_{L_0})$ such that $\lambda_{\kappa,n}(\alpha,\beta(\alpha))>\lambda_{\kappa,n}(0,{\alpha_{L}}),$ holds for $\alpha\in [\alpha_0,\alpha_{L_0}]$, so \eqref{amit-kell-igazolni-2} is verified. In particular,  $\beta_0=\beta(\alpha_0)$ verifies  the condition $\mathfrak  f_{\kappa,n,1}(\beta_0)=\lambda_{\kappa,n}(0,\alpha_{L})$, where $\beta_0$ is a pole of the function $\mathcal S_{\kappa,n}(\alpha_0,\cdot,\lambda_{\kappa,n}(0,\alpha_{L}))$.

Second,  by definition, we have that $\mathcal S_{\kappa,n}(0,\alpha_{L},\lambda_{\kappa,n}(0,\alpha_{L}))=0$ and the construction of  $\alpha_0>0$ and $\beta_0=\beta(\alpha_0)>0$ implies that   $\lim_{\alpha\nearrow \alpha_0}\mathcal S_{\kappa,n}(\alpha,\beta(\alpha),\lambda_{\kappa,n}(0,\alpha_{L}))=-\infty$. In fact, one has that $\mathcal S_{\kappa,n}(\alpha,\beta(\alpha),\lambda_{\kappa,n}(0,\alpha_{L}))<0$ for every $\alpha\in (0,\alpha_0)$. Indeed, since by \eqref{alfbet}  we have that $$[\alpha(1-\alpha)]^{\frac{n}{2}-1}+\beta'(\alpha)[\beta(\alpha)(1-\beta(\alpha))]^{\frac{n}{2}-1}=0,$$ a  similar computation as in Karp and Sitnik \cite{KS} shows  that
$\frac{\rm d}{\rm d\alpha}\mathcal S_{\kappa,n}(\alpha,\beta(\alpha),\lambda_{\kappa,n}(0,\alpha_{L}))<0$ for every $\alpha\in (0,\alpha_0)$. 
%MAPLE: (430)-(443) kepletek. 
Thus, the function $\alpha\mapsto \mathcal S_{\kappa,n}(\alpha,\beta(\alpha),\lambda_{\kappa,n}(0,\alpha_{L}))$ is decreasing on  $(0,\alpha_0)$ and 
\begin{equation}\label{s-k-n-kell}
	\mathcal S_{\kappa,n}(\alpha,\beta(\alpha),\lambda_{\kappa,n}(0,\alpha_{L}))< \mathcal S_{\kappa,n}(0,\alpha_{L},\lambda_{\kappa,n}(0,\alpha_{L}))=0,\ \forall  \alpha\in (0,\alpha_0).
\end{equation}
If we assume by contradiction that there exists $\alpha\in (0,\alpha_0)$ such that $	\lambda_{\kappa,n}(\alpha,\beta(\alpha))< \lambda_{\kappa,n}(0,\alpha_{L})$, by the property \eqref{s-k-n-kell} and the fact that $\lambda\mapsto \mathcal S_{\kappa,n}(\alpha,\beta,\lambda)$ is  increasing  between any two consecutive poles of $\mathcal S_{\kappa,n}(\alpha,\beta,\cdot)$, it turns out that 
$$0> \mathcal S_{\kappa,n}(\alpha,\beta(\alpha),\lambda_{\kappa,n}(0,\alpha_{L}))\geq \mathcal S_{\kappa,n}(\alpha,\beta(\alpha),\lambda_{\kappa,n}(\alpha,\beta(\alpha)))=0,$$ which is a contradiction.
%, where we have used the definition of $\lambda_{\kappa,n}(\alpha,\beta)$  in the last step.\ 
Therefore, $	\lambda_{\kappa,n}(\alpha,\beta(\alpha))\geq \lambda_{\kappa,n}(0,\alpha_{L})$ for every $\alpha\in (0,\alpha_0)$, which ends the proof of \eqref{amit-kell-igazolni-2}. \\

In conclusion,  it remains to investigate the validity of inequality \eqref{necessarily-1} which turns out to depend both on the \textit{dimension} $n\geq2$ and the {\it size} $L>0$. In this way we can decide the  validity of Lord Rayleigh's conjecture for any non-empty open smooth set $\Omega\subset M$ that verifies the equality $\Omega^\star=C_\kappa^n(L)$ and condition  \eqref{terfogatok-egyenloek}. The next subsection is devoted to this study. 

%$$\alpha\coloneqq \sin^2\left(\frac{\sqrt{\kappa} a}{2}\right),\ \beta\coloneqq \sin^2\left(\frac{\sqrt{\kappa} b}{2}\right)\ {\rm and}\ {\sf L}\coloneqq \sin^2\left(\frac{\sqrt{\kappa} L}{2}\right).$$

\subsection{Validity of the conjecture: dimension and  domain dependence}\label{high-dimension}

Using the  notations from \S \ref{subsection-reduction}, for every  $\kappa>0$ and $n\in \mathbb N_{\geq 2}$ we introduce the set 
$$\mathcal H_{\kappa,n}\coloneqq\left\{L\in \left(0,\frac{\pi}{\sqrt{\kappa}}\right): \mathfrak  f_{\kappa,n,1}(\alpha_{L_0})< \lambda_{\kappa,n}(0,\alpha_{L})\right\},$$
 where $\alpha_{L}=\sin^2\left(\frac{\sqrt{\kappa} L}{2}\right)$,  $\alpha_{L_0}=\sin^2\left(\frac{\sqrt{\kappa} L_0}{2}\right)$ and the value  $L_0>0$ is  the half-cap radius associated with $L>0,$ see \eqref{half-cap-radius}. The following result is crucial. 

\begin{proposition}\label{utolso-proposition}
If $\kappa>0$ and $n\in \mathbb N_{\geq 2}$, the following statements hold$:$
	\begin{itemize}
		\item[(i)] $\mathcal H_{\kappa,n}\neq \varnothing$ for every $n\geq 4;$
		\item[(ii)] $\mathcal H_{\kappa,2}=\mathcal H_{\kappa,3}=\varnothing.$
	\end{itemize}
\end{proposition}
\noindent {\it Proof.} 
(i) We first claim that $\mathcal H_{\kappa,n}$ does not contain elements close to ${\pi}/{\sqrt{\kappa}}$ for every $n\geq 2.$ Indeed,  for the half-cap radius $L_0$ associated with $L>0$ we have that $L_0\to \frac{\pi}{2\sqrt{\kappa}}$  whenever $L\to \frac{\pi}{\sqrt{\kappa}}$. Therefore, $\alpha_{L_0}=\sin^2\left(\frac{\sqrt{\kappa} L_0}{2}\right)\to \sin^2\left(\frac{\pi}{4}\right)=\frac{1}{2}$, and according to Proposition \ref{Ferrers-basic-lemma-1}/(vi) we have that 
\begin{equation}\label{f-rrr-esztimate}
	\mathfrak  f_{\kappa,n,1}(\alpha_{L_0})\to \sqrt{n\kappa}\ \ {\rm  as}\ \  L\to \frac{\pi}{\sqrt{\kappa}}.
\end{equation}
 On the other hand,  Theorem \ref{aszimptotak-CP-pozitiv}/(ii) yields that
\begin{equation}\label{lambda-rrr-esztimate}
\lambda_{\kappa,n}(0,\alpha_{L})=	\Lambda^\frac{1}{4}_{\kappa}(C_\kappa^n(L))\sim  \left\{ \begin{array}{lll}
		\sqrt{\mu_n \kappa} &\mbox{if} & n\in \{2,3\}, \\
		% u\geq 0 &\mbox{in} &   \Omega;\\
		0 &\mbox{if} &  n\geq 4,
	\end{array}\right. \ \ {\rm as}\ \ L\to \frac{\pi}{\sqrt{\kappa}},
\end{equation}
where  $\mu_2\approx 0.9125$ and $\mu_3\approx 1.0277$. Therefore, $\mathfrak  f_{\kappa,n,1}(\alpha_{L_0})>  \lambda_{\kappa,n}(0,\alpha_{L})$ if $L$ is sufficiently close  to ${\pi}/{\sqrt{\kappa}},$ which concludes the claim.

Again by \eqref{half-cap-radius}, the half-cap radius $L_0$ associated with $L>0$ verifies the asymptotic property  $2L_0^n\sim L^n$ whenever $L\ll 1$. Moreover, a similar argument as in \S \ref{section-small-caps} yields that $\mathfrak  f_{\kappa,n,1}(\alpha_{L_0})\sim {{C}/{L_0}}$ as $L\to 0$ with $J_{\frac{n}{2}-1}(C)=0$, i.e., $C={\mathfrak j}_{\frac{n}{2}-1}$. Thus, by 
Theorem \ref{aszimptotak-CP-pozitiv}/(i) it follows that 
\begin{equation}\label{1-hatarertek}
	\liminf_{L\to 0}\frac{\mathfrak  f_{\kappa,n,1}(\alpha_{L_0})}{\lambda_{\kappa,n}(0,\alpha_{L})}=2^\frac{1}{n}\frac{ {\mathfrak j}_{\frac{n}{2}-1}}{\mathfrak h_{\frac{n}{2}-1}}.
\end{equation}
We observe that $2^\frac{1}{n}\frac{{\mathfrak j}_{\frac{n}{2}-1}}{\mathfrak h_{\frac{n}{2}-1}}<1$ if and only if $n\geq 4.$ In particular, 
for every $n\geq 4$, we obtain that $\mathfrak  f_{\kappa,n,1}(\alpha_{L_0})< \lambda_{\kappa,n}(0,\alpha_{L})$ for sufficiently small $L>0$, proving that 
  $\mathcal H_{\kappa,n}\neq \varnothing$; this concludes the proof of (i). 
  
   (ii) The previous arguments show  that for $n\in \{2,3\}$ the set $\mathcal H_{\kappa,n}$ does not contain elements  in the vicinity of $0$ or close to  ${\pi}/{\sqrt{\kappa}}$. We will prove that this property persists to the whole interval $\left(0,{\pi}/{\sqrt{\kappa}}\right)$ whenever $n\in \{2,3\}$. To do this, it is enough to prove that the equation 
 $ \mathfrak  f_{\kappa,n,1}(\alpha_{L_0})= \lambda_{\kappa,n}(0,\alpha_{L})$ is not solvable in $L\in \left(0,{\pi}/{\sqrt{\kappa}}\right)$ for $n\in \{2,3\}$, where $\alpha_L=\sin^2\left(\frac{\sqrt{\kappa} L}{2}\right)$,   $\alpha_{L_0}=\sin^2\left(\frac{\sqrt{\kappa} L_0}{2}\right)$ and  $L_0>0$ is the half-cap radius associated with $L>0.$ Our proof is dimension-dependent.  
 
\textit{The case  $n=3$.} Using \eqref{g-sajatertekek} and $\alpha_{L_0}=\sin^2\left(\frac{\sqrt{\kappa} L_0}{2}\right)$, we recall  that
 	\begin{equation}\label{g-sajatertekek-1}
 	{\mathfrak  f}_{\kappa,3,1}(\alpha_{L_0})=\sqrt{\kappa\left(\frac{\pi^2}{\kappa L_0^2}-1\right)}.
 \end{equation}
Due to relations \eqref{g-sajatertekek-1} and \eqref{explicit3dimenzioban}, the solvability of $ \mathfrak  f_{\kappa,3,1}(\alpha_{L_0})= \lambda_{\kappa,3}(0,\alpha_{L})$ reduces to the equation 
$\mathcal K_{\kappa,3}(\alpha_{L},\mathfrak  f_{\kappa,3,1}(\alpha_{L_0}))=0$, i.e., 
\begin{equation}\label{3-dim-elso}
	\sqrt{\frac{\pi^2}{\kappa L_0^2}-2}\cdot\coth\left(\sqrt{\kappa}L\sqrt{\frac{\pi^2}{\kappa L_0^2}-2}\right)-\frac{\pi}{\sqrt{\kappa}L_0}\cdot \cot\left(\frac{\pi L}{L_0}\right)=0,
\end{equation}
where $L$ and $L_0$ verify \eqref{half-cap-radius}, i.e., 
\begin{equation}\label{3-dim-masodik}
2\left(L_0-\frac{\sin(2\sqrt{\kappa}L_0)}{2\sqrt{\kappa}}\right)=L-\frac{\sin(2\sqrt{\kappa}L)}{2\sqrt{\kappa}}.
\end{equation}
Our aim is to prove that equalities \eqref{3-dim-elso} and \eqref{3-dim-masodik} are incompatible, which will imply $\mathcal H_{\kappa,3}=\varnothing$. If  $x\coloneqq \sqrt{\kappa}L_0$ and $y\coloneqq \sqrt{\kappa}L$,  the implicit function theorem together with \eqref{3-dim-masodik} implies the existence of  a (unique) differentiable function $p:(0,\pi/2)\to (0,\pi)$ such that $P(x,p(x))=0$   for every $x\in (0,\pi/2)$, where $P:(0,\pi/2)\times (0,\pi)\to \mathbb R$ is \ $$ P(x,y)\coloneqq2\left(2x-{\sin(2x)}\right)-2y+\sin(2y).$$ 
%Due to  $P(x,p(x))=0$, one can prove  that $p$ is increasing and convex on $(0,\pi/2)$.
Let us also  consider 
$Q:(0,\pi/2)\times (0,\infty)\setminus S_Q\to \mathbb R$ defined by $$Q(x,y)\coloneqq	\sqrt{{\pi^2}-2{x^2}}\cdot\coth\left(y\sqrt{\frac{\pi^2}{x^2}-2}\right)-{\pi } \cot\left(\frac{\pi y}{x}\right),$$
where $S_Q\coloneqq\{(x,y)\in (0,\pi/2)\times (0,\infty):y/x\in \mathbb N\}$.  In order to prove the claim, it is enough to show that $p(x)>  q(x)$ for every $x\in (0,\pi/2)$, where  $q:(0,\pi/2)\to (0,\pi)$ is the smallest differentiable function such that $Q(x,q(x))=0$   for every $x\in (0,\pi/2)$, see \eqref{3-dim-elso}. 
%Although elementary, the proof requires subtle comparison arguments due to the involved form of the above equations. 
The idea of the proof is to separate $p$ and $q$ by \textit{piece-wise linear} functions; this fact is motivated by the monotonicity property of $x\mapsto Q(x,cx)$ whenever $c>0,$ which reduces to the monotonicity of $y\mapsto y\coth(y)$ on $(0,\infty)$.  
The separation argument can be described as follows. 
%Geometrically, the separation argument is illustrated on  \textbf{Figure ???} that is  explained analytically in the sequel. 

Since $p(x)\sim 2^\frac{1}{3}x=:r_1(x)$ as $x\to 0$, it turns out that $\lim_{x\to 0} Q(x,p(x))=\pi \coth\big( 2^\frac{1}{3}\pi\big)-\pi \cot\big( 2^\frac{1}{3}\pi\big)\approx 0.192>0$, thus $Q(x,p(x))\neq 0$ for small values of $x$; in fact, this property can be also deduced by \eqref{1-hatarertek}. On one hand, we observe that $p(x)> r_1(x)$ for every $x\in (0,\pi/2)$. Indeed, an elementary argument shows that $P(x,r_1(x))>0$ for every $x\in (0,\pi/2)$ and $p(\pi/2)=\pi>2^{-\frac{2}{3}}\pi=r_1(\pi/2)$. Thus, if there exists $\widetilde x\in (0,\pi/2)$ such that $p(\widetilde x)\leq r_1(\widetilde x)$, by a continuity reason there exists $\overline{x}\in [\widetilde x,\pi/2)$ such that $p(\overline x)=r_1(\overline x)$, and $0=P(\overline x,p(\overline x))=P(\overline x,r_1(\overline x))>0$, which is a contradiction. On the other hand,
we have that the function $x\mapsto Q(x,r_1(x))$ is decreasing on $(0,\pi/2)$ having its unique zero at $x_1\approx 0.767.$  If $c\approx 1.249$ is the first positive zero of the equation $\coth(c\pi )=\cot(c\pi )$, it follows that  $\lim_{x\to 0} Q(x,cx)=\pi \coth( c\pi)-\pi \cot( c\pi)=0$; taking into account that $c<2^\frac{1}{3}\approx1.259$ and $Q(x,q(x))=0$, the minimality property of $q$ implies that 
 $q(x)<r_1(x)$ for small values of $x>0$. By construction, it follows that $q(x_1)=r_1(x_1)=2^\frac{1}{3}x_1\approx 0.967$. Thus, if there exists $\widetilde x\in (0,x_1)$ with $q(\widetilde x)\geq r_1(\widetilde x)$, one can find $\overline{x}\in (0,\widetilde x]$ such that $q(\overline x)=r_1(\overline x)$, so $0=Q(\overline x,q(\overline x))=Q(\overline x,r_1(\overline x))>0$, which is a contradiction. Therefore,  $p(x)> r_1(x)> q(x)$ for every $x\in (0,x_1).$ Clearly, we also have $p(x_1)> r_1(x_1)= q(x_1)$.
 
 The function $r_1$ cannot separate $p$ and $q$ beyond the value $x_1$. Therefore, we consider  $r_2(x)\coloneqq \frac{p(x_1)}{x_1}x $ for every $x\in [x_1,\pi/2),$  and  a similar argument as above shows that $p(x)> r_2(x)> q(x)$ for every $x\in (x_1,x_2),$ where $x_2\approx 1.462$ is the intersection point of $r_2$ and $q$. Moreover, $p(x_2)> r_2(x_2)= q(x_2)$. Finally, the function $r_3(x)\coloneqq \frac{p(x_2)}{x_2}x $, $x\in [x_2,\pi/2),$ has the property that $p(x)> r_3(x)> q(x)$ for every $ x\in (x_2,\pi/2).$ 
 
 Summing up, we have $p(x)>  q(x)$ for every $ x\in (0,\pi/2),$  which proves that $\mathcal H_{\kappa,3}=\varnothing.$

%\begin{equation}\label{F=0=G}
%	F(x,f(x))=0=G(x,g(x))\ \ {\rm  for\ every}\ \ x\in (0,\pi/2),
%\end{equation}
% where    and  

\textit{The case  $n=2$.} 
Due to \eqref{half-cap-radius}, in the case $n=2$  we have that
\begin{equation}\label{fel-2dim}
2\alpha_{L_0}=2\sin^2\left(\frac{\sqrt{\kappa} L_0}{2}\right)=\frac{\kappa}{2\pi}V_\kappa(C_\kappa^2(L_0))=\frac{\kappa}{4\pi}V_\kappa(C_\kappa^2(L))=\sin^2\left(\frac{\sqrt{\kappa} L}{2}\right)=\alpha_{L}
.	
\end{equation}
Thus the question reduces to the non-solvability of 
$ \mathfrak  f_{\kappa,2,1}(t/2)= \lambda_{\kappa,2}(0,t)$  in $t\in \left(0,1\right)$. On one hand, by \eqref{f-rrr-esztimate}--\eqref{1-hatarertek} we know that this equation cannot be solved for values $t$ close to $0$ and $1$. On the other hand, due to \eqref{P-hypergoemetric}, \eqref{inversion} and \eqref{F-differential-2} (since $n=2$), we obtain that 
$$\mathcal F_\pm(t,\lambda,\kappa,2)=\LP_{-\frac{1}{2}+\Lambda_\pm}^0(1-2t)\ \ {\rm and}\ \ \mathcal F'_\pm(t,\lambda,\kappa,2)=\frac{1}{\sqrt{(1-t)t}}\LP_{-\frac{1}{2}+\Lambda_\pm}^{1}(1-2t),\ t\in (0,1),$$
where $\Lambda_\pm=\sqrt{\frac{1}{4}\pm \frac{\lambda^2}{\kappa}};$ see also Zhurina and Karmazina \cite{ZK}. In particular,  $\mathcal K_{\kappa,2}(t,\lambda)=0$ can be transformed into 
an equation containing only the associated Legendre functions $\LP_{-\frac{1}{2}+\Lambda_\pm}^r$  of integer orders $r\in \{0,1\}$. 
%$
%{_2F}_1\left(1+\nu,-\nu;1;\frac{1-t}{2}\right)={\bf P}_\nu^0(t)$, $t\in (-1,1),$ 
Thus, the  tables of zeros  with respect to the degree $\nu=-\frac{1}{2}+\Lambda_\pm$, see Bauer \cite{Bauer} (and also  Zhang and Jin \cite[Chapter 4]{Zhang-Jin}) imply that 
$$\mathfrak  f_{\kappa,2,1}\left(\frac{t}{2}\right)- \lambda_{\kappa,2}(0,t)> \frac{\sqrt{\kappa}}{5}, \ \forall t\in (0,1),\ \kappa>0.$$ 
Therefore, the latter estimate and relation \eqref{fel-2dim} imply that $\mathcal H_{\kappa,2}=\varnothing.$ \hfill $\square$
%%=\frac{2^{\frac{1}{n}}{\mathfrak j}_{\nu,1}}{\mathfrak  h_{\nu}}=
%\left\{ \begin{array}{lll}
%	\frac{2^{\frac{1}{2}}{\mathfrak j}_{0,1}}{\mathfrak  h_{0}}\approx  \frac{2^{\frac{1}{2}}\cdot 2.4048}{3.19622}\approx 1.064>1 &\mbox{if} &  n=2,\\
%	\\
%	% u\geq 0 &\mbox{in} &   \Omega;\\
%	%\frac{2^{\frac{1}{3}}\pi}{3.9266}\approx 1.008>1 &\mbox{if} &  n=3,
%	\frac{2^{\frac{1}{3}}{\mathfrak j}_{1/2,1}}{\mathfrak  h_{1/2}}\approx \frac{2^{\frac{1}{3}}\pi}{3.9266}\approx  1.008>1 &\mbox{if} &  n=3,
%\end{array}\right.

\begin{remark}\rm 
	Numerical tests show that the function $L\mapsto \frac{\mathfrak  f_{\kappa,n,1}(\alpha_{L_0})}{\lambda_{\kappa,n}(0,\alpha_{L})}$ is increasing on $\left(0,{\pi}/{\sqrt{\kappa}}\right)$ for every $n\in \mathbb N_{\geq 2}$ and $\kappa>0$, where $L_0$ is the half-cap radius  associated with $L>0$. If this statement indeed holds, we can present an alternative proof for Proposition \ref{utolso-proposition}/(ii). Indeed, by the assumed monotonicity and \eqref{1-hatarertek} we would have for every $L\in \left(0,{\pi}/{\sqrt{\kappa}}\right)$ that
	$$
	\frac{\mathfrak  f_{\kappa,n,1}(\alpha_{L_0})}{\lambda_{\kappa,n}(0,\alpha_{L})}\geq \liminf_{L\to 0} \frac{\mathfrak  f_{\kappa,n,1}(\alpha_{L_0})}{\lambda_{\kappa,n}(0,\alpha_{L})}=2^\frac{1}{n}\frac{ {\mathfrak j}_{\frac{n}{2}-1}}{\mathfrak h_{\frac{n}{2}-1}}=\left\{\def\arraystretch{1.9} \begin{array}{@{~}l@{~}l@{~}l@{~}l@{~}lll}
		2^{\frac{1}{2}}\frac{{\mathfrak j}_{0}}{\mathfrak  h_{0}}&\approx & 2^{\frac{1}{2}}\frac{ 2.4048}{3.1962}&\approx& 1.064>1 &\mbox{if} &  n=2,\\
		% u\geq 0 &\mbox{in} &   \Omega;\\
		%\frac{2^{\frac{1}{3}}\pi}{3.9266}\approx 1.008>1 &\mbox{if} &  n=3,
		2^{\frac{1}{3}}\frac{{\mathfrak j}_\frac{1}{2}}{\mathfrak  h_\frac{1}{2}}&\approx& 2^{\frac{1}{3}}\frac{\pi}{3.9266}&\approx & 1.008>1 &\mbox{if} &  n=3,
	\end{array}\right.$$
thus $\mathcal H_{\kappa,2}=\mathcal H_{\kappa,3}=\varnothing.$ 
%In addition, for $n\geq 4$, the above monotonicity shows that  ${\mathfrak  f_{\kappa,n,1}(\alpha_{L_0})}\geq {\lambda_{\kappa,n}(0,\alpha_{L})}$ for every $L\geq \sup \mathcal H_{\kappa,n}.$
\end{remark}

%Now, we are ready to conclude the proof of the main result of the paper.\\

\noindent {\it Proof of Theorem  \ref{fotetel-CP-pozitiv}.} Let $(M,g)$ be a compact  $n$-dimensional  Riemannian manifold  with ${\sf Ric}_{(M,g)}\geq (n-1)\kappa>0$, consider   a smooth domain $\Omega\subset M$ and let $\Omega^\star\subset \mathbb S^n_\kappa$ be a  spherical cap for which the conditions $$\frac{V_g(\Omega)}{V_g(M)}=\frac{V_{\kappa}(\Omega^\star)}{V_{\kappa}(\mathbb S^n_\kappa)}$$ and  $V_\kappa(\Omega^\star)=V_\kappa(C_\kappa^n(L))$ are satisfied for some $L\in \left(0,{\pi}/{\sqrt{\kappa}}\right)$.

On one hand,  Proposition \ref{utolso-proposition}/(ii) implies that inequality \eqref{necessarily-1} is verified for $n\in \{2,3\}$ and any $L\in \left(0,{\pi}/{\sqrt{\kappa}}\right)$; in particular, the (sufficiency) argument at the end of the previous subsection shows that Lord Rayleigh's conjecture is true, i.e., \eqref{foegyenlotlenseg} is valid  for any domain $\Omega\subset M$ whenever  $n\in \{2,3\}$.

On the other hand, when $n\geq 4$, it turns out by the proof of Proposition \ref{utolso-proposition}/(i) that Lord Rayleigh's conjecture holds true on $(M,g)$ for any domain $\Omega\subset M$ with ${V_g(\Omega)}> v_{n,\kappa}{V_g(M)}$, where 
\begin{equation}\label{v-n-kappa-L}
v_{n,\kappa}=\dfrac{V_{\kappa}\big(C_\kappa^n(L_{n,\kappa})\big)}{V_{\kappa}(\mathbb S^n_\kappa)}\in (0,1)\ \ {\rm and}\ \ L_{n,\kappa}\coloneqq \sup \mathcal H_{\kappa,n}\in \left(0,\frac{\pi}{\sqrt{\kappa}}\right).
\end{equation}
In addition, note that for every $\alpha\in (0,1)$, the expressions 
$f_{\kappa,n,1}(\alpha)/\sqrt{k}$ and $ \lambda_{\kappa,n}(0,\alpha)/\sqrt{\kappa}$ are $\kappa$-independent. Therefore, $L_{n,\kappa}=L_n/\sqrt{\kappa}$ for some $\kappa$-independent value $L_n\in (0,\pi)$, and a simple computation shows that $v_n\coloneqq v_{n,\kappa}$ does not depend on $\kappa>0.$

Assume that equality holds in \eqref{foegyenlotlenseg} for some $\Omega\subset M$. By the proof of Theorem \ref{talenti-result} (see the verification of inequalities \eqref{1-compar} and \eqref{2-compar}), we should have equality in the L\'evy--Gromov inequality \eqref{Levy--Gromov-inequal} for a.e.\ admissible $t>0$. In particular, this equality implies that $(M,g)$ is isometric to $(\mathbb S_\kappa^n,g_{\kappa})$, and the sets $\{x\in\Omega_\pm:u_\pm(x)>t \}$ and $\{x\in \Omega_\pm^\star:u_\pm^\star(x)>t \}$ are isometric for a.e.  $t\in \big[0,T_u^\pm\big]$. Since only one set of $\Omega_+$ and $\Omega_-$ remains (say  $\Omega_+$, cf. subsection \S \ref{subsection-reduction}), it turns out that  $\Omega=\Omega_+\subset M$ is isometric to the spherical cap $\Omega^\star=\Omega_+^\star\subset \mathbb S_\kappa^n$. The converse statement trivially holds.  

Finally, let $L^0_{n,\kappa}$ be the half-cap radius  associated with $L_{n,\kappa}>0$.  If we assume that $v_\infty\coloneqq\displaystyle\limsup_{n\to \infty} v_{n}=1$, which is equivalent to $\displaystyle\limsup_{n\to \infty} L_{n,\kappa}={\pi}/{\sqrt{\kappa}},$  a similar reasoning as in  \eqref{f-rrr-esztimate}--\eqref{lambda-rrr-esztimate} implies  that
$$0=\limsup_{n\to \infty}\lambda_{\kappa,n}(0,\alpha_{L_{n,\kappa}})=\limsup_{n\to \infty}\mathfrak  f_{\kappa,n,1}(\alpha_{L^0_{n,\kappa}})=\limsup_{n\to \infty}\sqrt{n\kappa}= +\infty,$$ which is a contradiction. Therefore, $v_\infty=\displaystyle\limsup_{n\to \infty} v_{n}<1$, which concludes the proof. 
\hfill $\square$

%A few values of $v_{n,\kappa}$ can be found in Table \ref{table-2}.

%In the sequel we present values for $L_{n,\kappa}$  and $v_{n,\kappa}$ in different dimensions  $n\geq 2.$ 
 
 \begin{remark}\rm  Theorem  \ref{fotetel-CP-pozitiv} states in  particular that in high-dimensions Lord Rayleigh's conjecture is true for large clamped plates on any compact Riemannian manifold $(M,g)$ of positive Ricci curvature. Note that the inequality $v_\infty<1$  implies that these clamped plates need not be very `close'  to the whole manifold; quantitatively, the arguments are valid for clamped plates $\Omega\subset M$ with $\frac{V_g(\Omega)}{V_g(M)}\in (v_\infty,1]$.   
% 	 It is still a challenging question whether Lord Rayleigh's conjecture holds for small clamped plates. 
 Numerical tests show that $\displaystyle\limsup_{n\to \infty} L_{n,\kappa}={\pi}/({2\sqrt{\kappa}}),$ thus 
  $v_\infty=1/2$, see
 	Table \ref{table-2}, which would imply that clamped plates $\Omega\subset M$ with at least  `half-measure' of $M$ verify Lord Rayleigh's conjecture.   
% 	In fact, numerical tests suggest that $L_{n,\kappa}\to \frac{\pi}{2\sqrt{\kappa}}$, thus $v_n=v_{n,\kappa}\to \frac{1}{2}$ whenever $n\to \infty;$ 
 \end{remark}
 \vspace{-0.3cm}

{
\setlength{\abovecaptionskip}{1pt}
\setlength{\belowcaptionskip}{-1pt}

\setlength{\abovedisplayskip}{0pt}
\setlength{\belowdisplayskip}{0pt}
\setlength{\abovedisplayshortskip}{0pt}
\setlength{\belowdisplayshortskip}{0pt}

	\renewcommand{\arraystretch}{1.8}	
\providecommand{\tabularnewline}{\\}
%	\begin{center}
\begin{table}[H]
	\centering
	\begin{tabular}{|c||>{\centering}p{1.0cm}|>{\centering}p{1.0cm}|>{\centering}p{1.0cm}|>{\centering}p{1.0cm}|>{\centering}p{1.0cm}|>{\centering}p{1.0cm}|>{\centering}p{1.0cm}|>{\centering}p{1.0cm}|>{\centering}p{1.0cm}|>{\centering}p{1.0cm}|>{\centering}p{1.0cm}|>{\centering}p{1.0cm}|}
		\hline 
		%\cline{2-5} 
		\cellcolor{green!46}	$n$	& \cellcolor{green!24} 2 \& 3 &  \cellcolor{green!20} 4 & \cellcolor{green!18} 5 &\cellcolor{green!16}  6 & \cellcolor{green!14} 7& \cellcolor{green!12} 10 & \cellcolor{green!10}  50 & \cellcolor{green!8} 100 & \cellcolor{green!6} 200 & \cellcolor{green!4} 500& \cellcolor{green!2} 1000\tabularnewline
		\hline
		\cellcolor{red!46}	
		$L_{n,\kappa}$ &0 \cellcolor{red!24}	  \cellcolor{red!24}	 &\cellcolor{red!20}	 $\frac{0.27\pi}{\sqrt{\kappa}}$ &\cellcolor{red!18}	 $\frac{0.355\pi}{\sqrt{\kappa}}$  &\cellcolor{red!16} $\frac{0.394\pi}{\sqrt{\kappa}}$  &\cellcolor{red!14} $\frac{0.417\pi}{\sqrt{\kappa}}$ &\cellcolor{red!12} $\frac{0.448\pi}{\sqrt{\kappa}}$&\cellcolor{red!10} $\frac{0.487\pi}{\sqrt{\kappa}}$ &\cellcolor{red!8} $\frac{0.492\pi}{\sqrt{\kappa}}$&\cellcolor{red!6} $\frac{0.495\pi}{\sqrt{\kappa}}$&\cellcolor{red!4} $\frac{0.4969\pi}{\sqrt{\kappa}}$&\cellcolor{red!2} $\frac{0.4981\pi}{\sqrt{\kappa}}$\tabularnewline
		\hline 
		\cellcolor{teal!46}	$v_{n}$  &\cellcolor{teal!24} 0&\cellcolor{teal!20} 0.0763&\cellcolor{teal!18} 0.1616 &\cellcolor{teal!16} 0.2146 &\cellcolor{teal!14}0.2515 &\cellcolor{teal!12} 0.3067&\cellcolor{teal!10} 0.3869&\cellcolor{teal!8}0.401&\cellcolor{teal!6} 0.4122&\cellcolor{teal!6} 0.4138&\cellcolor{teal!2} 0.4252
		\tabularnewline
		
		\hline 
	\end{tabular}

	\vspace*{0.2cm}
	\caption{Values of $L_{n,\kappa}$  and $v_{n}$ from \eqref{v-n-kappa-L} in certain dimensions $n\geq 2.$ 
%		Lord Rayleigh's conjecture holds on any compact $n$-dimensional Riemannian manifold with ${\rm Ric}(M,g)\geq (n-1)\kappa>0$  for any clamped domain $\Omega\subset M$ verifying $V_g(\Omega)>v_{n,\kappa}V_g(M)$.} 
}	\label{table-2}
\end{table}

}

%
%	\renewcommand{\arraystretch}{1.4}	
%\providecommand{\tabularnewline}{\\}
%%	\begin{center}
%\begin{table}[H]
%	\centering
%	\begin{tabular}{|c|>{\centering}p{3.4cm}|>{\centering}p{3.4cm}|>{\centering}p{3.4cm}|>{\centering}p{3.4cm}|}
%		\hline 
%		%\cline{2-5} 
%		$n$	&$L_{n,\kappa}\coloneqq \sup \mathcal H_{\kappa,n}$ &  $v_{n,\kappa}$ \tabularnewline
%		\hline 
%		$
%		2,3$ &0  &0\tabularnewline
%		\hline 
%		$	4$ & 0.27$\pi/\sqrt{\kappa}$ & 0.0762 \tabularnewline
%		\hline 
%		$5$ &0.355$\pi/\sqrt{\kappa}$ &0.1616 \tabularnewline
%		\hline 
%		$6$ &0.394$\pi/\sqrt{\kappa}$&0.2146 \tabularnewline
%		\hline 
%		$7$ &0.417$\pi/\sqrt{\kappa}$&0.2515 \tabularnewline
%		\hline 
%		$50$ &0.487$\pi/\sqrt{\kappa}$&0.3869 \tabularnewline
%		\hline 
%		$100$ &0.492$\pi/\sqrt{\kappa}$&0.401 \tabularnewline
%		\hline 
%		$200$ &0.495$\pi/\sqrt{\kappa}$&0.4122 \tabularnewline
%		\hline 
%		$500$ &0.497$\pi/\sqrt{\kappa}$&0.4165 \tabularnewline
%		\hline 
%	\end{tabular}
%	\vspace*{0.2cm}
%	\caption{Values of $L_{n,\kappa}$  and $v_{n,\kappa}$ from \eqref{v-n-kappa-L} for some  $n\geq 2.$ Lord Rayleigh's conjecture holds on any compact $n$-dimensional Riemannian manifold with ${\rm Ric}(M,g)\geq (n-1)\kappa>0$  for any clamped domain $\Omega\subset M$ verifying $V_g(\Omega)>v_{n,\kappa}V_g(M)$.} \label{table-2}
%\end{table}
%%\end{center}
%\renewcommand{\arraystretch}{1}

%
%\subsection{High dimensions: large domains verify the conjecture} 
%
%We assume that $n\geq 4.$
%
%\subsection{Low dimensions: any domain verifies the conjecture}\label{low-dimension} Ide jon a 2 es 3 dimenzio. 

	\section{Curvature limit in Lord Rayleigh's conjecture: proof of Theorem \ref{Huisken-CP}}\label{section-7}
	
Let $\Omega\subset M$ be a bounded open set in a complete non-compact $n$-dimensional Riemannian manifold $(M,g)$ with  ${\sf Ric}_{(M,g)}\geq 0$ and ${\sf AVR}_{(M,g)}> 0$, see \eqref{volume-ratio}. Due to the boundedness of $\Omega$, the injectivity radius is positive on $\Omega$, see Klingenberg \cite[Proposition 2.1.10]{Kli}, thus 
the Sobolev space  $H_0^{2}(\Omega)=W_0^{2,2}(\Omega)$ is an appropriate function space for the clamped problem on $\Omega$, see Hebey \cite[Proposition 3.3]{Hebey}. Moreover, the fundamental tone  $\Lambda_g(\Omega)$ defined by \eqref{variational-charact} is achieved by 
%
%	$$	\Lambda_g(\Omega)=\inf_{u\in W_0^{2,2}(\Omega)\setminus \{0\}}\frac{\displaystyle \int_{\Omega}(\Delta_g u)^2 {\rm d}v_g}{\displaystyle \int_{\Omega}u^2 {\rm d}v_g}$$
 a minimizer $u\in W_0^{2,2}(\Omega)$; in fact,  $u\in \mathcal C^\infty(\Omega)$. 
	
The spirit of the proof of Theorem \ref{Huisken-CP}  is similar to the one presented in \S \ref{subsection-ABNT}, by performing an 	Ashbaugh--Benguria--Nadirashvili--Talenti  nodal-decom\-po\-sition on $(M,g)$ combined  with an appropriate comparison argument that involves the asymptotic volume ratio ${\sf AVR}_{(M,g)}> 0$.  We outline the proof, by emphasizing the differences with respect to the arguments used in \S \ref{subsection-ABNT}.  

Let $u_+\coloneqq\max(u,0)$ and $u_-\coloneqq-\min(u,0)$ be the positive and negative parts of $u$, respectively, and consider their preimages  $\Omega_+\coloneqq\{x\in \Omega: u_+(x)>0\}$ and $\Omega_-\coloneqq\{x\in \Omega: u_-(x)>0\}$ as well. 
Assume that 
$V_g(\Omega_+)V_g(\Omega_-)>0;$ otherwise the proof reduces to the sign-definite case. 
Let $u_\pm^\star:\mathbb R^n\to [0,\infty)$ be the Euclidean radial rearrangements of $u_\pm:\Omega_\pm\to [0,\infty),$  i.e., 
% $u\in W_0^{2,2}(\Omega)$ the minimizer in (\ref{variational-charact}) and its positive and negative parts $u_+$ and $u_-$, respectively, 
%Let $u_+^\star,u_-^\star:N_\kappa^n\to [0,\infty)$ such that 
for every $t>0,$ 
\begin{equation}\label{szimmetrizacio-U0}
	V_0(\{x\in \mathbb R^n:u_+^\star(x)>t \})=V_g(\{x\in\Omega:u_+(x)>t \})=:j_0(t),
\end{equation}
\begin{equation}\label{szimmetrizacio-U-00}
	V_0(\{x\in \mathbb R^n:u_-^\star(x)>t \})=V_g(\{x\in\Omega:u_-(x)>t \})=:h_0(t).
\end{equation}
The functions $u_\pm^\star$ are well-defined and radially symmetric having the property that  
\begin{equation}\label{alfa-r-t}
	\{x\in \mathbb R^n:u_+^\star(x)>t \}=B_0(r_t)\ \ {\rm and}\ \ \{x\in \mathbb R^n:u_-^\star(x)>t \}=B_0(\rho_t),\end{equation}
for some $r_t>0$ and $\rho_t>0$  with $V_0(B_0(r_t))=j_0(t)$ and $V_0(B_0(\rho_t))=h_0(t)$, respectively.  For further use, let $a,b\geq 0$ be such that 
$
V_0(B_0(a))=V_g(\Omega_+) $ and $  V_0(B_0(b))=V_g(\Omega_-).$
In particular, $a^n+b^n=L^n$ where $L>0$ is given by the condition $V_0(B_0(L))=V_g(\Omega)$.

%In same the way as in \eqref{szimmetrizacio-U} and \eqref{szimmetrizacio-U-0}, we  introduce the relative rearrangements $(\Delta_gu)_\pm^\star$ of $(\Delta_gu)_\pm:\Omega_\pm^\Delta\to (0,\infty)$.  We extend $u_\pm^\star:\Omega_\pm^\star\to (0,\infty)$ and $(\Delta_gu)_\pm^\star:(\Omega_\pm^\Delta)^\star\to (0,\infty)$ by zero  to the whole $\Omega^\star$ outside of  $\Omega_\pm^\star$ and $(\Omega_\pm^\Delta)^\star$, respectively. 	
We introduce the functions  $$\mathcal J_0(s)\coloneqq(\Delta_g u)^*_-(s)-(\Delta_g u)^*_+(V_0(B_0(L))-s)\ \ {\rm and}\ \ \mathcal H_0(s)\coloneqq-\mathcal J_0(V_0(B_0(L))-s),\ \ s\in [0,V_0(B_0(L))],$$	
where
\begin{equation}\label{J-J-notation-0}
	(\Delta_g u)^*_\pm(s)\coloneqq (\Delta_g u)^\star_\pm(x)\ \ {\rm with}\ \  s=\omega_n|x|^n,\ x\in B_0(L);
\end{equation}	
here $(\Delta_g u)^\star_\pm$ are the Euclidean radial rearrangements of $(\Delta_g u)_\pm$. Similarly as in Lemma \ref {basic=J-H}, we can prove that  
$$\ds\int_0^\sigma \mathcal J_0(s){\rm d}s\geq \ds\int_0^{V_0(B_0(L))} \mathcal J_0(s){\rm d}s=0,\ \ \int_0^\sigma \mathcal H_0(s){\rm d}s\geq \ds\int_0^{V_0(B_0(L))} \mathcal H_0(s){\rm d}s=0,\ \forall \sigma\in [0,V_0(B_0(L))],$$
and
\begin{equation}\label{basic=J-H-0}
	\ds\int_{B_0(a)} \mathcal J_0(\omega_n|x|^n){\rm d}x= \int_{B_0(b)} \mathcal H_0(\omega_n|x|^n){\rm d}x.
\end{equation}
Furthermore, analogously to Proposition \ref{f-hasonlitas}, we have that 
\begin{equation}\label{pro-4-1-a}
\ds\displaystyle\int_0^{j_0(t)} \mathcal J_0(s){\rm d}s\geq -\int_{\{u>t \}}\Delta_g u(x){\rm d}v_g(x),\ \forall  t\in [0,T_u^+],
\end{equation}
and
\begin{equation}\label{pro-4-1-b}
	\ds\displaystyle\int_0^{h_0(t)} \mathcal H_0(s){\rm d}s\geq -\int_{\{u<-t \}}\Delta_g u(x){\rm d}v_g(x),\ \forall  t\in [0,T_u^-],
\end{equation}
where $T_u^\pm\coloneqq \sup_{x\in \Omega_\pm}u_\pm(x)\geq 0,$ and either $(\Delta_g u)^*_-(s)=0$ or $ (\Delta_g u)^*_+(V_0(B_0(L))-s) = 0$ for every $s\in [0,V_0(B_0(L))]$.

The analogue of Theorem \ref{talenti-result} can be stated as follows. 

\begin{proposition} \label{talenti-result-0} The real functions  
	\begin{equation}\label{v-definit-1}
		w_a(x)\coloneqq\frac{1}{n\omega_n}\int_{|x|}^a \rho^{1-n}\left(\int_0^{\omega\rho^n}\mathcal J_0(s){\rm d}s\right){\rm d}\rho,\ x\in\Omega^\star,
		%
		%\left\{ \begin{array}{lll}
		%\ds\frac{1}{n\omega_n}\int_{d_\kappa(x)}^a \left(\frac{\sinh(\kappa\rho)}{\kappa}\right)^{1-n}\left(\int_0^{V_\kappa(\rho)}F(s){\rm d}s\right){\rm d}\rho &\mbox{if} &  \kappa>0; \\
		%% u\geq 0 &\mbox{in} &   \Omega;\\
		%\ds\frac{1}{n\omega_n}\int_{d_0(x)}^a \rho^{1-n}\left(\int_0^{V_0(\rho)}F(s){\rm d}s\right){\rm d}\rho  &\mbox{if} &  \kappa=0.
		%\end{array}\right.
	\end{equation}
	and
	\begin{equation}\label{w-definit-2}
		w_b(x)\coloneqq\frac{1}{n\omega_n}\int_{|x|}^b \rho^{1-n}\left(\int_0^{\omega\rho^n}\mathcal H_0(s){\rm d}s\right){\rm d}\rho,\ x\in \Omega^\star,
	\end{equation} 
satisfy the following statements$:$  
	\begin{itemize}
		\item[(i)] $\displaystyle\int_{\Omega} (\Delta_g u)^2{\rm d}v_g=\int_{B_0(a)} (\Delta w_a)^2{\rm d}x+ \int_{B_0(b)} (\Delta w_b)^2{\rm d}x;$
		\item[(ii)] $\displaystyle{\sf AVR}_{(M,g)}^\frac{4}{n}\int_{\Omega}u^2 {\rm d}v_g\leq \int_{B_0(a)}w_a^2 {\rm d}x+\int_{B_0(b)}w_b^2 {\rm d}x.$
	\end{itemize}
\end{proposition}

%Keeping the same notations as before, the asymptotic volume ratio ${\sf AVR}_{(M,g)}> 0$ plays a crucial role in the following  comparison involving the functions $w_a,w_b$ and $u$, respectively.  

\noindent {\it Proof.} 
One can easily verify that $w_a$ and $w_b$ are solutions to the Dirichlet problems 
\begin{equation}\label{CP-problem-v-w}
	\left\{ \begin{array}{@{~}r@{~}l@{~}l@{~}l@{~}l@{~}}
		-\Delta w_a(x)&=&\mathcal J_0(\omega_n|x|^n) &\mbox{in} &  B_0(a); \\
		% u\geq 0 &\mbox{in} &   \Omega;\\
		w_a&=&0  &\mbox{on} &  \partial B_0(a);
	\end{array}\right.
	\ {\rm and}\ 
	\left\{ \begin{array}{@{~}r@{~}l@{~}l@{~}l@{~}l@{~}}
		-\Delta w_b(x)&=&\mathcal H_0(\omega_n|x|^n) &\mbox{in} &  B_0(b); \\
		% u\geq 0 &\mbox{in} &   \Omega;\\
		w_b&=&0  &\mbox{on} &  \partial B_0(b),
	\end{array}\right.
\end{equation}
respectively. Since (i) is similar to the proof of Theorem \ref{talenti-result}/(i), we focus on property (ii). To complete this part, we recall the following sharp isoperimetric inequality on $(M,g)$ (see   Brendle \cite{Brendle},  and Balogh and Krist\'aly \cite{Balogh-Kristaly}): 
for every bounded open subset $\Omega\subset M$ with smooth
boundary $\partial \Omega$, one has that
\begin{equation}\label{eq-isoperimetric-1}
	\mathcal P_g(\partial \Omega)\geq  n\omega_n^\frac{1}{n} \ {\sf AVR}_{(M,g)}^\frac{1}{n}V_g(\Omega)^\frac{n-1}{n},	
\end{equation}
where $\mathcal P_g(\partial \Omega)$ is the perimeter of $\partial \Omega$; moreover, equality holds in \eqref{eq-isoperimetric-1} if and only if ${\sf AVR}_{(M,g)}=1$, i.e.,  $(M,g)$ is isometric to $(\mathbb R^n,g_0)$  and $\Omega\subset M$ is isometric to a ball of volume $V_g(\Omega).$\footnote{Inequality \eqref{eq-isoperimetric-1}  has been proven first by Agostiniani, Fogagnolo and Mazzieri \cite{AFM} by using Huisken's mean curvature
	flows; note however that their argument works only in dimension $3$. Later on, Fogagnolo and Mazzieri \cite{Fogagnolo-Mazzieri} extended their arguments up to dimension 7.}

 We confine our argument only to $w_a$; the case of $w_b$ works similarly.  We  consider the sets 
 $$
 \Lambda_t^\star\coloneqq\partial(\{x\in \mathbb R^n: u_+^\star(x)>t \}),\ \ \Lambda_t\coloneqq\partial(\{x\in \Omega: u_+(x)>t \}),\ \ t\in [0,T_u^+].
 $$
Due to \eqref{szimmetrizacio-U0} and \eqref{eq-isoperimetric-1}, and since the balls in $(\mathbb R^n,g_0)$ are the isoperimetric sets,  we obtain that
 \begin{equation}\label{isop-Brendle-kovetk}
  \mathcal P_g(\Lambda_t)\geq n\omega_n^\frac{1}{n} \ {\sf AVR}_{(M,g)}^\frac{1}{n}j_0(t)^\frac{n-1}{n}= \ {\sf AVR}_{(M,g)}^\frac{1}{n}\mathcal P_0(\Lambda_t^\star),\ \ t\in [0,T_u^+],	
 \end{equation}
where $\mathcal P_0(\Lambda_t^\star)$ denotes the Euclidean perimeter of  the sphere $\Lambda_t^\star\subset \mathbb R^n$.

A similar argument as in the proof of \eqref{peri-1} implies that  
\begin{equation}\label{13-as}
	\ds{\mathcal P}^2_g(\Lambda_t)\leq j_0'(t)\int_{\{u>t\}}\Delta_gu{\rm d}v_g\ \ {\rm for \ a.e.}\ t\in [0,T_u^+].
\end{equation}
Inequality \eqref{13-as} together with \eqref{pro-4-1-a} yields that
$$\ {\sf AVR}_{(M,g)}^\frac{2}{n}\mathcal P_0(\Lambda_t^\star)^2\leq -j_0'(t)\displaystyle\int_0^{j_0(t)} \mathcal J_0(s){\rm d}s\ \ {\rm for \ a.e.}\ t\in [0,T_u^+].$$ 
Since $j_0(t)=V_0(B_0(r_t))=\omega_nr_t^n$, see \eqref{alfa-r-t}, it turns out that $j_0'(t)=\mathcal P_0(\Lambda_t^\star)r_t'=n\omega_nr_t^{n-1}r_t'$ for a.e.\ $t\in [0,T_u^+]$. Therefore, we obtain the inequality 
\begin{equation}\label{utolsok-kozott}
 {\sf AVR}_{(M,g)}^\frac{2}{n}n\omega_n\leq - r_t^{1-n}r_t'\displaystyle\int_0^{\omega_nr_t^n} \mathcal J_0(s){\rm d}s\ \ {\rm for \ a.e.}\ t\in [0,T_u^+].	
\end{equation}
If $x\in B_0(a)$ is arbitrarily fixed, one can find a unique value $\eta \in [0,T_u^+]$ such that $|x|=a_\eta$; moreover, by construction, $u_+^\star(x)=\eta$. Accordingly, by integrating the  inequality \eqref{utolsok-kozott} on $[0,\eta]$ and performing  a change of variables, one can conclude  that 
$$\ {\sf AVR}_{(M,g)}^\frac{2}{n}u_+^\star(x) \leq \frac{1}{n\omega_n}\int_{|x|}^a \rho^{1-n}\left(\int_0^{\omega\rho^n}\mathcal J_0(s){\rm d}s\right){\rm d}\rho= w_a(x).$$
In a similar way, by using \eqref{alfa-r-t} and  \eqref{pro-4-1-b}, we obtain that 
$$\ {\sf AVR}_{(M,g)}^\frac{2}{n}u_-^\star(x) \leq \frac{1}{n\omega_n}\int_{|x|}^b \rho^{1-n}\left(\int_0^{\omega\rho^n}\mathcal H_0(s){\rm d}s\right){\rm d}\rho\equiv w_b(x).$$
Thus, we infer that 
	\begin{align*}
	\int_{\Omega}u^2 {\rm d}v_g&= \int_{\Omega_+}u_+^2 {\rm d}v_g+\int_{\Omega_-}u_-^2 {\rm d}v_g=\int_{B_0(a)}(u_+^\star)^2 {\rm d}x+\int_{B_0(b)}(u_-^\star)^2 {\rm d}x\\&\leq{\sf AVR}_{(M,g)}^{-\frac{4}{n}}\left(
	\int_{B_0(a)}w_a^2 {\rm d}x+\int_{B_0(b)}w_b^2 {\rm d}x\right),
\end{align*}
which concludes the proof of (ii). 
 \hfill $\square$

\noindent {\it Proof of Theorem \ref{Huisken-CP}.} In view of Proposition \ref{talenti-result-0}, one has that 
\begin{eqnarray}\label{Avg-ut}
\nonumber	\Lambda_g(\Omega)&=&\min_{u\in W_0^{2,2}(\Omega)\setminus \{0\}}\frac{\displaystyle \int_{\Omega}(\Delta_g u)^2 {\rm d}v_g}{\displaystyle \int_{\Omega}u^2 {\rm d}v_g}\\ &\geq& {\sf AVR}_{(M,g)}^{\frac{4}{n}}\min_{(u_a,u_b)\neq (0,0)}\frac{\displaystyle \int_{B_0(a)} (\Delta u_a)^2{\rm d}x+ \int_{B_0(b)} (\Delta u_b)^2{\rm d}x}{\displaystyle\int_{B_0(a)}u_a^2 {\rm d}x+\int_{B_0(b)}u_b^2 {\rm d}x},
\end{eqnarray}
where $a^n+b^n=L^n$ (with $V_g(\Omega)=\omega_nL^n$)  and $(u_a,u_b)$ is taken over of all pairs of radially symmetric functions with   $u_a\in W^{1,2}_0(B_0(a))\cap W^{2,2}(B_0(a))$ and $u_b\in W^{1,2}_0(B_0(b))\cap W^{2,2}(B_0(b))$, verifying the boundary condition  
$$u_a'(a)a^{n-1}=u_b'(b)b^{n-1};$$ indeed, the latter condition is implied by \eqref{basic=J-H-0} and the Dirichlet problems \eqref{CP-problem-v-w}. 

As we can observe, the optimization problem in \eqref{Avg-ut} is precisely the one that appears in Ashbaugh and Benguria \cite[p.\ 6]{A-B}. Therefore, the following cases are distinguished.  
\begin{itemize}
	\item The case $n\in \{2,3\}$:  by \cite[Theorem 1]{A-B}, we obtain that
	\begin{equation}\label{A_B_isoper-1}
		\Lambda_g(\Omega)\geq {\sf AVR}_{(M,g)}^{\frac{4}{n}}\Lambda_0(B_0(L)),
	\end{equation}
	which occurs when  $a=L$ and $b=0$ (or vice-versa) in \eqref{Avg-ut}. 
		\item The case $n\geq 4$: by Ashbaugh and Laugesen \cite[Theorem 4]{A-L}, it follows that
		$$\Lambda_g(\Omega)\geq {\sf AVR}_{(M,g)}^{\frac{4}{n}}w_n\Lambda_0(B_0(L)),$$
		where 
		$$w_n=2^\frac{4}{n} \frac{ {\mathfrak j}_{\frac{n}{2}-1}^4}{\mathfrak  h_{\frac{n}{2}-1}^4}<1,$$ which appears when $a=b=2^{-\frac{1}{n}}L$ in \eqref{Avg-ut}.  By \cite{A-L}, we also have that  $
		\displaystyle\lim_{n\to \infty} w_{n}=1.$
\end{itemize}

Let $\Omega\subset M$ be an open bounded set such that equality holds in \eqref{A_B_isoper-1}. In particular, the Ashbaugh--Benguria--Nadirashvili--Talenti  nodal-decom\-po\-sition argument implies that  the minimizer $u$ has a constant sign (say, $u>0 $ in $\Omega$, since, e.g.\ $b=0$ and $a=L$); moreover, one necessarily has equality also in \eqref{isop-Brendle-kovetk} for a.e.\ $t\in [0,T_u^+]$. Therefore,  $(M,g)$ is isometric to $(\mathbb R^n,g_0)$  and $\Lambda_t\subset \Omega$ is isometric to the ball $\{x\in\mathbb R^n:u_+^\star(x)>t \}$ for a.e.\ $t\in [0,T_u^+]$, cf.\ \eqref{eq-isoperimetric-1}. In particular, $\Omega$ is also isometric to $B_0(L)$, which concludes the proof. The converse statement holds trivially.
  \hfill $\square$

% $$
%\Pi_t^\star=\partial(\{x\in \mathbb R^n: u_-^\star(x)>t \}),\ \ \Pi_t=\partial(\{x\in \Omega: -u(x)>t \}), 
%$$

% Let $\Omega_\pm^\star\subset \mathbb R^n$  be the balls (with center in the origin) with the property that $V_0(\Omega_\pm^\star)=V_g(\Omega_\pm)$ and

%	\vspace{2cm}
	
%	\renewcommand{\theequation}{A.\arabic{equation}}
\appendix	\section{Special functions} \label{section-8}
This section briefly lists some basic properties of those special functions that were used throughout the paper; these properties can also be found in 
Olver, Lozier,  Boisvert and Clark \cite{Digital}. 

Let  $\mu>-1$ be fixed. The Bessel and modified Bessel  functions of the first kind are defined as 
\begin{equation}\label{Bessel-J}
J_\mu(x)=\sum_{m=0}^\infty\frac{(-1)^m}{m!\Gamma(m+\mu+1)}\left(\frac{x}{2}\right)^{2m+\mu},\ x\in \mathbb R,
\end{equation}
and
\begin{equation}\label{Bessel-I}
	I_\mu(x)=\mathfrak{i}^{-\mu}J_\mu(\mathfrak{i}x)=\sum_{m=0}^\infty\frac{1}{m!\Gamma(m+\mu+1)}\left(\frac{x}{2}\right)^{2m+\mu},\ x\in \mathbb R,
\end{equation}
respectively, see \cite[\S 10]{Digital}.
The following recurrence relations hold for the Bessel functions and their derivatives; namely, 
\begin{equation}\label{Bessel-recurrence}
	J'_{\mu}(x)=-J_{\mu+1}(x)+\frac{\mu}{x}J_{\mu}(x) \ \ {\rm and} \ 	I'_{\mu}(x)=I_{\mu+1}(x)+\frac{\mu}{x}I_{\mu}(x),\ x>0.
\end{equation} 
 For  $\mu\notin \mathbb Z$, the Bessel and modified Bessel  functions of the second kind are defined as  
$$Y_\mu(x)=\frac{J_\mu(x)\cos(\mu \pi)-J_{-\mu}(x)}{\sin(\mu \pi)}\ \  {\rm and}\ \ K_\mu(x)=\frac{\pi}{2}\frac{I_{-\mu}(x)-I_\mu(x)}{\sin(\mu \pi)},$$
respectively, see  \cite[rels.\ (10.2.3) and (10.27.4)]{Digital}, 
while in the case of any integer order $n$, we have that $$Y_n(x)=\lim_{\mu\to n}Y_\mu(x)\ \  {\rm and}\ \ K_n(x)=\lim_{\mu\to n}K_\mu(x).$$
We also have that 
\begin{equation}\label{Y-I-K}
Y_\mu(\mathfrak iz)=e^\frac{(\mu+1)\mathfrak{i}\pi}{2}I_\mu(z)-\frac{2}{\pi}e^{-\frac{\mu \mathfrak{i}\pi}{2}}K_\mu(z),\ \ z\in \mathbb C.
\end{equation}
An alternative, more explicit representation for $Y_n$ $(n\in \mathbb N)$ is 
\begin{equation}\label{Y_n-forma}
\begin{array}{@{~}r@{~}c@{~}l}
	Y_n(z) & = & -\dfrac{(\frac{z}{2})^{-n}}{\pi}\displaystyle\sum_{m=0}^{n-1}\dfrac{(n-m-1)!}{m!}\left(\dfrac{z}{2}\right)^{2m}+\dfrac{2}{\pi}J_n(z)\ln\dfrac{z}{2}\\
	&  & -\dfrac{(\frac{z}{2})^{n}}{\pi}\displaystyle\sum_{m=0}^{\infty}\dfrac{(-1)^m}{m!(n+m)!}(\Psi(m+1)+\Psi(n+m+1))\left(\dfrac{z}{2}\right)^{2m}, \ z\in \mathbb C,
\end{array}
\end{equation}
%
%%\displaystyle\sum
%%\frac{}{} --> \dfrac{}{}
%
%
%
%
%\begin{align}\label{Y_n-forma}
%\nonumber Y_n(z)=&-\frac{(\frac{z}{2})^{-n}}{\pi}\sum_{m=0}^{n-1}\frac{(n-m-1)!}{m!}\left(\frac{z}{2}\right)^{2m}+\frac{2}{\pi}J_n(z)\ln\frac{z}{2}\\&-\frac{(\frac{z}{2})^{n}}{\pi}\sum_{m=0}^{\infty}\frac{(-1)^m}{m!(n+m)!}(\Psi(m+1)+\Psi(n+m+1))\left(\frac{z}{2}\right)^{2m}, \ z\in \mathbb C,
%\end{align}
where $\Psi=(\ln {\Gamma})'$
is the Digamma function.  The function $\Psi$ has the pointwise properties 
\begin{equation}\label{digamma-prop-11}
	\Psi(z+1)-\Psi(z)=\frac{1}{z},\ z\neq  0,-1,-2\ldots,
\end{equation}
\begin{equation}\label{digamma-prop-12}
	 \Psi(1-z)-\Psi(z)=\pi \cot(\pi z),\  z\neq  0,\pm1,\pm 2,\ldots,
\end{equation}
\begin{equation}\label{digamma-prop-2}
	\Im  \Psi\left(\frac{1}{2}+\mathfrak{i}y\right)=\frac{\pi}{2}\tanh(y\pi),\   y\in \mathbb R,
\end{equation}
see \cite[rels.\ (5.5.2), (5.5.4)  and (5.4.17)]{Digital}, where $\Im z$ denotes the imaginary part of $z\in \mathbb C,$ and the asymptotic property 
\begin{equation}\label{Digamma-converges}
	\Psi(z)\sim\ln z-\frac{1}{2z}-\sum_{m=1}^\infty \frac{B_{2m}}{2mz^{2m}}\ {\rm as}\ z\to \infty,| {\rm ph}z|<\pi,
\end{equation}
where $\{B_{2m}\}_{m\in \mathbb N_{\geq 1}}$ are the Bernoulli numbers, see \cite[rel.\ (5.11.2)]{Digital}. 

If $t\in (-1,1)$, $\mu\in (-\infty,0]$ and $\nu\in \mathbb C$ with 
$\nu(1+\nu)\in \mathbb R,$
%$\nu+\overline \nu=-1$, 
the Legendre (called also Ferrers) and Gaussian hypergeometric  functions of the first kind are connected by the relation 
\begin{equation}\label{P-hypergoemetric}
\LP{}_\nu^\mu(t)=\frac{1}{\Gamma(1-\mu)}\left(\frac{1+t}{1-t}\right)^\frac{\mu}{2}{_2F}_1\left(1+\nu,-\nu;1-\mu;\frac{1-t}{2}\right)\in \mathbb R,
\end{equation}
see \cite[rel.\ (14.3.1)]{Digital}. 
The inversion formula %(14.9.3) by Olver
for $\mu=m\in \mathbb N$ gives that 
\begin{equation}\label{inversion}
	\LP_\nu^{-m}(t)=(-1)^m\frac{\Gamma(\nu-m+1)}{\Gamma(\nu+m+1)}\LP_\nu^{m}(t),\ t\in (-1,1),
\end{equation}
see \cite[rel.\ (14.3.5)]{Digital}
Using \cite[rel.\ (15.10.11)]{Digital}, we recall the  Euler--Pfaff transformations  
\begin{equation}\label{Olver-1}
{_2F}_1(a,b;c;z)=(1-z)^{-a}{_2F}_1\left(a,c-b;c;\frac{z}{z-1}\right)=(1-z)^{c-a-b}{_2F}_1\left(c-a,c-b;c;z\right),\ \  |z|<1.
\end{equation}
The differentiation formula gives that
\begin{equation}\label{F-differential}
	\frac{\rm d}{{\rm d}z}\, {_2F}_1(a,b;c;z)=\frac{ab}{c}\,{_2F}_1(a+1,b+1;c+1;z),
\end{equation}
or 
\begin{equation}\label{F-differential-2}
	c(1-z)\frac{\rm d}{{\rm d}z}\,{_2F}_1(a,b;c;z)=(c-a)(c-b)\,{_2F}_1(a,b;c+1;z)+ c(a+b-c)\,{_2F}_1(a,b;c;z).
\end{equation}
Moreover, the derivation formula for the Legendre function reads as
\begin{equation}\label{derivation-legendre}
	(1-t^2)\frac{{\rm d}}{{\rm d}t}\LP_\nu^\mu(t)=(\mu+\nu)\LP_{\nu-1}^\mu(t)-\nu t\LP_\nu^\mu(t),\ \ t\in (-1,1),
\end{equation}
while the derivatives with respect to the degree of the Legendre function is 
\begin{equation}\label{derivative-Legendre}
\frac{\partial}{\partial \nu}\LP_\nu^\mu(t)=\pi \cot(\nu\pi)\LP_\nu^\mu(t)-\frac{1}{\pi}\LA_\nu^\mu(t),\ \ t\in (-1,1),
\end{equation}
see \cite[rels.\ (14.10.5) and (14.11.1)]{Digital}, where 
\begin{equation}\label{A-expression}
\LA_\nu^\mu(t)=\sin(\nu\pi)\left(\frac{1+t}{1-t}\right)^\frac{\mu}{2}\sum_{m=0}^\infty \frac{\Gamma(m-\nu)\Gamma(m+\nu+1)\left(\Psi(m+\nu+1)-\Psi(m-\nu)\right)}{m!\Gamma(m-\mu+1)}\left(\frac{1-t}{2}\right)^m.
\end{equation}
%14.11.3
The symmetrization formula has the form
%(14.9.10)
\begin{equation}\label{inversion-P}
	\LP{}_\nu^0(-t)=\cos(\nu \pi)\LP{}_\nu^0(t)-\frac{2}{\pi}\sin(\nu \pi)\LQ_\nu^0(t),\ \ t\in (-1,1),
\end{equation}
see \cite[rel.\ (14.9.10)]{Digital},  where $\LQ_\nu^0$ is the associated Legendre function; moreover, 
\begin{equation}\label{singular-Q}
\LP_\nu^0(1)=1 \ \ {\rm and} \ \	\LQ_\nu^0(t)=-\frac{1}{2}\ln{(1-t)}+\frac{\ln 2}{2}-\gamma-\Psi(\nu+1)+\mathcal O((1-t)\ln(1-t))\ \ {\rm as}\ \ t\nearrow 1,
\end{equation}
see \cite[rel.\ (14.8.3)]{Digital}, where $\nu\neq -1,-2,\ldots$ and $\gamma\approx 0.5772$ is the Euler constant. We also have  that 
\begin{equation}\label{Gamma-special}
\cosh(z\pi)\Gamma\left(\frac{1}{2}+\mathfrak{i}z\right)\Gamma\left(\frac{1}{2}-\mathfrak{i}z\right)=\pi,\ \ z\in \mathbb C,
\end{equation}
see \cite[rel.\ (5.4.4)]{Digital}.  

The behavior of the Gaussian hypergeometric functions at the singularity $1$ is described as 
\begin{equation}\label{singularity-1}
\lim_{z\nearrow 1} 	{_2F}_1(a,b;c;z)=\frac{\Gamma(c-a-b)\Gamma(c)}{\Gamma(c-a)\Gamma(c-b)},\ \  {\rm whenever}\ \ \Re(c-a-b)>0;
\end{equation} 
 \begin{equation}\label{singularity-2}
 	\lim_{z\nearrow 1} 	\frac{{_2F}_1(a,b;c;z)}{-\ln(1-z)}=\frac{\Gamma(c)}{\Gamma(a)\Gamma(b)},\ \  {\rm whenever}\ \  c=a+b;
 \end{equation} 
 \begin{equation}\label{singularity-3}
	\lim_{z\nearrow 1} 	\frac{{_2F}_1(a,b;c;z)}{(1-z)^{c-a-b}}=\frac{\Gamma(c)\Gamma (a+b-c)}{\Gamma(a)\Gamma(b)},\ \  {\rm whenever}\ \  \Re(c-a-b)<0,
\end{equation} 
 see \cite[\S 15.4]{Digital}.	Moreover, by \cite[rel.\ (15.4.30)]{Digital}  we have that
 \begin{equation}\label{fel-terfogathoz}
 	{_2F}_1\left(a,1-a;c;\frac{1}{2}\right)=\frac{2^{1-c}\sqrt{\pi}\Gamma(c)}{\Gamma(\frac{a+c}{2})\Gamma(\frac{c-a+1}{2})}.
 \end{equation}

\vspace{0.5cm}
\noindent {\bf Acknowledgment.} The author thanks the anonymous Referee
for all valuable comments and suggestions that significantly improved the presentation of the paper.\
Thanks are also due to the Referee of the paper \cite{Kristaly-Adv-math} who suggested to investigate the missing part (i.e., the positively curved case) of Lord Rayleigh's conjecture for clamped plates. 
The author is also grateful to \'Arp\'ad Baricz and Dmitrii Karp  for their help in the theory of special functions and to \'Agoston R\'oth for his advice, figures and help concerning the present form of the manuscript.

\end{document}